\documentclass[a4]{article}

\usepackage[utf8]{inputenc}
\usepackage[T1]{fontenc}
\usepackage{geometry}
\usepackage[english]{babel}

\usepackage{amsmath,amssymb,amsthm}
\usepackage{mathrsfs}
\usepackage{dsfont}
\usepackage{color}
\usepackage[colorlinks]{hyperref}
\usepackage{graphicx}
\usepackage{ulem}
\usepackage[dvipsnames]{xcolor}
\usepackage{cancel}
\usepackage{enumitem}
\usepackage{dsfont}
\usepackage{stmaryrd}
\usepackage{nicefrac}
\usepackage{comment}

\usepackage{caption}
\usepackage{subcaption}
\usepackage{cleveref}

\makeatletter
\def\paragraph{\@startsection{paragraph}{4}%
  \z@\z@{-\fontdimen2\font}%
  {\normalfont\bfseries}}
\makeatother

\newtheorem{thm}{Theorem}[section]
\newtheorem{lem}[thm]{Lemma}
\newtheorem{cor}[thm]{Corollary}
\newtheorem{prop}[thm]{Proposition}
\theoremstyle{definition}
\newtheorem{dfn}[thm]{Definition}
\newtheorem{asm}{Assumption}
\newtheorem{cond}{Condition}
\newtheorem{rem}[thm]{Remark}

\newtheorem{innermanualasm}{Assumption}
\newenvironment{manualasm}[1]
  {\renewcommand\theinnermanualasm{#1}\innermanualasm}
  {\endinnermanualasm}

\newtheorem{innermanualcond}{Condition}
\newenvironment{manualcond}[1]
  {\renewcommand\theinnermanualcond{#1}\innermanualcond}
  {\endinnermanualcond}

\newcommand{\E}{\mathbb{E}}

\newcommand{\Lin}{\mathbb{L}}
\newcommand{\N}{\mathbb{N}}
\renewcommand{\P}{\mathbb{P}}
\newcommand{\R}{\mathbb{R}}
\renewcommand{\S}{\mathbb{S}}
\newcommand{\V}{\mathbb{V}}
\newcommand{\W}{\mathbb{W}}
\newcommand{\Z}{\mathbb{Z}}
\newcommand{\ind}{\mathds{1}}

\newcommand{\bA}{\mathbf{A}}
\newcommand{\bB}{\mathbf{B}}
\newcommand{\bJ}{\mathbf{J}}
\newcommand{\bK}{\mathbf{K}}
\newcommand{\bS}{\mathbf{S}}
\newcommand{\bT}{\mathbf{T}}
\newcommand{\bU}{\mathbf{U}}
\newcommand{\bX}{\mathbf{X}}
\newcommand{\bY}{\mathbf{Y}}
\newcommand{\bOm}{\mathbf{\Omega}}

\newcommand{\vx}{\mathbf{x}}
\newcommand{\vy}{\mathbf{y}}
\newcommand{\vf}{\mathbf{f}}
\newcommand{\vg}{\mathbf{g}}
\newcommand{\vh}{\mathbf{h}}
\newcommand{\vm}{\mathbf{m}}
\newcommand{\vn}{\mathbf{n}}
\newcommand{\vp}{\mathbf{p}}

\newcommand{\vr}{\mathbf{r}}
\newcommand{\vu}{\mathbf{u}}
\newcommand{\vv}{\mathbf{v}}
\newcommand{\vw}{\mathbf{w}}
\newcommand{\vz}{\mathbf{z}}
\newcommand{\vO}{\mathbf{0}}

\newcommand{\balpha}{\boldsymbol\alpha}
\newcommand{\bbeta}{\boldsymbol\beta}
\newcommand{\bmu}{\boldsymbol\mu}
\newcommand{\bomega}{\boldsymbol\omega}
\newcommand{\bsigma}{\boldsymbol\sigma}
\newcommand{\bvarphi}{\boldsymbol\varphi}

\newcommand{\mM}{\mathrm{M}}
\newcommand{\mF}{\mathrm{F}}

\newcommand{\cA}{\mathcal{A}}
\newcommand{\cB}{\mathcal{B}}
\newcommand{\cC}{\mathcal{C}}

\newcommand{\cO}{\mathcal{O}}
\newcommand{\cL}{\mathcal{L}}
\newcommand{\cM}{\mathcal{M}}
\newcommand{\cP}{\mathcal{P}}
\newcommand{\cQ}{\mathcal{Q}}
\newcommand{\cX}{\mathcal{X}}
\newcommand{\cY}{\mathcal{Y}}

\newcommand{\st}{\text{s.t.}}
\newcommand{\op}{\kappa'}
\newcommand{\od}{\mathrm{d}}
\newcommand{\x}{\times}

\newcommand{\grad}{\mathbf{grad}}
\newcommand{\Hess}{\mathrm{Hess}}
\newcommand{\e}{\mathrm{e}}
\newcommand{\lint}{\llbracket}
\newcommand{\rint}{\rrbracket}
\newcommand{\gmd}{\kappa^\star}
\newcommand{\gmp}{\kappa'}
\renewcommand{\epsilon}{\varepsilon}
\renewcommand{\div}{\mathrm{div}}
\newcommand{\norm}[1]{\left\| #1\right\|}

\DeclareMathOperator{\argmin}{\mathrm{argmin}}

\renewcommand{\emph}[1]{{\it #1}}

\makeatletter
\newcommand{\mylabel}[2]{#2\def\@currentlabel{#2}\label{#1}}
\makeatother
\title{Convergence rates for the moment-SoS hierarchy}
\author{Corbinian Schlosser$^{1}$, Matteo Tacchi-Bénard$^{2}$, Alexey Lazarev$^{3}$}
\date{}

\begin{document}

\maketitle

\footnotetext[1]{INRIA - Ecole Normale Supérieure - PSL Research University, Paris, France. \textbf{Corresponding author.}}
\footnotetext[2]{Univ. Grenoble Alpes, CNRS, Grenoble INP (Institute of Engineering Univ. Grenoble Alpes), GIPSA-lab, France.}
\footnotetext[3]{Toulouse Institute of Mathematics, France.}

\begin{abstract}
We introduce a comprehensive framework for analyzing convergence rates for infinite dimensional linear programming problems (LPs) within the context of the moment-sum-of-squares hierarchy. Our primary focus is on extending the existing convergence rate analysis, initially developed for polynomial optimization, to the more general and challenging domain of the generalized moment problem. We establish an easy-to-follow procedure for obtaining convergence rates. Our methodology is based on, firstly, a state-of-the-art degree bound for Putinar's Positivstellensatz, secondly,  quantitative polynomial approximation bounds, and, thirdly, a geometric Slater condition on the infinite-dimensional LP. We address a broad problem formulation that encompasses various applications, such as optimal control, volume computation and exit location of stochastic processes. We illustrate the procedure {on} these three problems and, using a recent improvement on effective versions of Putinar's Positivstellensatz, we improve existing convergence rates.

\begin{center}
{\bf Keywords}
\end{center}

Real algebraic geometry; convex optimization; convergence rates; optimal control; stochastic processes, numerical methods for multivariate integration.

\begin{center}
{\bf Acknowledgments}
\end{center}

The work of Corbinian Schlosser was supported by the Polynomial Optimization, Efficiency through Moments and Algebra (POEMA) network and by the European Research Council (grant REAL 947908).

The work of M. Tacchi was funded by the Swiss National Science Foundation under the ``NCCR Automation'' grant n$^\circ$51NF40{\_}180545.

The work of Alexey Lazarev has benefited from the AI Interdisciplinary Institute ANITI. ANITI is funded by the French "Investing for the Future - PIA3" program under the Grant agreement n$^\circ$ANR-19-PI3A-0004.

The authors are very grateful to Lorenzo Baldi for his comments and advice on this work. We also want to express our thanks to the anonymous reviewers whose thoughtful comments helped us to improve the text significantly.
\end{abstract}

\tableofcontents

\section{Introduction}

\paragraph{\bf Context.}

In recent years, various kinds of non-linear problems have been recast, approximated, or bounded using convex sum-of-squares (SoS) or moment-SoS problems. The list of methods and applications is broad, see for instance \cite{lasserre2009moments,parrilo2020sum,marshall2008positive} for an introduction to polynomial optimization and moment problems with many applications, \cite{parrilo2000structured, parrilo2003semidefinite, ahmadi2013stability} for applications to polynomial optimization and control and complexity, \cite{helton2007linear, scheiderer2018spectrahedra} for semidefinite representations of subsets of euclidean space, \cite{laurent2005semidefinite} for integer programming, \cite{rudi2024finding} for applications to machine learning, \cite{barak2014sum} for application to complexity theory, \cite{de2002approximation} for applications to approximating the stability number of a graph, \cite{polak2019new} for bounding the Shannon capacity, \cite{schmudgen2017moment} for addressing the moment problem in probability and operator theory, \cite{curmei2023shape} for shape constraint optimization, \cite{bertsimas2002relation} for option pricing, \cite{hohmann2020moment} for application to energy storage, \cite{laraki2012semidefinite} for applications to game theory, or \cite{fang2021quantum,klep2023state} for its use in quantum information theory. A common pillar among the listed examples is that non-negativity is certified through SoS representations or moment constraints. When this pillar is fruitfully paired with semidefinite representations of SoS \cite{lasserre2009moments,parrilo2000structured,scheiderer2024book}, efficient methods for solving SDPs \cite{vandenberghe1996semidefinite,vandenberghe2015chordal} (or SoS tailored methods \cite{nesterov2000squared,riener2013exploiting,korda2023convergence,papachristodoulou2013sostools,blekherman2021sparse}) and (effective) Positivstellensätze \cite{schmudgen2017moment,putinar1993positive,nie2007complexity,scheiderer2017psatz,blekherman2019bounds,laurent2023effective,baldi2023psatz,slot2021christoffel}, an efficient, strong and widely applicable framework is built.

Some of the above applications have the goal of verifying the non-negativity of a fixed polynomial. 
In others of the above applications (such as \cite{curmei2023shape,bertsimas2002relation,laraki2012semidefinite,rudi2024finding,parrilo2000structured}) the polynomial, for which we want to certify non-negativity, is a variable. 
This contribution investigates the most general case where polynomials are decision variables of an infinite dimensional linear programming problem (LP), whose Lagrangian dual is known as the generalized moment problem (GMP). The GMP is approximated with the moment-SoS hierarchy: this work proposes a systematic methodology to compute convergence rates for such approximations.

In recent years, many (nonlinear) problems have been formulated via GMPs. Closest to our setting are, among others, volume computation for semialgebraic sets \cite{lasserre2017stokes}, optimal transport \cite{mula2024moment}, as well as set separation a la Urysohn \cite{korda2022urysohn}. Other examples arise from dynamical systems and include optimal control \cite{korda2017convergence}, stability analysis \cite{jones2021converse,papachristodoulou2002construction}, localization of global attractors \cite{goluskin2020bounding,schlosser2021converging,jones2021converse}, safety analysis~\cite{miller2023bounding}, as well as partial differential equations~\cite{marx2020burgers} and variational inequalities~\cite{fantuzzi2022verification}. Due to the natural role of Borel measures in the GMP, it is also widely used to study stochastic systems, with applications to exit location \cite{henrion2023moment}, infinite time averaging \cite{fantuzzi2016bounds} and peak estimation~\cite{miller2023var}.

Under certain compactness assumptions, the problems mentioned above can be represented in the framework of GMPs by dual pairs of linear programming problems: the primal problem (measure LP) is defined on the infinite-dimensional space of Borel measures and features (possibly infinitely many) linear equality constraints; the dual problem (function LP) is defined on a (possibly infinite-dimensional) vector space of continuous functions (or polynomials) and features infinite dimensional (functional) linear inequality constraints. The moment-sum-of-squares (moment-SoS) hierarchy is a two-step procedure that provides a powerful tool for tackling such linear programming problems and has been applied widely. For the primal problem, in the first step, the primal decision variables, i.e. Borel measures, are represented by their moments, which are characterized via linear matrix inequality (LMI) constraints on the so-called moment matrices. In the second step, the moment matrices are truncated, i.e. only moments up to a finite degree $\ell \in \N$ are considered and paired with the constraints in the LP; resulting in a hierarchy of semidefinite programs. This operation is referred to as \textit{moment relaxation}. By duality, this procedure leads to a tightening in its dual, the function LP. There, the inequality constraints of the function LP are first strengthened to SoS constraints (by the aid of Putinar's Positivstellensatz \cite{putinar1993positive}). In the second step, the degree of the SoS polynomials is truncated, thus obtaining a hierarchy of so-called \textit{SoS strengthenings}.
The scheme of the moment-SoS hierarchy approach is summarized in Figure \ref{fig:hierarchy}.

\begin{figure}
    \centering
    \includegraphics[scale=0.45]{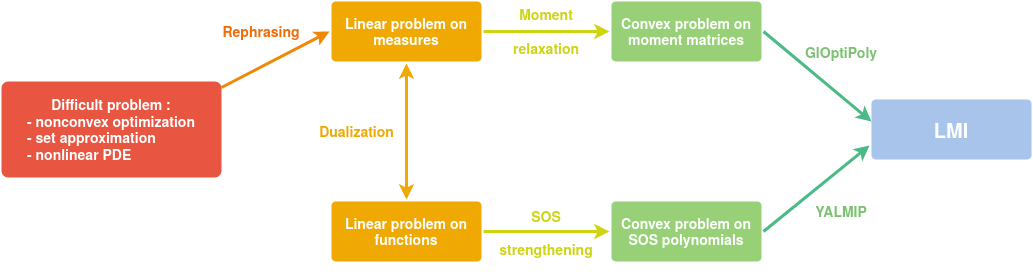}
    \caption{The standard application of the moment-SoS hierarchy.}
    \label{fig:hierarchy}
\end{figure}
\vspace{2mm}

The moment-SoS hierarchy provides guaranteed convergence \cite{lasserre2009moments,nie2014exact,tacchi2022convergence} but, apart from the case of polynomial optimization, the speed of the convergence for the infinite-dimensional problems has been rarely investigated. Two examples where explicit convergence rates were derived can be found in \cite{korda2017convergence} and \cite{korda2018convergence} where a slow convergence rate was presented based on~\cite{nie2007complexity}. Since then, the problem of providing bounds on the minimal truncation required to fully represent positive polynomials as sums of squares has been deeply studied, with recent strong improvements both in the generic case~\cite{baldi2021moment,baldi2022moment,slot2022convergence} and in specific settings~\cite{slot2021christoffel,baldi2023psatz,fang2021quantum,bach2023exponential}.

\paragraph{\bf Contribution.}

The speed of convergence can be derived from such bounds, and is determined by two main factors:
\begin{enumerate}
    \item[1.]\label{Speed1} The regularity of optimal solutions to the function LP and, if they are not polynomial, their approximation with polynomials;
    \item[2.] \label{Speed2} The degree $\ell \in \N$ needed for an SoS feasibility certificate for the approximating polynomials.
\end{enumerate}
We treat those two principal concepts in Section \ref{sec:method} and aim at describing an interplay between results on degree bounds in Putinar's Positivstellensatz, structural approximation properties for polynomials and compatibility conditions of the LPs concerning polynomials. 

The main objective of this article is to provide a method for deriving convergence rates of the moment-SoS hierarchy when applied to a specific instance of the GMP, using the degree bounds provided in~\cite{baldi2021moment}. We provide examples of computing and improving the convergence rates of the hierarchy with the state-of-the-art effective versions of Putinar's Positivstellensatz. Namely, we use \cite{baldi2021moment} to improve the convergence rates for the optimal control problem stated in \cite{korda2017convergence} and for the standard volume problem \cite{korda2018convergence}. Additionally, we derive an original convergence rate for the problem of exit location of stochastic processes \cite{henrion2023moment}. Last but not least, we use our methodology to answer a long-standing question related to volume computation, namely: how much does the use of Stokes' theorem improve the moment-SoS hierarchy for volume computation? Indeed, the first application of the moment-SoS hierarchy to this problem exhibited a very slow convergence in practice~\cite{korda2018convergence} and was soon complemented with~\cite{lasserre2017stokes}, yielding a sharp improvement in the numerical accuracy of the relaxations. To further understand this improvement, a qualitative study~\cite{tacchi2023stokes} showed that it was related to the Gibbs phenomenon and regularity of solutions in the function LP. In this work, we complement the qualitative study with a first \textit{quantitative} analysis of the two formulations, by computing and comparing {\textit{upper bounds}} on the convergence rates in both cases.

In most examples covered in this work, we get a convergence rate of $\cO(\ell^{-\nicefrac{1}{c}})$ for some constant $c >0$ (the only exception being generic optimal control, for which the rate is $\cO\left(\nicefrac{1}{\log \ell }\right)$, see {\bf Corollary \ref{prop:PropOCPConvRate}}, although mild assumptions allowed us to bring back a polynomial convergence rate in {\bf Theorem \ref{prop:ocpConvRatePoly}}), which is a significant improvement compared to the {\bf double} $\log$ bounds obtained in~\cite{korda2017convergence,korda2018convergence}.

\paragraph{\bf Outline.}

The paper is structured as follows: In Section \ref{sec:prelim}, we fix the notation and focus on the central underlying concept of moment-SoS hierarchy for the generalized moment problem (GMP). Section \ref{sec:rates} recalls the current state of an effective version of Putinar's Positivstellensatz and we introduce and motivate our general procedure for obtaining effective degree bounds for SoS tightening of the infinite-dimensional function LPs. In Sections \ref{sec:ApplicationDynSys} and \ref{sec:vol} we apply the procedure to establish convergence rates for old and new instances of the moment-SoS hierarchy and where strong improvements compared to the existing rates are demonstrated. Section \ref{sec:ApplicationDynSys} treats dynamical systems, where optimal control for deterministic systems and the exit location problem of stochastic processes are considered. Section \ref{sec:vol} is concerned with volume computation with and without the aid of reinforcing Stokes constraints. {Before concluding in Section~\ref{sec:Conclusion}, we discuss in Section \ref{sec:Limitations} the current limitations of the approach and possible improvements based on further development of effective Positivstellensätze.}

\section{Preliminaries: the moment-SoS hierarchy} \label{sec:prelim}

\subsection{Basic notations}

We work with the standard notations for usual sets $\R$ (real numbers), $\Z$ (integers), $\N$ (natural integers), for which the superscript $^\star$ indicates that we remove the element $0$. Real intervals are denoted $[a,b]$ when closed, $(a,b)$ when open; integer intervals are denoted $\lint a,b \rint$ (with particular case $\lint n \rint = \lint 1,n \rint$ for $n \in \N^\star$). For $x \in \R$, $\lfloor x \rfloor := \max([x-1,x]\cap\Z)$ denotes the floor and $\lceil x \rceil := \min([x,x+1]\cap\Z)$ denotes the ceiling.

For a topological space $\bX$, $\cC(\bX)$ denotes the space of bounded continuous functions from $\bX$ to $\R$ equipped with the topology of uniform convergence. For two real vector spaces $\V$, $\W$, the set $\Lin(\V,\W)$ denotes the space of linear maps from $\V$ to $\W$. For a real Banach space $\V$, define the dual space $\V' := \Lin(\V,\R) \cap \cC(\V)$, with duality $\langle v , v' \rangle \in \R$, $v \in \V, v' \in \V'$. In particular, for a compact Hausdorff space $\bX$, the Banach space of signed Radon measures $\cM(\bX)$ is identified with $\cC(\bX)'$.

For $\balpha = (\alpha_1,\ldots,\alpha_n) \in \N^n$, $|\balpha| := \alpha_1+\ldots+\alpha_n$ is the range of $\balpha$ and $ (x_1,\ldots,x_n) = \vx \mapsto \vx^{\balpha} := x_1^{\alpha_1}\cdots x_n^{\alpha_n}$ is the corresponding monomial. For $n,d \in \N$, $\N^n_d := \{\balpha \in \N^n \; ; |\balpha| \leq d\}$ is the set of bounded multi-indices, $\R_d[\vx]:=\{\vx \mapsto \sum_{|\balpha| \leq d} c_{\balpha} \, \vx^{\balpha} \; ; (c_{\balpha})_{\balpha} \in \R^{\N^n_d}\}$ is the space of polynomial functions of degree at most $d$, $\R[\vx] := \cup_{d\in\N} \R_d[\vx]$ is the space of polynomials. For $\bOm \subset \R^n$, $\cP_d(\bOm) := \cC(\bOm)\cap\R_d[\vx]$ and $\cP(\bOm) := \cC(\bOm) \cap \R[\vx]$.

We denote by $\R^n_+$ the cone of coordinate-wise non-negative vectors in $\R^n$. By $\cC(\bX)_+$ (respectively $\cC(\bX)_\oplus$) we denote the cone of pointwise (strictly) non-negative functions. The dual cone of $\cC(\bX)_+$ of non-negative Borel measures is denoted by $\cM(\bX)_+$. {The supremum norm on $\cC(\bX)_+$ will be denoted by $\|\cdot\|_{\infty}^{\bX}$.}

$\S^n = \{\mM \in \R^{n\x n} \; ; \mM^\top = \mM\}$ is the vector space of symmetric real matrices of size $n$ and $\S^n_+$ is the usual cone of {positive semidefinite matrices}. For $\vh = (h_1,\ldots,h_r) \in \R[\vx]^r$, define the basic semialgebraic set $\bS(\vh) :={\{\vx \in \R^n: \vh(\vx) \in \R_+^r\} }$.

We make the convention that the {\bf Assumptions} hold throughout the whole paper while {\bf Conditions} hold only when explicitly stated.

\subsection{(Effective) Putinar's Positivstellensatz and polynomial operators} \label{sec:lasserre}

The moment-SoS hierarchy builds on real algebraic geometry certificates for non-negativity of polynomials {and} signed Borel measures. Further it allows {recasting} certain infinite dimensional LPs a sequence of finite-dimensional convex optimization problems. A central pillar in this procedure is the celebrated Positivstellensatz by Putinar.

\begin{thm}[Putinar's Positivstellensatz {\cite[Theorem 1.3 \& Lemma 3.2]{putinar1993positive}}] \label{thm:psatz} \leavevmode

Let $r,n \in \N^\star$ be positive integers, $\vh \in \R[\vx]^r$ a family of $r$ polynomials in $n$ variables.

\noindent Introduce the closed semialgebraic set ${\bX} := \left\{\vx \in \R^n \; ; \vh(\vx) \in \R_+^r\right\}$ as well as the {\bf convex cones}

\begin{center}
\begin{tabular}{ll}
$\cP({\bX})_+ := \left\{p \in \R[\vx] \; ; p({\bX}) \subset \R_+\right\}$ & $\Sigma[\vx] := \left\{\sum_{k=1}^K p_k^2 \; ; K \in \N^\star, p_1,\ldots,p_K \in \R[\vx]\right\}$ \\
$\cP({\bX})_\oplus := \left\{p \in \R[\vx] \; ; p({\bX}) \subset \R_\oplus\right\}$ & $\cQ(\vh) := \left\{\begin{pmatrix} \vh^\top & 1 \end{pmatrix} \bsigma \; ; \bsigma \in \Sigma[\vx]^{r+1}\right\} \subset \cP({\bX})_+$.
\end{tabular}
\end{center}

\noindent If there exists $R \in \mathbb{Q}$ s.t. $R^2 - \vx^\top\vx \in \cQ(\vh)$ \emph{(Archimedean property)}, then $\cP({\bX})_\oplus \subset \cQ(\vh)$.

\noindent Under the same Archimedean condition, the dual cones $\cQ(\vh)'$ and $\cM({\bX})_+$ are isomorphic.
\end{thm}

\begin{rem}[On the Archimedean condition] \label{rem:archi} \leavevmode

As $\cQ(\vh) \subset \cP({\bX})_+$, the Archimedean property automatically yields that 
\begin{equation}\label{eq:ballinc}
    {\bX} \subset \bB_R := \left\{\vx \in \R^n \; ; \vx^\top\vx \leq R^2 \right\},
\end{equation}
i.e. ${\bX}$ is bounded (and thus compact as it is closed).  Conversely, if ${\bX} \subset \bB_R$ for some $R>0$, then adding a polynomial $h_{r+1} := R^2 - \vx^\top\vx$ to $\vh$ does not change the geometry of ${\bX}$, while it results in adding $h_{r+1}$ to $\cQ(\vh)$. Thus, in practice, the Archimedean condition is considered equivalent to compactness of ${\bX}$.
\end{rem}

Effective versions of Putinar's Positivstellensatz quantify the degree of SoS multipliers in a SoS representation of positive polynomials. Therefore it is helpful to introduce the truncated quadratic module $\cQ_{\ell}(\vh)$ defined for $\ell \in \N$ by
\begin{equation} \label{eq:QuadModule}
\cQ_{\ell}(\vh) := \left\{ \begin{pmatrix} \vh^\top & 1 \end{pmatrix} \bsigma \in \cQ(\vh) \; ; \forall i \in \lint r \rint , \; \max(\deg(\sigma_i \, h_i), \deg(\sigma_{r+1})) \leq 2\ell \right\}
\end{equation}
which happens to be a finite-dimensional convex cone.

In this text, we will use an effective version of Putinar's Positivstellensatz from~\cite{baldi2021moment}. To state the result we need to introduce the \L{}ojasiewicz exponent.

\begin{thm}[\L{}ojasiewicz exponent {\cite[{\bf Theorem 2.3}, {\bf Definition 2.4}]{baldi2021moment}}] \label{thm:loja} \leavevmode

For $\vx \in [-1,1]^n$, let
$$ H(\vx) := |\min(h_1(\vx),\ldots,h_r(\vx),0)|
\qquad
D(\vx) := \min\{|\vx-\vx'| \; ; \vx' \in \bS(\vh)\},$$
where $|\vx| := \sqrt{\vx^\top\vx}$ is the Euclidean norm of $\vx$. Then there exists \L$,\mathfrak{c} \in \R_+^\star$ s.t. for $\vx \in [-1,1]^n$
\begin{equation} \label{eq:Loja}
D(\vx)^{\text{\L}} \leq \mathfrak{c} \, H(\vx).
\end{equation}
\end{thm}

For the statement of the following {\bf Theorem~\ref{thm:OldBaldi}} we introduce the additional notations
\begin{equation*}
    p^\star_\bX := \inf\limits_{\vx \in \bX} p(\vx), \quad \|p\| := \max\limits_{\vx \in [-1,1]^{n}} |p(\vx)|.
\end{equation*}

\begin{thm}[Effective Putinar Positivstellensatz {\cite[{\bf Theorem 1.7}]{baldi2021moment}}]\label{thm:OldBaldi} \leavevmode

Let $n \geq 2$, $p \in \cP(\bX)_\oplus$. Assume, what later will be denoted by $\mathbf{Assumption \ \ref{asm:poprateOld}}$, {that is}
\begin{equation*}
    1 - \vx^\top \vx \in \cQ(\vh) \quad \text{and} \quad \forall i \in \lint r \rint, \quad \text{and} \quad \|h_{i}\| \leq \frac{1}{2}.
\end{equation*}
Then one has
\begin{equation}\label{eq:effputinarOld}
\ell \geq \gamma(n,\vh) \deg(p)^{3.5n\text{\L}} \left(\nicefrac{\|p\|}{p^\star_\bX}\right)^{2.5n\text{\L}} \quad \Longrightarrow \quad p \in \cQ_\ell(\vh)
\end{equation}
where $1 \leq \gamma(n,\vh) \leq \Gamma \, n^3 \, 2^{5\text{\L}-1} \,r^n \, \mathfrak{c}^{2n} \, \deg(\vh)^n$ and $\Gamma > 0$ does not depend on $n,p,\vh$. In the rest of this paper, we will consider fixed $n$ and $\vh$, so that we simplify the notation $\gamma(n,\vh)$ into $\gamma$.
\end{thm}

\begin{rem}[Farkas Lemma] \leavevmode

Regardless of $n$, if $\deg(p) = \deg(\vh) = 1$ (affine forms), then  $p \in \cP(\bX)_+ \Longleftrightarrow p \in \cQ_1(\vh)$.
\end{rem}

In the rest of this section, we are concerned with operators $\cA$ that preserve polynomial structure.
\begin{dfn}[Polynomial operator]\label{def:PolyOp}
    Let $\bX \subset \R^{n}$ and $\bY \subset \R^{m}$ and $\cA: \cP(\bY)\rightarrow \cC(\bX)$ be a linear operator. We call $\cA$ a polynomial operator if $\cA \, w \in \cP(\bX)$ for all $w \in \cP(\bY)$.
\end{dfn}
Examples of polynomial operators are restriction operators $\cA \ w := w\big|_\bX$ when $\bY \supset \bX$, multiplication operators $\cA \ w := w \cdot p$ for $p\in \cP(\bX)$, differential operators $\cA:\cP(\bX)\rightarrow \cC(\bX)$ with $\bX = \bY$ and $ \cA \ w := \grad \ w \cdot \vp$ for a map $\vp\in \cP(\R^{n})^{n}$, or composition operators $\cA \ w:= w\circ p$ for a polynomial map $p:\bX \rightarrow \bY$.

Dual to polynomial operators are moment operators $\cB:\cM(\bX)\rightarrow \cM(\bY)$ whose action on an element $\mu \in \cM(\bX)$ can be described by an action on the moments of $\mu$. An operator $\cB$ is called a moment operator if there exists a sequence of polynomials $(\varphi_{\bbeta})_{\bbeta} \in \cP(\bX)^{\N^{n}}$ such that, for all $\mu \in \cM(\bX)$, $\cB \, \mu$ is the linear operator defined on the monomial basis by

\begin{equation} \label{eq:momop}
    \langle \cB \, \mu, \vy^{\bbeta}\rangle = \int \varphi_{\bbeta} \; \od\mu.
\end{equation}

\begin{lem}\label{lem:stab} \leavevmode
    Let $\cA:\cP(\bY) \rightarrow \cC(\bX)$ and $\cB:\cM(\bX)\rightarrow \cM(\bY)$ be linear operators with $\cB = \cA'$. Then $\cA$ is a polynomial operator if and only if $\cB$ is a moment-operator.
\end{lem}

\begin{proof}
    Let us assume that $\cB = \cA'$ is a moment operator. By linearity of $\cA'$, it is sufficient to prove that for all $\bbeta \in \N^{m}$, $\cA[\vy^{\bbeta}] \in \cP(\bX)$. Let $\bbeta \in \N^{m}$, $\vx \in \bX$ and consider the Dirac measure $\mu = \delta_{\vx}$. By ~\eqref{eq:momop} and by definition of the adjoint operator $\cB = \cA'$, one has,
    \begin{align*}
        \cA[\vy^{\bbeta}](\vx) & = \int \cA[\vy^{\bbeta}] \; \od\delta_{\vx} = \int \cA[\vy^{\bbeta}] \; \od\mu  = \langle \cB \, \mu , \vy^{\bbeta}\rangle \overset{\eqref{eq:momop}}{=} \int \varphi_{\bbeta} \; \od\mu = \int \varphi_{\bbeta} \; \od\delta_{\vx} = \varphi_{\bbeta}(\vx).
    \end{align*}
    As this holds for all $\vx \in \bX$, we deduce that $\cA[\vy^{\bbeta}] = \varphi_{\bbeta} \in \cP(\bX)$. We now prove the reverse implication. Let $\mu \in \cM(\bX)$, $\bbeta \in \N^{m}$. Then, one has
    $$\langle \cB \, \mu, \vy^{\bbeta}\rangle = \int \underset{\varphi_{\bbeta}}{\underbrace{\cA[\vy^{\bbeta}]}} \; \od\mu$$
    which concludes the proof because $\varphi_{\bbeta} = \cA[\vy^{\bbeta}] \in \cP(\bX)$ since $\cA$ is a polynomial operator.
\end{proof}

\begin{cor}[Action on bounded degree polynomials] \label{cor:momop} \leavevmode

    Let $\cA$ be a polynomial operator. For all $\ell \in \N$, there exists $d_\ell \in \N$ such that any $p \in \mathrm{Im}(\cA) \cap \R_{2\ell}[\vx]$ has a preimage of degree at most $d_\ell$:
    $$ \exists w \in \R_{d_\ell}[\vy] \quad \text{with} \quad p = \cA\, w. $$
\end{cor}
\begin{proof}
{The set $\mathrm{Im}(\cA) \cap \R_{2\ell}[\vx]$ is a finite dimensional vector space. Thus, there exist $u \in \N$ and $p_1,\ldots,p_u \in \R[\vy]$ such that
\begin{equation*}
    \mathrm{span}\{\cA p_1,\ldots,\cA p_u\} = \mathrm{Im}(\cA) \cap \R_{2\ell}[\vx].
\end{equation*}
Setting $d_\ell := \max\limits_{i = 1,\ldots, u} \deg p_i$, we have
\begin{equation*}
    \cA \left( \R[\vy]_{d_\ell}\right) \supset \cA \left(\mathrm{span}\{p_1,\ldots,p_u\} \right) = \mathrm{span}\{\cA p_1,\ldots,\cA p_u\} = \mathrm{Im}(\cA) \cap \R_{2\ell}[\vx].
\end{equation*}
This shows that any element in $\mathrm{Im}(\cA) \cap \R_{2\ell}[\vx]$ has a preimage in $\R[\vy]_{d_\ell}$, i.e. the statement.
}
\end{proof}

\subsection{Generalized Moment Problem} \label{sec:gmp}

Let $\bX \subset \R^{n}$ and $\bY \subset \R^{m}$ be compact. Using the duality between the space of continuous functions $\cC(\bX)$ and the space of Borel measures $\cM(\bX)$ we set

\begin{center}
\begin{tabular}{ll}
$\cX := \cC(\bX)$,  & $\cY \:= \cP(\bY)$.\\
$\cX' := \cM(\bX)$,  & $\cY' := \cP(\bY)'$,  \\\\
\end{tabular}
\end{center}

On $\cX$ we use the topology of uniform convergence, on $\cY$ we use an appropriate norm on functions depending on the problem at hand (e.g. a $C^k$ norm or an $L^s$ norm; this choice determines $\cY'$) and on $\cX'$ and $\cY'$ we use the induced weak-$*$ topologies.

Let $\cA : \cY \longrightarrow \cX$ be a continuous linear operator, $T \in \cY'$ be a linear form (i.e. \textit{moment} sequence), $g \in \cX$ be a vector of continuous functions. We define the functional LP as

\begin{equation}\label{eq:smp}
\fcolorbox{red}{white}{$
\begin{array}{lcl}
\kappa^\star := & \inf\limits_{w \in \cY} & \langle T ,  w \rangle \\
& \st & \cA \, w - g \in \cC(\bX)_+
\end{array}
$}
\end{equation}

\begin{rem}[On the generality of~\eqref{eq:smp}] \leavevmode
    Note that the generic framework $\bX = \bS(\vh)$ allows for any compact basic semialgebraic set $\bX$. In particular, binary optimization has an LP formulation as~\eqref{eq:smp} which is proved to be equivalent to semidefinite programming~\cite{laurent2005semidefinite}. From this simple remark, one can observe that problem~\eqref{eq:smp} can feature finite-sized PSD constraints, as proposed in~\cite{miller2023var}. Hence, the tools displayed in this work can be used on functional LPs featuring finite-sized LP, convex QP, SOCP, and SDP constraints on coefficients of the decision variables.
\end{rem}

For the rest of the text, we will make the following assumptions.

\begin{asm}[Existence of feasible solutions]\label{asm:existenceSol}\leavevmode
The feasible set of~\eqref{eq:smp} is not empty.
\end{asm}

\begin{asm}[Polynomial operator $A$ and polynomial $g$] \label{asm:momop} \leavevmode
The operator $\cA$ is a polynomial operator and $g \in \cP(\bX)$.
\end{asm}

\begin{asm}\label{asm:poprateOld} \leavevmode The set $\bX$ is a compact basic semialgebraic with representation $\bX := \bS(\vh) \subset \R^{n}$ with $\vh = (h_{1},\ldots,h_{r}) \in \R[\vx]^{r}$ for some $r \in \N^\star$, and for which it holds
\begin{enumerate}
    \item $1 - \vx^\top \vx \in \cQ(\vh)$ \emph{(normalized Archimedean property)},
    \item $\forall j \in \lint r \rint, \quad \|h_{j}\| := \max\limits_{\vx \in [-1,1]^n} |h_{j}(\vx)| \leq \frac{1}{2}$
\end{enumerate}
\end{asm}

The Assumptions~\ref{asm:existenceSol},~\ref{asm:momop} and~\ref{asm:poprateOld} guarantee that the LP~\ref{eq:smp} is non-trivial, has polynomial structure, and that we will be able to apply the effective Positivstellensatz {\bf Theorem~\ref{thm:OldBaldi}}.

\begin{rem}[On the validity of {\bf Assumption \ref{asm:momop}}] \label{rem:stab} \leavevmode
    
    To our best knowledge, many existing formulations of~\eqref{eq:smp} for problems with polynomial data (including those in \cite{decastro2016super, fantuzzi2022verification, fantuzzi2020bounding, goluskin2020bounding, jones2021converse, korda2017convergence, lasserre2009moments, miller2023bounding, mula2024moment, schlosser2021converging}) satisfy {\bf Assumption \ref{asm:momop}}. Indeed, operations such as summation, polynomial multiplication/pushforward/composition and differentiation are polynomial operators.
\end{rem}

Through Assumption \ref{asm:poprateOld} we impose the algebraic structure on the problem that will allow us to apply the effective version of Putinar's Positivstellensatz \textbf{Theorem}~\ref{thm:OldBaldi}.

\begin{rem}[On the validity of {\bf Assumption \ref{asm:poprateOld}}] \leavevmode

    The normalized Archimedean property can be seen as a restatement of compactness of $\bX$. For compact $\bX = \bS(\vh)$, up to rescaling, $\bS(\vh)$ is included in the unit ball so that it is possible to add the redundant inequality constraint $1 - \vx^\top\vx \geq 0$ to the description of $\bS(\vh)$. This is the practical approach for guaranteeing {\bf Assumption \ref{asm:poprateOld}.1}. The second condition in {\bf Assumption \ref{asm:poprateOld}} is only of technical nature and can be obtained by scaling $\vh$.
\end{rem}

The optimization problem~\eqref{eq:smp} is an infinite dimensional instance of conically constrained linear programs (CCLP), and as such it is subject to Lagrange duality. To formulate the dual problem, we introduce the Lagrange operator
\begin{equation} \label{eq:lagrange}
\Lambda := \cX' \x \cY \ni (\mu,w) \longmapsto \langle T ,  w \rangle + \int g - \cA \, w \; \od\mu.
\end{equation}

It is straightforward (using the fact that $\sup_{\mu \in \cM(\bX)_+} \int g - \cA \, w \; \od\mu = \infty$ iff there exists $\vx\in \bX$ with $\cA \, w (\vx) - g(\vx) <0$) that
$$\gmd = \inf\limits_{w \in \cP(\bY)}\vphantom\int\sup\limits_{\mu \in \cM(\bX)_+ }\vphantom\sum\Lambda(\mu,w)$$
Finally, the dual problem to \eqref{eq:smp} is obtained by swapping the $\sup$ and $\inf$ operators:
$$\gmp := \sup\limits_{\mu \in \cM(\bX)_+ }\inf\limits_{w \in \cP(\bY)}\vphantom\sum\Lambda(\mu,w).$$

Using the the adjoint operator of $\cA\big|_{\cY}$, given by $\cA': \cX' \rightarrow \cY'$  with
$$ \langle \cA' \mu , w \rangle :=  \int \cA \, w \; \od\mu $$
for $w\in \cY = \cP(\bY)$ and $\mu \in \cX' = \cM(\cX)$, we have $\Lambda(\mu,w) = \displaystyle \langle T - \cA'\mu , w \rangle + \int g \; \od\mu$. Again, $\inf_{w\in\cY} \langle T - \cA'\mu , w \rangle = -\infty$ iff $T - \cA'\mu \neq 0$, gives
\begin{equation} \label{eq:gmp}
\fcolorbox{red}{white}{$
\begin{array}{lcl}
\gmp = & \sup\limits_{\mu\in\cX'} & \displaystyle\int g \; \od \mu \\
& \st & \cA' \mu = T.
\end{array}
$}
\end{equation}
{We refer to~\eqref{eq:gmp} as a \textit{generalized moment problem} (GMP).\footnote{The case where $g=0$, $\bX = \bY$, $\cA' \, \mu = (\int \vx^{\balpha} \; \od\mu)_{\balpha \in \N^n}$ {and $T = (t_{\balpha})_{\balpha \in \N^n}$ a given sequence of scalars}, is called the $\bX$-moment problem.} While in this work we compute convergence rates through the analysis of~\eqref{eq:smp} (so that we call it the primal problem, and introduced the GMP as its dual), we wish to emphasize that in many cases, the meaning of the LP at hand is best captured by~\eqref{eq:gmp}, which in most existing contributions is the first to be formulated, and hence often called primal problem.}

This duality between~\eqref{eq:gmp} and~\eqref{eq:smp} comes with two interesting properties~\cite{barvinok2002convexity,tacchi2022convergence}, which we state next.

\begin{prop}[Weak duality] \label{prop:weak} \leavevmode

In all generality, with the above notations, one has $\gmp \leq \gmd$.
\end{prop}

\begin{prop}[Strong duality] \label{prop:strong} \leavevmode

One has $\gmp = \gmd$ if one of the following two {\bf Conditions} is satisfied.
\end{prop}
\begin{cond}[Slater {\cite{slater1950lagrange}}] \label{condition:slater}
    $\exists \overset{\circ}{w} \in \cY \ \st \ \cA\,\overset{\circ}{w} - g \in \cC(\bX)_\oplus$.
\end{cond}
\begin{cond}[Measure compactness {\cite{tacchi2021thesis}}] \label{condition:tacchi}
    $\exists B > 0 \ \st \ \forall \mu \in \cX'$ feasible for the generalized moment problem in~\eqref{eq:gmp}, one has $\displaystyle \int 1 \; \od\mu \leq B $.
\end{cond}

{\bf Condition~\ref{condition:slater}} is instrumental in numerically constructing approximate solutions of~\eqref{eq:smp}, while {\bf Condition \ref{condition:tacchi}} is used in~\cite{tacchi2022convergence} to prove a strong convergence result on the numerical approximation of~\eqref{eq:gmp}. Ideally, one would like to deduce both conditions from one, stronger condition.

\begin{lem}\label{lem:StricPosImageA}
    If there exists $\hat{w} \in \cY$ with $\cA\, \hat{w} > 0$ on $\bX$ then {\bf Condition~\ref{condition:slater}} and {\bf Condition~\ref{condition:tacchi}} are satisfied.
\end{lem}

\begin{proof}
    We first verify that {\bf Condition~\ref{condition:slater}} is verified. By compactness of $\bX$ there exists $\delta > 0$ with $\cA\, \hat{w} \geq \delta$ on $\bX$. Setting $\overset{\circ}{w}:= \delta^{-1}\cdot(1 + \max\limits_{\vx\in \bX}g(\vx)) \cdot \hat{w}$ we have on $\bX$
    \begin{equation*}
        \cA \, \overset{\circ}{w} - g = \delta^{-1}\cdot(1 + \max\limits_{\vx\in \bX}g(\vx)) \, \cA \, \hat{w} - g \geq 1 + \max\limits_{\vx\in \bX}g(\vx) - g \geq 1 > 0.
    \end{equation*}
    This shows Slater's {\bf Condition~\ref{condition:slater}} is satisfied. To see that {\bf Condition~\ref{condition:tacchi}} is satisfied, let $\mu$ be feasible for \eqref{eq:gmp}. By {the} definition of $\delta$, one gets
    \begin{align*}
         \int 1 \; \od\mu & \leq \delta^{-1} \int \cA\,\hat{w} \; \od\mu = \delta^{-1} \langle \cA' \mu , \hat{w} \rangle = \delta^{-1} \langle T , \hat{w} \rangle =: B < \infty
    \end{align*}
    which is exactly {\bf Condition \ref{condition:tacchi}}.
\end{proof}

As a consequence of {\bf Lemma~\ref{lem:StricPosImageA}}, Slater's {\bf Condition~\ref{condition:slater}} implies {\bf Condition \ref{condition:tacchi}} when $g\in \cP(\bX)_+$. The reason is that $g \geq 0$ together with Slater's {\bf Condition~\ref{condition:slater}} implies $\cA \, \overset{\circ}{w} > g \geq 0$, i.e. the condition in {\bf Lemma~\ref{lem:StricPosImageA}}.

\subsection{The moment-SoS hierarchy} \label{subsec:hierarchy}
In this text, we {apply the so-called moment-SoS hierarchy for solving the functional LP~\eqref{eq:smp} and its dual, the GMP~\eqref{eq:gmp}. This follows an established line of reasoning based on so-called SoS tightenings respectively moment relaxations and semidefinite programming. More precisely, } it consists of the following steps.

\begin{subequations}
{\bf Step 1:} First, we make use of {\bf Assumption \ref{asm:existenceSol}} and {\bf Assumption \ref{asm:momop}}, and we assume that {\bf Condition \ref{condition:slater}} holds. We look for a minimizing sequence of \textit{strictly feasible} polynomials $w$ {for the LP~\eqref{eq:smp}}. Without changing the optimal value, we get:

\begin{equation} \label{eq:pos}
\begin{array}{lcl}
    \gmd = & \inf\limits_{w \in \R[\vy]} & \langle T ,  w \rangle \\
& \st & \cA \, w - g \in \cP(\bX)_\oplus.
\end{array}
\end{equation}

{\bf Step 2:} By {\textbf{Assumption~\ref{asm:poprateOld}}}, using {\bf Theorem \ref{thm:psatz}}, we recast the positivity constraint of \eqref{eq:pos} as a quadratic module constraint: 

\begin{equation} \label{eq:sos}
\begin{array}{lcl}
\gmd = & \inf\limits_{w \in \R[\vy]} & \langle T ,  w \rangle \\
& \st & \cA \, w - g \in \cQ(\vh).
\end{array}
\end{equation}

{\bf Step 3:} Eventually, we \textit{project} our infinite-dimensional quadratic module onto the bounded degree quadratic module $\cQ_{\ell}(\vh)$ (see~\eqref{eq:QuadModule} for the definition), obtaining the following SoS programming problem:
\begin{equation} \label{eq:sosd}
\fcolorbox{red}{white}{$
\begin{array}{lcl}
\gmd_\ell := & \inf\limits_{w \in \R_{d_\ell}[\vy]} & \langle T ,  w \rangle \\
& \st & \cA \, w - g \in \cQ_\ell(\vh),
\end{array}
$}
\end{equation}

where {\bf Corollary \ref{cor:momop}} is used to bound the degree of $w$. Note that \eqref{eq:sosd} is a \textit{tigthening} of \eqref{eq:sos} in the sense that we replaced the feasible set with a strictly smaller one (even finite-dimensional): in general,
$$ \gmd_\ell > \gmd.$$
Moreover, as $\cQ(\vh) = \cup_{\ell \in \N} \cQ_\ell(\vh)$ and $\cQ_L(\vh) = \cup_{\ell \leq L} \cQ_\ell(\vh)$ are clear, one also has the following monotone convergence theorem for free:
$$ \gmd_\ell \underset{\ell \to \infty}{\searrow} \gmd. $$

\end{subequations}

\subsection{The general case}
In applications, typically a more general situation occurs in which the spaces $\cX$ and $\cY$ are product space
\begin{center}
\begin{tabular}{ll}
$\cX := \cC(\bX_1)\x\ldots\x\cC(\bX_N)$,  & $\cY := \cP(\bY_1)\x\ldots\x\cP(\bY_M)$.\\
$\cX' := \cM(\bX_1)\x\ldots\x\cM(\bX_N)$,  & $\cY' \:= \cP(\bY_1)'\x\ldots\x\cP(\bY_M)'$, \\
\end{tabular}
\end{center}
for $N,M \in \N^\star$, $n_1,\ldots,n_N,m_1,\ldots,m_M \in \N$ and sets $\bX_i \subset \R^{n_i}$, $\bY_j \subset \R^{m_j}$.

The notions, results, and strategies carry over to product spaces. This is why, for notational simplicity, we continue with the setting of $\cX = \cC(\bX)$ and $\cY = \cP(\bY)$. For the general LPs, notationally not much changes, only that now $\cX, \cX'$ and $\cY,\cY'$ are products and $\bT = (T_1,\ldots,T_M)$, $\vg = (g_1,\ldots,g_N)$, $\vw = (w_1,\ldots,w_M)$ and $\bmu = (\mu_1,\ldots,\mu_N)$ have several components:
\begin{equation} \label{eq:LPGeneral}
\fcolorbox{red}{white}{$
\begin{array}{lclclcl}
\gmd= & \inf\limits_{\vw \in \cY} & \sum\limits_{i=1}^N \langle T_i ,  w_i \rangle & \quad \text{ and } \quad \quad & \op = & \sup\limits_{\bmu \in \cX'} & \sum\limits_{i=1}^M \displaystyle\int g_i \; \od \mu_i\\
& \st & \forall i \in \lint N \rint & & & \st & \forall i \in \lint N \rint, \quad \mu_i \in \cM(\bX_i)_+ \vphantom{\sup\limits_{\mu \in \cM(\bX)} \displaystyle\int}\\
& & (\cA \, \vw - \vg)_i \in \cC(\bX_i)_+ & & & & \cA' \, \bmu = \bT
\end{array}$}
\end{equation}

For details and notational subtleties, we refer to {\bf Appendix~\ref{appendix:ProductSpaces}}.

\begin{rem}[Existing GMPs] \leavevmode
    
    The framework of GMPs covers a large class of problems, notably polynomial optimization \cite{lasserre2009moments}, but also the LPs from \cite{bertsimas2002relation,curmei2023shape,decastro2016super,fantuzzi2022verification, fantuzzi2020bounding, goluskin2020bounding, jones2021converse,korda2018convergence, korda2017convergence,laraki2012semidefinite, lasserre2009moments, miller2023bounding, mula2024moment, parrilo2000structured,schlosser2021converging,tacchi2023stokes}, to name only a few, can all be represented in the form~\eqref{eq:LPGeneral}.
\end{rem}

\section{Method for convergence rates computation} \label{sec:rates}

The aim of this section (and more generally of this article) is to design a method for computing the rate of the convergence of the optimal values {$\gmd_\ell$ in~\eqref{eq:sosd} towards the optimal value $\gmd$ of~\eqref{eq:smp}}. From particular examples, we derive a generic method for computing such convergence rate, depending on the solutions of the infinite-dimensional problem \eqref{eq:LPGeneral}.

Our strategy consists of the following steps:
\begin{enumerate}
    \item {\bf Construction of a suitable minimizing sequence of polynomials {for (\ref{eq:smp})}.} In this step, it is important to control -- quantitatively and simultaneously -- the degree of those polynomials, their feasibility, and the convergence of their cost toward the optimal value.
    \item {\bf Application of an effective version of Positivstellensätze.} In this step, explicit convergence rates are derived. They are based on the convergence rates for Positivstellensätze and the minimizing sequence from the previous step.
\end{enumerate}

There is an interplay between the two steps inherent to the choice of the minimizing sequence. We will see an adversarial behavior between, on the one hand, a good approximation of the optimal point via high-degree polynomials and, on the other hand, degree bounds in the SDP relaxations.

\begin{rem}[Focusing on the function LP]  \leavevmode

    We will focus on the functional LP~\eqref{eq:smp} and not on the GMP~\eqref{eq:gmp} simply because we will use the effective version of Putinar's Positivstellensatz \textbf{Theorem \ref{thm:OldBaldi}}, which is more adapted to the LP \eqref{eq:smp} than to~\eqref{eq:gmp}. However, under Slater's \textbf{Condition \ref{condition:slater}}, strong duality holds and the two problems are equivalent.
\end{rem}

\begin{rem}[Sparse and symmetric problems]\label{rem:SparseAndSymProblems} \leavevmode

    The number of variables in the SDP for the $\ell$-th level of the moment-SoS hierarchy grows combinatorial with $\ell \in \N$. Thus exploiting sparsity or symmetry, when present, is important in practice. For problems~\eqref{eq:smp} with symmetry, we can restrict to symmetric solutions~\cite[Proposition 4.1]{fantuzzi2020bounding} for which the symmetry in the SDPs can be exploited without loss of accuracy~\cite{riener2013exploiting}. Therefore, whenever the symmetry is compatible with strict feasibility, convergence rates translate from the full moment-SoS hierarchy to the symmetry-reduced one. Correlation-sparsity can be treated via \cite{korda2023convergence} {allowing the transfer of convergence rates}. For sparse dynamical systems the convergence rates can even be improved as long as the bounds in the effective version of Putinar's Positivstellensatz grow with increasing state dimension, see \cite{schlosser2020sparse}.
\end{rem}

\subsection{General method and function approximation} \label{sec:method}

In this section, we specify the procedure that we have indicated at the beginning of Section \ref{sec:rates}. Let $\epsilon > 0$. Supposing that~\eqref{eq:smp} has an optimal solution $w^\star\in \R[\vx]$ and noting $d := \deg(w^\star)$, we want to perturb it with some $\tilde{w} \in \R_d[\vx]$ such that $\hat{w} := w^\star + \tilde{w}$ satisfies $\langle T, \hat{w} \rangle \leq \gmd + \epsilon$ and $p := \cA \, \hat{w} - g \in \cP(\bX)_\oplus$, \textit{i.e.} $\hat{w}$ is \textit{strictly} feasible for \eqref{eq:smp} and thus feasible for \eqref{eq:sos} by {\bf Theorem \ref{thm:psatz}}. Then, {\bf Theorem \ref{thm:OldBaldi}} gives a lower bound on $\ell \in \N$ such that $\hat{w}$ is feasible for our SoS strenghtening \eqref{eq:sosd}, which will prove that $\gmd\leq \gmd_{\ell} \leq \gmd + \epsilon$.

Hence, the general process for computing the degree $\ell$ needed for a given $\epsilon > 0$ accuracy and the corresponding convergence rate is summarized as follows:

\begin{enumerate}
    \item[1.] Take a minimizer $w^\star$ of the LP \eqref{eq:smp}.
    \item[2.] If no minimizer exists, then take an approximate minimal point $w_\epsilon \in \cY$ with $\langle T, w_\epsilon\rangle \leq \gmd+ \nicefrac{\varepsilon}{2}$.
    \item[3.] Perturb the minimizer $w^\star$ (resp. $w_\epsilon$) into a strictly feasible polynomial $\hat{w}$ with $\langle T, \hat{w} \rangle \leq \gmd+ \varepsilon$.
    \item[4.] Apply (effective) Positivstellenstätze{, such as \textbf{Theorem \ref{thm:OldBaldi}},} to show that $\hat{w}$ is feasible for the SDP hierarchy at some level $\ell \in \N$.
    \item[5.] Relate the approximation error $\varepsilon$ and the hierarchy level $\ell$ to derive a convergence rate.
\end{enumerate}

{Steps 1-4 in the above procedure guarantee asymptotic convergence of the moment-SoS hierarchy. To obtain explicit convergence rates, the effective Positivstellensatz and regularity properties of the minimizer $w^\star$ play central roles. Their interplay is discussed in the following remark.}

\begin{rem}[Desirable properties of the minimizing sequence]
    A good choice of minimizing sequence arises from an interplay of two properties: fast convergence of the respective cost and compatibility with the effective Positivstellensatz. The former property determines how fast the minimizing sequence approaches the optimal value and the latter determines at which level of the SoS hierarchy the elements in the minimizing sequence are feasible for the SDPs~\eqref{eq:sosd}. Below we specify these properties a bit more and relate them to the optimal point: 
    \begin{enumerate}
        \item Convergence to the optimal cost: To achieve a good approximation of the optimal cost, the regularity of the optimal point $w^\star$ and the geometry of the cost $T$ should be leveraged.
        \item Compatibility with the effective Positivstellensatz: This typically includes moderate growth of the degree of the polynomials in the minimizing sequence, moderate growth in the supremum on the set of interest, and sufficient strict feasibility.
    \end{enumerate}
    Some of the stated properties of a good minimizing sequence are adversarial to each other. For instance, consider the situation where the optimal point is not strictly feasible. Then, a fast approximation of the optimal point by a minimizing sequence enforces some feasibility constraints to become tight rapidly, which negatively affects the effective degree bound in Putinar's Positivstellensatz.
\end{rem}

In the applications in Sections \ref{sec:ApplicationDynSys}, \ref{sec:vol} we will encounter situations where no polynomial minimizer exists but we know a non-polynomial ``candidate-minimizer'' $\Bar{w}$. Formally, we can relax the LP~\eqref{eq:smp} to include such ``candidate-minimizer''. We discuss such an extension of~\eqref{eq:smp} in the following Remark.

\begin{rem}[Relaxing~\eqref{eq:smp}]\label{rem:NoMinimizer}
    If~\eqref{eq:smp} does not have an optimal point, it can be helpful to relax~\eqref{eq:smp} into
    \begin{equation} \label{eq:gfp}
        \begin{array}{lcl}
            \kappa^\dagger = & \inf\limits_{w \in \hat{\cY}} & \langle \hat{T} ,  w \rangle \\
            & \st & \hat{\cA} \, w - g \in \hat{\cX}_+.
        \end{array}
    \end{equation}
    where $\hat{\cY} \supset \cY$ and $\hat{\cX} \supset \cX$ are 
    {vector spaces equipped with
    ``well-chosen'' topologies (preserving continuity of $\cA$ and $T$ and such that $\cY$ resp. $\cX$ is dense in $\hat{\cY}$ resp. $\hat{\cX}$)}, the set $\hat{\cX}_+ \subset \hat{\cX}$ is a cone of non-negative functions on $\bX$, and $\hat{T}:\hat{\cY} \rightarrow \R$ and $\hat{\cA}:\hat{\cY}\rightarrow \hat{\cX}$ are continuous linear extensions of $T$ resp. $\cA$. To our best knowledge, in many relevant applications, a good choice for the topology on $\cY$ often results in \eqref{eq:gfp} having an optimal solution $\bar{w} \in \hat{\cY}$ with $\langle \hat{T}, \bar{w} \rangle = \kappa^\dagger = \gmd$ with $\gmd$ from (\ref{eq:smp}). The steps 1.-5. from the beginning of this section can again be followed. Nevertheless, we want to emphasize that in this procedure, bounding $\ell$ using {\bf Theorem \ref{thm:OldBaldi}} can become much more delicate because we impose less regularity on the function $\bar{w} \in \hat{\cY}$. The relaxation~\eqref{eq:gfp} can be found in this text in the optimal control problem in \textbf{Section~\ref{sec:ocp}} where $\cY$ is extended to $\hat{\cY} := C^1(\bY)$, the exit location problem in \textbf{Section~\ref{sec:ExitLocation}} where $\cY$ is extended to $\hat{\cY} := C^2(\bY)$ (see \textbf{Remark~\ref{rem:ExtensionSDEExit}}), or the volume computation in \textbf{Section~\ref{sec:vol}} where $\cY$ {and $\cX$ are extended to spaces of integrable functions} (see \textbf{Remark~\ref{rem:ExtensionVol}}).
\end{rem}

When, as in~\eqref{eq:gfp}, extending the search space from $\cY$ to $\hat{\cY}$, we relax the problem, i.e. it holds $\kappa^\dagger \leq \gmd$ and the inequality can be strict (for example if the only feasible point is non-polynomial). In the following lemma, we investigate a situation in which the optimal value does not change by relaxing~\eqref{eq:smp} to~\eqref{eq:gfp}, i.e. that $\kappa^\dagger = \gmd$ holds. Once again, Slater's {\bf Condition~\ref{condition:slater}} will be helpful.

\begin{lem}\label{lem:d*=bard*}
    Let {$\hat{\cY}$ be a topological vector space} such that $\cY \subset \hat{\cY}$ {is dense}. Assume $\cA$ and $T$ have continuous extensions $\hat{\cA}:\hat{\cY}\rightarrow \cX$ and $\hat{T}:\hat{\cY}\rightarrow \R$. Then, under Slater's {\bf Condition~\ref{condition:slater}} for~\eqref{eq:smp}, it holds $\kappa^\dagger = \gmd$.
\end{lem}

\begin{proof} 
    Because~\eqref{eq:gfp} is a relaxation of~\eqref{eq:smp} it holds $\gmd \geq \kappa^\dagger$ and it remains to show $\gmd \leq \kappa^\dagger$. To do so, by {\bf Condition~\ref{condition:slater}} for~\eqref{eq:smp}, we can take $v\in \cY$ with $\cA \ v - g > 0$ on $\bX$. Let $w\in \hat{\cY}$ be an arbitrary feasible point for~\eqref{eq:gfp} and let $\varepsilon > 0$. For $t\in (0,1]$ consider the element $w_t:= w + t(v-w) \in \hat{\cY}$. We have
    \begin{equation*}
        \hat{\cA} \, w_t - g = (1-t) \underbrace{(\hat{\cA} \, w - g)}_{\geq 0} + t \underbrace{(\cA v - g)}_{> 0} =: \delta_t > 0,
    \end{equation*}
    i.e. $w_t$ is strictly feasible for~\eqref{eq:gfp}. 
    {By density of $\cY \subset \hat{\cY}$ and continuity of $\hat{\cA}$ and $\hat{T}$, there exists $z_t \in \cY$ with $|\hat{\cA} \,(z_t - w_t)| \leq \frac{\delta}{2}$ on $\bX$ and $|\langle \hat{T} ,  w_t -z_t \rangle| \leq \varepsilon$}. This means that {$\hat{\cA}z_t - g \geq \hat{\cA}w_t - g - \frac{\delta_t}{2} = \frac{\delta_t}{2} > 0$} i.e. $z_t$ is feasible for~\eqref{eq:smp}, and its objective value $\langle T ,  z_t \rangle$ differs at most $\varepsilon$ from $\langle \hat{T} ,  w_t\rangle$. Now, letting $t\rightarrow 0$, we observe that $w_t \rightarrow w$ and get
    \begin{equation*}
        \limsup\limits_{t\searrow 0} |\langle \hat{T} ,  w -z_t \rangle| \leq \limsup\limits_{t\searrow 0} |\langle \hat{T} ,  w -w_t \rangle| + |\langle \hat{T} ,  w_t -z_t \rangle|\leq 0 +\varepsilon.
    \end{equation*}
    This shows $\kappa^\dagger \geq \gmd - \varepsilon$. Because $\varepsilon >0$ was arbitrary we conclude the statement.
\end{proof}

{
\begin{rem}
    In Section \ref{sec:VolumeStandard} we consider an LP whose relaxation~\eqref{eq:gfp} considers a space $\hat{\cX} \supsetneq \cX$, see {\bf Remark~\ref{rem:ExtensionVol}}. In such a situation the above \textbf{Lemma~\ref{lem:d*=bard*}} does not apply and we need another result to conclude. 
\end{rem}

\begin{lem}\label{lem:StandVolRelaxation}
    Let $\hat{\cX} \supset \cX$ and $\hat{\cY} \supset \cY$ be equipped with some topologies such that $\cA$ and $T$ have continuous extensions $\hat{\cA}: \hat{\cY} \longrightarrow \hat{\cX}$ and $\hat{T}: \hat{\cY} \longrightarrow \R$. Assume moreover that:
    \begin{enumerate}
        \item $\cY$ is one-sided dense in $\hat{\cY}$, i.e. $\forall w \in \hat{\cY}$, any open $\mathcal{N} \ni w$ contains a $\tilde{w}\in \cY$ such that $\tilde{w} \geq w$ on $\bY$
        \item $\hat{\cA}$ is positive, i.e. for $w\in \hat{\cY}$ with $w \geq 0$ on $\bY$, it holds $\hat{\cA}w \geq 0$ on $\bX$
    \end{enumerate}
\end{lem}

\begin{proof}
    We only need to show $\gmd \leq \kappa^\dagger$. Therefore, let $w\in \hat{\cY}$ be feasible for~\eqref{eq:gfp} with cost $c := \langle \hat{T},w\rangle$. Let $\varepsilon > 0$ and $\mathcal{N} := \hat{T}^{-1}((c-\varepsilon,c+\varepsilon))$ be an open neighbourhood of $w$. By the one-sided density of $\cY$ in $\hat{\cY}$ there exists $\tilde{w} \in \cY \cap \mathcal{N}$ with $\tilde{w} \geq w$ on $\bY$. By positivity of $\hat{\cA}$ it follows that $\tilde{w}$ is feasible for~\eqref{eq:smp} with cost $\langle T,\tilde{w}\rangle = \langle \hat{T},\tilde{w}\rangle \leq \langle \hat{T},w\rangle + \varepsilon$. Taking the infimum over $\tilde{w}$ and letting $\varepsilon$ to $0$ yields $\gmd \leq \langle \hat{T},w\rangle$. Minimizing w.r.t. $w$ concludes the proof.
\end{proof}
}

Also for the relaxed LP~\eqref{eq:gfp} we aim at approximating a minimizer $\bar{w}\in \hat{\cY}$ of~\eqref{eq:gfp} by suitable polynomials $w_\varepsilon \in \cY$ as in Step 2. from the beginning of this section. The convergence rate is then obtained by combining {\bf Theorem \ref{thm:OldBaldi}} with approximation theorems on upper bounds on the degree $d$ required for $w_\epsilon \in \R_d[\vy]$ to approximate $\Bar{w}$. Such approximation results are treated in the following section. 

\subsection{Polynomial approximation (finding $w_\varepsilon$)} \label{Subsec:PolyApproximation}

This section introduces polynomial approximation results that depend on the regularity of the function to be approximated. For $\bY \subset \R^m$ open or compact with non-empty interior, we need to define the vector spaces $C_b^k(\bY), k\in \N$ by induction: $C^0_b(\bY) := \cC_b(\bY) := \left\{f \in \cC(\bY) \; ; f \text{ is bounded on } \bY\right\}$ and 
$$C_b^{k+1}(\bY) := \left\{f \in \cC_b(\bY) \; : \forall j \in \lint m \rint, \, \frac{\partial f}{\partial y_j} \in C_b^k(\bY)\right\}.$$
These vector spaces are equipped with the norms $\|f\|_{C^0_b(\bY)} := \|f\|^\bY_\infty = \sup_\bY |f|$ and, again by induction, 
$$\|f\|_{C^{k+1}_b(\bY)} := \|f\|^\bY_\infty + \sum_{j=1}^n \left\|\frac{\partial f}{\partial y_j}\right\|_{C^k_b(\bY)}.$$
When $\bY$ is compact, the subscript $b$ is omitted as continuous functions are bounded on compact sets.

An important object is the modulus of continuity $\omega_{f,k}(\vy,\rho)$ of a function $f \in C_b^k(\bY)$ of order $k$ at a point $\vy \in \bY \subset \R^n$ for the radius $\rho >0$, defined as
\begin{subequations} \label{eq:ModConti}
\begin{equation}\label{eq:ModContiX}
    \omega_{f,k}(\vy,\rho) := \sup\limits_{\balpha \in \N^n_k} \left( \sup\limits_{\|\vy-\vy'\| \leq \rho} \left|\partial_{\balpha} f(\vy) - \partial_{\balpha} f(\vy')\right|\right)
\end{equation}
$$\text{where} \qquad \partial_{\balpha}f = \frac{\partial^{\alpha_1}}{\partial y_1^{\alpha_1}} \cdots \frac{\partial^{\alpha_m}}{\partial y_m^{\alpha_m}} f, $$
as well as the global modulus of continuity
\begin{equation}\label{eq:ModContiGlobal}
\omega^{L^\infty}_{f,k}(\bY,\rho) := \|\omega_{f,k}(\cdot,\rho)\|_\infty^\bY = \sup\limits_{\vy \in \bY} \omega_{f,k}(\vy,\rho) \leq \infty
\end{equation}
and, for $\mu \in \cM(\bY)_+$ and $s \geq 1$, the $L^s(\mu)$-averaged modulus of continuity
\begin{eqnarray}\label{eq:ModContiAverage}
    \omega^{L^s}_{f,k}(\mu,\rho) := \|\omega_{f,k}(\cdot,\rho)\|_{L^s(\mu)} := \left(\int \omega_{f,k}(\cdot,\rho)^s \; d\mu\right)^{\nicefrac{1}{s}}
\end{eqnarray}
\end{subequations}

Notice that $\bY \subset \bY' \Longrightarrow \omega_{f,k}^{L^{\infty}}(\bY,\cdot) \leq \omega_{f,k}^{L^{\infty}}(\bY',\cdot)$ and $\mu_1 - \mu_2 \in \cM(\bY)_+ \Longrightarrow \omega_{f,k}^{L^s}(\mu_1,\cdot) \geq \omega_{f,k}^{L^s}(\mu_2,\cdot)$. With the notion of modulus of continuity we can state the following theorem from \cite{bagby2002multivariate} concerning convergence speed for polynomial approximation of regular functions.

\begin{subequations} \label{eq:stone}
\begin{thm}[An extended Jackson inequality \cite{bagby2002multivariate}]\label{thm:CkApprox} \leavevmode

Let $\bY \subset \R^n$ be open and bounded, $f \in C^k_b(\bY)$. For $d \in \N$ there is a polynomial $p_d \in \R_d[\vy]$ such that for each $\balpha \in \N^n$ with $|\balpha| \leq \min (k, d)$ we have
\begin{equation} \label{eq:CkbApprox}
    \left\| \partial_{\balpha} (f - p_d)\right\|^\bY_\infty \leq \frac{c}{d^{k-|\balpha|}} \, \omega_{f,k}^{L^\infty} \left(\bY,\nicefrac{1}{d}\right).
\end{equation}
where $c$ is a positive constant depending only on $n, k$ and $\bY$.
\end{thm}

\begin{cor}[Approximating smooth functions] \label{cor:CkApprox} \leavevmode

    Let $\bY \subset \R^n$ be open and bounded, $f \in C^k(\overline{\bY})$. For $d \geq k$ there is a polynomial $p_d \in \R_d[\vy]$ such that
    \begin{equation} \label{eq:CkApprox}
        \|f-p_d\|_{C^k(\overline{\bY})} \leq c_0 \, \omega_{f,k}^{L^\infty}\left(\overline{\bY},\nicefrac{1}{d}\right)
    \end{equation}
where $c_0$ is a positive constant depending only on $n, k$ and $\overline{\bY}$.
\end{cor}
\begin{proof} Consider the polynomial $p_d$ given by {\bf Theorem \ref{thm:CkApprox}} and $c$ the corresponding constant. Then, one has

    $$\begin{array}{lll}
        \|f-p_d\|_{C^k(\overline{\bY})} & = \sum_{|\balpha| \leq k} \|\partial_{\balpha}(f-p_d)\|_\infty^\bY \vphantom{\displaystyle\int} \\
        & \stackrel{\eqref{eq:CkbApprox}}{\leq} \sum_{|\balpha| \leq k} \frac{c}{d^{k-|\balpha|}} \, \omega_{f,k}^{L^\infty} \left(\bY,\nicefrac{1}{d}\right) \\ & = \frac{c}{d^k} \, \omega_{f,k}^{L^\infty} \left(\bY,\nicefrac{1}{d}\right) \sum_{|\balpha|\leq k} d^{|\balpha|} \vphantom{\displaystyle\int} \\
        & \leq \frac{c}{d^k} \, \omega_{f,k}^{L^\infty} \left(\bY,\nicefrac{1}{d}\right) \, d^k \binom{n+k}{k}\vphantom{\displaystyle\int} \\
        & = \binom{n+k}{k}\  c\ \omega_{f,k}^{L^\infty} \left(\bY,\nicefrac{1}{d}\right)
    \end{array}$$
    where we have used $\sum_{|\balpha|\leq k} d^{|\balpha|} \leq d^k \sum_{|\balpha|\leq k} 1 = d^k \binom{n+k}{k}$. Choosing the constant $c_0 = c_0(n,k,\overline{\bY}) := \binom{n+k}{k}\  c$ concludes the claim.
\end{proof}

With a view toward the problem of computing the volume of a semialgebraic set from Section~\ref{sec:vol}, the following one-sided approximation result is useful.

\begin{thm}[One-sided polynomial approximation \cite{bhaya2010interpolating}] \label{thm:onesideApprox} \leavevmode

    Let $\bY \subset [-1,1]^m$, $\lambda_\bY$ be the Lebesgue measure on $\bY$ and $f : \bY \longrightarrow \R$ be bounded and measurable.

    For all $s \in [1,\infty)$ and $d \in \N$ there is a polynomial $p_{d} \in \R_d[\vy]$ such that $p_{d} \geq f$ on $\bY$ and
    \begin{equation} \label{eq:onesideApprox}
        \int \left(p_{d} - f \vphantom{\sum}\right)^s \; \od\lambda_\bY \leq \overline{c} \, \omega^{L^s}_{f,0}\left(\lambda,\nicefrac{1}{d}\right)
    \end{equation}
    for some constant $\overline{c}$ depending only on $m$ and $s$.
    
    Moreover, for all $d \in \N$ there is a polynomial $p_d \in \R_d[\vy]$ such that $p_d \geq f$ on $\bY$ and
    \begin{equation} \label{eq:onesideUnifApprox}
        \lambda \left( \left\{\vy \in \bY \; ; \; p_d(\vy) > f(\vy) + \hat{c} \, \omega_{f,0}^{L^\infty}(\bY,\nicefrac{1}{d}) \right\}\right) = 0
    \end{equation}
    for some constant $\hat{c}$ depending only on $m$.
\end{thm}
\end{subequations}

\subsection{Inward-pointing condition (finding $\tilde{w}$)}\label{sec:InwardPointingCond}

In this section, we complement polynomial approximations from the previous section with conditions that assure feasibility for those approximations. We will see in {\bf Lemma \ref{cor:InwStricFeas}} that the following condition is sufficient for guaranteeing the existence of a minimizing sequence of strictly feasible polynomials.

\begin{cond}[Inward-pointing condition]\label{condition:Inpointing} \leavevmode

    We say the LP~\eqref{eq:smp} satisfies the inward-pointing condition if for each feasible point $w$ for \eqref{eq:smp} there exists $\phi \in \cY$ such that
    \begin{equation}
        \cA (w+\theta \phi) - g > 0 \quad \text{on } \bX
    \end{equation}
    for all $\theta \in (0,1]$.
\end{cond}

\begin{lem}\label{cor:InwStricFeas} 
    Under {\bf Condition \ref{condition:Inpointing}}, there exists a minimizing sequence of strictly feasible polynomials.
\end{lem}

\begin{proof}
\begin{subequations}
    Let $\varepsilon > 0$ and $w_\varepsilon \in \cY$ be feasible with
    \begin{equation}\label{eq:weps}
        \langle T,w_\varepsilon\rangle < \od + \frac{\varepsilon}{2}.
    \end{equation}
    Let $\phi_\varepsilon$ be as in the inward-pointing {\bf Condition \ref{condition:Inpointing}}, such that $w_\varepsilon + \theta \phi_\varepsilon$ is strictly feasible for all $\theta \in (0,1]$. Let $\theta_\epsilon \in (0,1]$ be small enough such that
    \begin{equation}\label{eq:thetaeps}
        \left|\langle T,\theta _\epsilon\phi_\varepsilon\rangle \right| = \theta_\epsilon \left|\langle T,\phi_\varepsilon\rangle \right|< \frac{\varepsilon}{2}.
    \end{equation}
    We set $p_\varepsilon := w_\varepsilon + \theta_\varepsilon \phi_\varepsilon \in \cY$. Then $p_\varepsilon$ is strictly feasible and putting together \eqref{eq:weps} and \eqref{eq:thetaeps} gives
    \begin{eqnarray*}
        \langle T,p_\varepsilon\rangle \leq \langle  T,w_\varepsilon\rangle
         + \left|\langle T,\theta _\epsilon\phi_\varepsilon\rangle \right| + 
          < \od + \frac{\varepsilon}{2} + \frac{\varepsilon}{2} = \od + \epsilon. 
    \end{eqnarray*}
    Letting $\varepsilon$ go to zero shows the statement.
\end{subequations}
\end{proof}

Similar to {\bf Lemma~\ref{lem:StricPosImageA}}, a simpler version of {\bf Condition \ref{condition:Inpointing}} is that there exists $\phi \in \overline{\cY}$ with $\cA\phi > 0$.

\begin{lem} \label{lem:positivepointing}
    Assume there exists $\phi \in \cY$ with $\cA\phi > 0$. Then {\bf Condition \ref{condition:Inpointing}} is satisfied.
\end{lem}

\begin{proof}
    Let $w \in \cY$ be feasible. Then for all $\theta >0$ it holds $\cA(w + \theta \phi) - g = \cA(w) - g + \theta \overline{\cA}(\phi) > 0$ on $\bX$.
\end{proof}

The inward-pointing {\bf Condition \ref{condition:Inpointing}} is closely related to the Slater {\bf Condition \ref{condition:slater}}. We address this shortly in the following proposition.

\begin{prop}[Inward-pointing and Slater conditions] \label{lem:Inpointing} \leavevmode

The inward-pointing {\bf Condition \ref{condition:Inpointing}} and the Slater {\bf Condition \ref{condition:slater}} are equivalent.
\end{prop}
\begin{proof}
    If {\bf Condition \ref{condition:Inpointing}} is satisfied, then {\bf Condition \ref{condition:slater}} is satisfied. This follows immediately because for feasible $w \in \cA$ and $\varepsilon$ and $\phi$ as in {\bf Condition \ref{condition:Inpointing}} the point $w + \varepsilon \phi$ is strictly feasible, i.e. (by continuity of $\cA$) lies in the interior of the feasible set. On the other hand, if {\bf Condition \ref{condition:slater}} is satisfied with strictly feasible point $\overset{\circ}{w} \in \cY$ then, by convexity of the feasible set, {\bf Condition \ref{condition:Inpointing}} is satisfied. Indeed, for $w$ feasible let $\phi := \overset{\circ}{w} - w$; then for all $\theta \in (0,1]$ it holds
    \begin{equation*}
        \cA (w +\theta \phi) - g = \cA ((1-\theta)w +\theta \overset{\circ}{w}) - g = (1-\theta) \underbrace{(\cA w - g)}_{\geq 0} + \theta \underbrace{(\cA \overset{\circ}{w}- g)}_{> 0} >0.
    \end{equation*}
\end{proof}

\begin{rem}[On the relevance of {\bf Condition \ref{condition:Inpointing}}] \leavevmode

    Upon reading {\bf Proposition \ref{lem:Inpointing}}, one could wonder why the inward-pointing condition is important, as it is equivalent to the better-known Slater condition. The reason is that this condition {allows to ``quantify'' Slater's condition, in the sense that {by choosing $\theta$} we can control ``how positive'' the function $\cA(w + \theta \phi) - g$ gets.} This will be instrumental in the computation of the convergence rates. {In analogy to Putinar's Positivstellensatz}, the inward-pointing {\bf Condition \ref{condition:Inpointing}} is to Slater's {\bf Condition \ref{condition:slater}} what an effective Putinar Positivstellensatz is to the original Putinar Positivstellensatz {\bf Theorem \ref{thm:psatz}}. {Further, a practical advantage of the inward-pointing condition is that often it can be verified via {\bf Lemma \ref{lem:positivepointing}}.}
\end{rem}

\subsection{Obtaining the convergence rates}

Here, we put together the steps we discussed in this section. We consider the functional LP \eqref{eq:smp} and formulate the following (quantitative) conditions.

\begin{cond}\label{cond:GenMethod} \leavevmode
\end{cond}
\begin{enumerate}
    \item[\mylabel{cond:GenMethod1}{{\bf\ref*{cond:GenMethod}.1}}] Effective Putinar's Positivstellensatz: There is an effective degree bound $\ell_b(n,\vh,p_\bX^\star,\|p\|_\infty^{\bK},\deg(p))$ for Putinar's Positivstellensatz. That is, for $\bX = \bS (\vh) \subset \R^n$ and $p\in \cP(\bX)$, it holds
    \begin{equation}\label{eq:GenDegreeBound}
        p_\bX^\star := \min\{ p(\vx) \; ; \, \vx \in \bX\} > 0 \; \; \text{ and } \;\;\ell \geq \ell_b(n,\vh,p_\bX^\star,\|p\|_\infty^{\bK},\deg(p)) \quad \Longrightarrow p \in \cQ_\ell(\vh)
    \end{equation}
    where $\|p\|_\infty^{\bK} := \sup\limits_{\vx \in \bK} |p(\vx)|$ for a given set $\bK \supset \bX$.
    \item[\mylabel{cond:GenMethod2}{{\bf\ref*{cond:GenMethod}.2}}] Existence and regularity of minimizer: There exists a minimizer $w^\star$ of the functional LP \eqref{eq:smp} (or more generally for the LP~\eqref{eq:LPGeneral}).
    \item[\mylabel{cond:GenMethod3}{{\bf\ref*{cond:GenMethod}.3}}]  Quantitative inward-pointing condition: We have access to quantitative estimates of ``how much the inward-pointing direction is pointing inward''. That is, we can bound from below the function $\psi(\theta)$ given by
    \begin{equation*}
        \psi(\theta) := \inf\limits_{\vx \in \bX} \cA (w^\star+\theta \phi)(\vx) - g (\vx) > 0.
    \end{equation*}
\end{enumerate}

{\bf Condition \ref{cond:GenMethod1}} is formulated such that it is satisfied for the effective Positivstellensatz {\bf Theorem \ref{thm:OldBaldi}}. We will discuss properties of effective Positivstellensätze such as not including an ambient set $\bK \supset \bX$ in {\bf Remark~\ref{rem:AmbientSet}}.

{\bf Conditions \ref{cond:GenMethod2}} and {\bf \ref{cond:GenMethod3}} are specifically formulated to keep simultaneous control of $\deg(p)$, $p^\star_\bX$ and $\|p\|_\infty^\bX$, allowing for simple use of the effective Positivstellensatz. Hence, the ``only'' remaining difficulty lies in effectively verifying those two conditions. In Section~\ref{sec:Limitations}, we discuss possible relaxations of {\bf Conditions \ref{cond:GenMethod2}} to allow for minimizers that are not polynomial.

Examples of how the three concepts in {\bf Conditions \ref{cond:GenMethod}} work together for obtaining convergence rates are demonstrated in the following sections, in which we focus on optimal control problems, volume computation of semialgebraic sets and the central example of polynomial optimization.

\subsection{Example: polynomial optimization}

The Polynomial Optimization Problem (POP) is at the root of the development of the moment-SoS hierarchy. It will serve as a fundamental and motivating example for our convergence rates computation. It consists in \textit{globally} minimizing a polynomial $f \in \R[\vx]$ on a nonempty, compact basic semi-algebraic set $\varnothing \neq \bX := \bS(\vh) \subset \R^n$, where $\vh \in \R[\vx]^r$:
\begin{equation} \label{eq:pop}
\begin{array}{rl}
f^\star_\bX := \min\limits_{\vx \in \R^n} & f(\vx) \\
\st & \vh(\vx) \in \R_+^r.
\end{array}
\end{equation}

\noindent By definition of the minimum, it is straightforward that $f^\star_\bX = \max \{w \in \R ; f-w \geq 0 \text{ on } \bY\}$, which gives rise to the below functional LP formulation
\begin{equation}\label{eq:PopLP}
\begin{array}{lcl}
f^\star_\bX = & \sup\limits_{w \in \R} & w \\
& \st & f- w \in \cC(\bX)_+.
\end{array}
\end{equation}
This optimization problem is of the form~\eqref{eq:smp} after substituting the $\inf$ by $\sup$ for $T := 1$, $\bY := \{0\}$ a one-point set (i.e. $\cP(\bY)$ can be identified with $\R$), $\cA w(\vx) := -w$ and $g := -f$. Now, we perform, step-by-step, our method described in Section~\ref{sec:method}.
\begin{enumerate}
    \item[1.] Find optimal point $w^\star$: The optimal solution of~\eqref{eq:PopLP} given by $w^\star = f_\bX^\star$.
    \item[2.] Verify if $w^\star \in \cP(\bY)$: This is satisfied because $w^\star = f_\bX^\star \in \R = \cP(\bY)$.
    \item[3.] Find inward-pointing direction: We choose the inward-pointing direction $\phi := -1\in \R$. Indeed, for all $\theta \in (0,1]$
    \begin{equation}\label{eq:PopInwardPointing}
        f-(w^\star + \theta \phi) = f-f_\bX^\star + \theta \geq \theta > 0 \qquad \text{on } \bX.
    \end{equation}
    \item[4.-5.] Formulate the SDP hierarchy and bound $\ell$: The moment-SoS hierarchy for~\eqref{eq:PopLP} reads
    \begin{equation} \label{eq:popsos}
        \begin{array}{rl}
        f^\ell_{\bX} := \max\limits_{w\in\R} & w \\
        \st & f - w \in \cQ_\ell(\vh).
        \end{array}
    \end{equation}
    Using~\eqref{eq:PopInwardPointing}, we can apply {\bf Theorem~\ref{thm:OldBaldi}} and get a bound on $\ell$ respectively a convergence rate for $f^\ell_{\bX}$ to $f^\star_\bX$. We state this in {\bf Corollary~\ref{thm:poprate}}, where we use the notation $\|f\| := \max\limits_{\vx \in [-1,1]^n} |f(\vx)|$ from {\bf Theorem~\ref{thm:OldBaldi}}.
\end{enumerate}
    \begin{cor}[Convergence rate for POP {\cite[{\bf Theorems 4.2 \& 4.3}]{baldi2021moment}}] \label{thm:poprate} \leavevmode

For $n\geq 2$ and $\ell \in \N$ one has

{\begin{equation}\label{eq:poprate}
    0\leq f^\star_{\bX} - f^\ell_{\bX} \leq \frac{2\|f\|(\deg (f))^{\frac{7}{5}}\left(\frac{\gamma}{\ell})\right)^{\frac{1}{(2.5n\text{\L})}}}{1-(\deg (f))^{\frac{7}{5}}\left(\frac{\gamma}{\ell})\right)^{\frac{1}{(2.5n\text{\L})}}} \in \mathcal{O}\left(\ell^{-\frac{1}{(2.5n\text{\L})}}\right).
\end{equation}}
\end{cor}

{
\vspace{-3mm}
\begin{proof}
    Let $\ell \in \N$. We search for a range of $\varepsilon \geq 0$ for which we can verify $f^\ell_{\bX} \geq f^\star_{\bX} - \varepsilon$. That is, we want to certify, using \textbf{Theorem \ref{thm:OldBaldi}}, that $p:= f - f^\star_{\bX} + \varepsilon$ belongs to $\cQ_\ell(\vh)$. To apply \textbf{Theorem \ref{thm:OldBaldi}}, we first bound $\|p\|$ and $p^\star_{\bX}$ by
    \begin{equation*}
        \|p\| \leq 2 \|f\| + \varepsilon \quad \text{and} \quad p^\star_{\bX} = \varepsilon.
    \end{equation*}
    From \textbf{Theorem \ref{thm:OldBaldi}} we get $p\in \cQ_\ell(\vh)$ for
    \begin{equation*}
        \ell \geq \gamma(n,\vh) \deg(p)^{3.5n\text{\L}} \left(\nicefrac{\|p\|}{p^\star_\bX}\right)^{2.5n\text{\L}} = \gamma(n,\vh) \deg(f)^{3.5n\text{\L}} \left(\nicefrac{\|p\|}{\varepsilon}\right)^{2.5n\text{\L}}.
    \end{equation*}
    Using $\|p\| \leq 2 \|f\| + \varepsilon$, we search for $\varepsilon = \varepsilon_\ell\geq 0$ such that (after taking the $2.5n \text{\L}$-th root)
    \begin{eqnarray*}
        \ell^{\frac{1}{2.5n\text{\L}}} \geq \gamma(n,\vh)^{\frac{1}{2.5n\text{\L}}} \deg(f)^{\frac{7}{5}} \frac{2\|f\| + \varepsilon}{\varepsilon},
    \end{eqnarray*}
    i.e. we can choose
    \begin{equation*}
        \varepsilon_\ell := \frac{2\|f\|(\deg (f))^{\frac{7}{5}}\left(\frac{\gamma(n,\vh)}{\ell})\right)^{\frac{1}{(2.5n\text{\L})}}}{1-(\deg (f))^{\frac{7}{5}}\left(\frac{\gamma(n,\vh)}{\ell})\right)^{\frac{1}{(2.5n\text{\L})}}}.
    \end{equation*}
    We conclude
    \begin{equation*}
        0\leq f^\star_{\bX} - f^\ell_{\bX} \leq \varepsilon_\ell \in \mathcal{O}\left(\ell^{-\frac{1}{(2.5n\text{\L})}}\right)
    \end{equation*}
\end{proof}
}

It is worth mentioning that the bound in {\bf Corollary~\ref{thm:poprate}} is asymptotic, however, it is possible to have $f^\star_\bX = f^\ell_{\bX}$ already for finite $\ell \in \N$. This is exactly the case when $f - f^\star_\bX \in \cQ(\vh)$.

\begin{rem}
    Even though finite convergence for polynomial optimization is a generic property, see~\cite{nie2014exact}, testing for finite convergence cannot be performed in polynomial time unless $\mathrm{P} = \mathrm{NP}$, see~\cite{vargas2024hardness}. Similarly to the result in \cite{vargas2024hardness}, in~\cite{ahmadi2013complexity} several problems from dynamical systems, that can be approached via functional LPs~\eqref{eq:smp} and the moment-SoS hierarchy~\cite{papachristodoulou2002construction,schlosser2021converging,miller2023bounding}, are shown to be unsolvable in polynomial time unless $\mathrm{P} = \mathrm{NP}$. In practice, we see this as an indicator not to expect finite convergence for the moment-SoS hierarchy for the LPs~\eqref{eq:LPGeneral}. {More precisely,} the applications in \textbf{Sections~\ref{sec:ApplicationDynSys}} and \textbf{\ref{sec:vol}} do not (in general) enjoy finite convergence simply because of the absence of polynomial optimal points.
\end{rem}

\section{Application: dynamical systems}\label{sec:ApplicationDynSys}

We are now going to demonstrate the methodology on instances of the moment-SoS hierarchy related to the study of dynamical systems, such as optimal control~\cite{korda2017convergence} or stochastic differential equations~\cite{henrion2023moment}.

\subsection{Optimal control}\label{sec:ocp}

In this section, we consider the infinite horizon optimal control problem as presented in~\cite{korda2017convergence}:

\begin{align}
V^\star(\vy_0) := \inf\limits_{\vy(\cdot),\vu(\cdot)} & \displaystyle\int_0^\infty \e^{-\beta t} \ g(\vy(t),\vu(t)) \; \od t \notag \\
\st \; \; & \vy(t) = \vy_0 + \displaystyle \int_0^t \vf(\vy(s),\vu(s)) \; \od s \label{eq:ocp} \\
& \vy(t) \in \bY, \quad \vu(t) \in \bU \notag
\end{align}
with discount factor $\beta > 0$, $\vf \in \R[\vy,\vu]^{m}$, $g \in \R[\vy,\vu]$ and compact basic semi-algebraic sets $\bY := \bS(\vh_\bY) \subset \R^{m}$, $\bU := \bS(\vh_\bU) \subset \R^{n_\vu}$, $\vh_\bY \in \R[\vy]^{r_\vy}$, $\vh_\bU \in \R[\vu]^{r_\vu}$.

{A central object in the analysis of optimal control problems is the Hamilton-Jacobi-Bellman inequality
\begin{equation}\label{eq:HJBIneq}
    g - \beta \, V - \vf \cdot \grad \, V \geq 0 \text{ on } \bX := \bY \x \bU.
\end{equation}
For any $V$ satisfying~\eqref{eq:HJBIneq} it holds $V \leq V^\star$ on $\bY$, see for instance \cite{korda2017convergence}.}
Hence, for any probability measure $\mu_0 \in \cM(\bY)_+$ (i.e. s.t. $\mu_0(\bY)=1$) defining a random initial condition $Y_0$\footnote{This includes the deterministic setting under the form $\mu_0 = \delta_{\vy_0}$, where $\vy_0 \in \bY$ and $\delta_{\vy_0}$ is the Dirac measure in $\vy_0$ s.t. for all Borel measurable $\bA \subset \bY$, $\delta_{\vy_0}(\bA) = 1$ if $\vy_0 \in \bA$, $0$ else. Then, $\P(Y_0 = \vy_0) = \mu_0(\{\vy_0\}) = 1$: $Y_0$ is deterministic, and $\E_{\mu_0}[V^\star(Y_0)] := \int V^\star \; \od\mu_0 = V^\star(\vy_0)$.} {it holds}
\begin{subequations}
\begin{equation} \label{eq:hjb}
\begin{array}{rl}
\E_{\mu_0}[V^\star(Y_0)] \geq \sup\limits_{V \in \cP(\bY)} & \displaystyle \int V(\vy_0) \; \od\mu_0(\vy_0) =: \E_{\mu_0}[V(Y_0)] \\
\st \; \; \; \, & g - \beta \, V - \vf \cdot \grad \, V \geq 0 \text{ on } \bX := \bY \x \bU
\end{array}
\end{equation}

{
\begin{rem}[Regularity of $V^\star$]
    {\color{black}Typically, the optimal value function $V^\star$ is not polynomial} and does not even need to be differentiable (an example of an optimal control problem with non-smooth $V^\star$ is given in (\ref{eq:OptContrNonSmooth})). However, to investigate the gap between the left-hand side and the right-hand side in~\eqref{eq:hjb}, it is helpful to assume $V^\star$ being continuously differentiable, as we will do in {\bf Condition \ref{cond:ocp3}}. Such regularity of $V^\star$ implies that $V^\star$ satisfies the Hamilton-Jacobi-Bellman inequality with equality. More precisely it is the unique solution to the so-called Hamilton-Jacobi-Bellman equality. In particular, $V^\star$ becomes feasible for~\eqref{eq:hjb} if we extend the decision space $\cY = \cP(\bY)$ to $\hat{\cY} = C^1(\bY)$ as discussed in \textbf{Remark~\ref{rem:NoMinimizer}}. This observation will be key in our convergence analysis.
\end{rem}}

The optimization problem~\eqref{eq:hjb} is a function LP of the form~\eqref{eq:LPGeneral} for $\cY := \cP(\bY)$ equipped with the induced topology from $C^1(\bY)$, the space $\cX$ given by $\cX := \cC(\bX)$, the linear form $T$ on $\cY$ given by 
\begin{equation*}
    \langle T,w\rangle := \displaystyle\int w \; \od \mu_0,
\end{equation*}
the operator $\cA:\cY \rightarrow \cX$ given by
\begin{equation*}
    \cA \ V := \beta \, V - \vf \cdot \grad \, V,
\end{equation*}
and, by abuse of notation, the function $g$ in~\eqref{eq:smp} given by $-g$ (for $g$ from~\eqref{eq:hjb}). 
The corresponding hierarchy of SoS strengthenings~\eqref{eq:sosd} {for~\eqref{eq:hjb}}, is given formulated below: defining $r := r_\vy + r_\vu$, $\vh := (\vh_\bY,\vh_\bU) \in \R[\vy,\vu]^{r}$, it reads
\begin{equation} \label{eq:hjbsos}
\begin{array}{rl}
\overline{V_{\ell}}(\mu_0) := \sup\limits_{V \in \R_{d_\ell}[\vy]} & \displaystyle \int V(\vy_0) \; \od\mu_0(\vy_0) \\
\st \; \; \; \, & g - \beta \, V - \vf \cdot \grad \, V \in \cQ_\ell(\vh).
\end{array}
\end{equation}
\end{subequations}

{By~\eqref{eq:hjb}, the bound $\overline{V_{\ell}}(\mu_0)$ from~\eqref{eq:hjbsos} is a lower bound on $\E_{\mu_0}[V^\star(Y_0)]$. To assure convergence of $\overline{V_{\ell}}(\mu_0)$ to $\E_{\mu_0}[V^\star(Y_0)]$ as $\ell$ tends to infinity,} we rely on the following Condition (see~\cite[{\bf Assumption 1}]{korda2017convergence}):

\begin{cond} \label{cond:ocp} The following conditions hold:
\begin{enumerate}
    \item[\mylabel{cond:ocp1}{{\bf\ref*{cond:ocp}.1}}] {Setting $\bX := \bY \x \bU$ and $\vx = (\vy,\vu)$, it holds} $\bX \subset \bB = \{\vx \in \R^n \; ; \vx^\top\vx \leq 1\}$ where  $n = m + n_\vu$.
    \item[\mylabel{cond:ocp2}{{\bf\ref*{cond:ocp}.2}}] $\vh(\vO) \in \R_\oplus^r$ (i.e. the interior of $\bX$ contains $\vO$).
    \item[\mylabel{cond:ocp3}{{\bf\ref*{cond:ocp}.3}}] $V^\star \in C^{1,1}(\bY)$, that is $V^\star$ is differentiable and $\grad \, V^\star$ is Lipschitz continuous on $\bY$.
    \item[\mylabel{cond:ocp4}{{\bf\ref*{cond:ocp}.4}}] For all $\vy \in \bY$, the set $\vf(\vy,\bU)$ and the map $\vv \mapsto \inf \{ g(\vy,\vu) \; ; \vu \in \bU \ \text{and} \ \vf(\vy,\vu) = \vv\}$ are convex.
\end{enumerate}
\end{cond}

\noindent
We want to briefly discuss the properties in {\bf Condition~\ref{cond:ocp}}. Using {\bf Condition \ref{cond:ocp1}}, one can complement the description of $\bX = \bS(\vh)$ with the polynomial $$h_{r+1}(\vx) := 1 - \vx^\top\vx,$$ so that {\bf Assumption \ref{asm:poprateOld}} always holds. The second property in {\bf Condition~\ref{cond:ocp}} assures that the set $\bX$ has non-empty interior which loosens feasibility restrictions on the trajectories as well as the control in the interior of $\bX$. {\bf Condition \ref{cond:ocp3}} assures the existence of a minimizer of~\eqref{eq:hjb} -- namely $V^\star$ -- and further we can bound the modulus of continuity of its derivative (which is used when we apply {\bf Theorem~\ref{thm:CkApprox}}). Finally, {\bf Condition \ref{cond:ocp4}} is of technical nature and assures that there is no relaxation gap, i.e. that the optimal value of the right-hand side in~\eqref{eq:hjb} equals, and is not only bounded by, $\E_{\mu_0}[V^\star(Y_0)]$. We refer to~\cite{vinter1993convex} for details.

\begin{thm}[{\cite[{\bf Theorem 1}]{korda2017convergence}}] \label{thm:ocpcv} \leavevmode

Let $\mu_0 \in \cM(\bY)_+$ be a probability measure. Then, under $\mathbf{Condition \, \ref{cond:ocp}}$,
\begin{enumerate}
\item There is $\ell_0 \in \N$ s.t. $\forall \ell \geq \ell_0$, $\overline{V_\ell}(\mu_0) > -\infty$, i.e. \eqref{eq:hjbsos} is feasible.
\item Any $V$ feasible for \eqref{eq:hjbsos} satisfies $V \leq V^\star$ (i.e. it lower bounds the optimal value function).
\item $\forall \ell \geq \ell_0$, \eqref{eq:hjbsos} has an optimal solution $V_{\ell}^\star \in \R_{d_\ell}[\vy]$ s.t. $\overline{V_{\ell}}(\mu_0) = \E_{\mu_0}[V_{\ell}^\star]$.
\item $\E_{\mu_0}[V^\star(Y_0)] - \overline{V_\ell}(\mu_0) = \E_{\mu_0}[V^\star(Y_0) - V_{\ell}^\star(Y_0)] \underset{\ell\to\infty}{\longrightarrow} 0$ i.e. $V_{\ell}^\star$ converges to $V^\star$ in $L^1(\mu_0)$.
\end{enumerate}
\end{thm}
In \cite{korda2017convergence}, the authors give an upper bound on the convergence rate in item \textit{4.} in {\bf Theorem \ref{thm:ocpcv}}. In this section, we improve this bound using the effective Positivstellensatz {\bf Theorem~\ref{thm:OldBaldi}}. Throughout this section, for dynamics $\vf = (f_1,\ldots,f_m):\bX \rightarrow \R^m$, we use the notation
\begin{equation*}
    \|\vf\|_\infty^{\bX} := \sum\limits_{i = 1}^m \|f_i\|_\infty^{\bX}.
\end{equation*}

We first introduce two technical lemmata that will be instrumental in the convergence rate computation and also informative on possible improvements for effective Positivstellensätze.

\begin{lem}\label{lem:PerturbOptimalValueFunction}
    For $d \in \N$ let $V_d \in \R_d[\vy]$ with $\|V_d-V^\star\|_{C^1(\bY)} \leq \frac{c_1}{d}$, where the constant $c_1$ is deduced from {\bf Corollary \ref{cor:CkApprox}}. For any $\eta >0$ let $V_{d,\eta} := V_d - \frac{c_1}{d}\left(1+\frac{\|\vf\|_\infty^{\bX}}{\beta}\right) - \eta \in \R_d[\vy]$. Then $V_{d,\eta}$ satisfies
    \begin{subequations}
    \begin{equation}\label{eq:Vde-V*}
        \|V_{d,\eta} - V^\star\|_{C^1(\bY)} \leq \left(2+\frac{\|\vf\|_\infty^{\bX}}{\beta}\right) \frac{c_1}{d} + \eta
    \end{equation}
    and
    \begin{equation}\label{eq:VdeHJB}
        g - \beta V_{d,\eta} + \vf \cdot \grad \, V_{d,\eta} \geq \beta \eta \; \text{ on } \bX = \bY \times \bU.
    \end{equation}
    \end{subequations}
\end{lem}

\begin{proof}
    The arguments can be found in \cite[Lemma 3]{korda2018convergence} and \cite[Lemma 2]{korda2018convergence}. Since the arguments are short, we will state the proof here as well. First, we compute $c_1$ using {\bf Corollary \ref{cor:CkApprox}}: There exists $V_d \in \R_d[\vy]$ such that
    \begin{align*}
        \|V_d - V^\star\|_{C^1(\bY)} \leq c_0 \, \omega_{V^\star,1}^{L^\infty}(\bY,\nicefrac{1}{d})
    \end{align*}
    with, using the Lipschitz condition on $\grad \, V^\star$ given in {\bf Condition \ref{cond:ocp}}, $\omega_{V^\star,1}^{L^\infty}(\bY,\frac{1}{d}) \propto \frac{1}{d}$, yielding the constant $c_1$ such that $c_0\omega_{V^\star,k}^{L^\infty}(\bY,r) \leq c_1 \frac{1}{d}$.
    Then, we have
    \begin{align*}
        \|V_{d,\eta} - V^\star\|_{C^1(\bY)} & \leq \|V_{d} - V^\star\|_{C^1(\bY)} + \|V_{d,\eta} - V_d\|_{C^1(\bY)} \leq \frac{c_1}{d} + \frac{c_1}{d}\left(1+\frac{\|\vf\|_\infty^{\bX}}{\beta}\right) + \eta.
    \end{align*}
    This is \eqref{eq:Vde-V*}. For \eqref{eq:VdeHJB} we compute
    \begin{align*}
        g-\beta V_{d,\eta} + \vf \cdot \grad \, V_{d,\eta} & = \underset{\geq 0}{\underbrace{g-\beta V^\star + \vf \cdot \grad \, V^\star}} + \beta (V^\star - V_{d,\eta}) + \vf \cdot \grad (V_{d,\eta} - V^\star)\\
        & \geq 0 + \beta\left(V^\star - V_d + \frac{c_1}{d}\left(1+\frac{\|\vf\|_\infty^{\bX}}{\beta}\right) + \eta\right) + \vf \cdot \grad(V_d - V^\star)\\
        & \geq \beta \left(-\| V^\star - V_d\|_\infty^\bY + \frac{c_1}{d}\left(1+\frac{\|\vf\|_\infty^{\bX}}{\beta}\right) + \eta\right) - \|\grad(V^\star - V_d)\|_\infty^\bY \|\vf\|_\infty^{\bX}\\
        & \geq \beta \left(-\| V^\star - V_d\|_{C^1(\bY)} + \frac{c_1}{d}\left(1+\frac{\|\vf\|_\infty^{\bX}}{\beta}\right) + \eta\right) - \|V^\star - V_d\|_{C^1(\bY)} \|\vf\|_\infty^{\bX} \\
        & \geq \beta \left(-\frac{c_1}{d} + \frac{c_1}{d}\left(1+\frac{\|\vf\|_\infty^{\bX}}{\beta}\right) + \eta\right) - \frac{c_1}{d} \|\vf\|_\infty^{\bX}\\
        & = \beta \eta.
    \end{align*}
\end{proof}

\begin{rem}[Inward-pointing condition for smooth control] \label{rem:PerturbOCP} \leavevmode

    Notice that {\bf Lemma \ref{lem:PerturbOptimalValueFunction}} is essentially a statement of satisfaction of the inward-pointing {\bf Condition \ref{condition:Inpointing}} while keeping control on the objective, in the special case of the optimal control problem with Lipschitz differentiable optimal value function.
\end{rem}

\begin{lem}\label{lem:MatteoBound}

    For any nonnegative polynomial $p \in \cP(\bX)_+$, it holds
    \begin{equation} \label{eq:BOUND}
        \|p\| \leq \left(1 + \frac{\deg(p)^2}{4}\left(\nicefrac{2}{b}\right)^{\deg(p)+1}\right) \|p\|_\infty^\bX,
    \end{equation}
    where $b \in (0,1)$ is such that $[-b,b]^n \subset \bX$ (whose existence is guaranteed by {\bf Condition \ref{cond:ocp}.2}).
\end{lem}
    \begin{proof}
    See {\bf Appendix \ref{appendix:MattBound}}.
    \end{proof}

\begin{rem}
Notice that {\bf Lemma~\ref{lem:MatteoBound}} introduces a gap between the norm {$\|p\|$} that we need to use in the effective Positivstellensatz and the norm {$\|p\|_\infty^\bX$} that is available in our setting (this is discussed in more detail in {\bf Remark~\ref{rem:AmbientSet}}).
\end{rem}  

Now we are already in the position to apply an effective version of Putinar's Positivstellensatz.

\begin{thm}[Effective Putinar Positivstellensatz for optimal control]\label{prop:PropOCPepsilon} \leavevmode

    Let $\eta > 0$ and $d \in \N$ such that $d_\vf := \deg(f)+d \geq \deg (g)$. {Let $V_{d,\eta}$ be defined as in {\bf Lemma \ref{lem:PerturbOptimalValueFunction}}}. Under {\bf Condition \ref{cond:ocp}}, there exist $A,B,C \in \R_\oplus$ so that $V_{d,\eta}$ is feasible for \eqref{eq:hjbsos} for any
    \begin{subequations} \label{eq:OCPepsilon}
        \begin{equation} \label{eq:PropOCPQM}
            \ell \geq \gamma d_\vf^{3.5 n \text{\L}} \left(\frac{A}{\eta} + \frac{B}{\eta \, d} + 1\right)^{2.5 n \text{\L}}\left(1 + C^{d_\vf+1}\cdot \nicefrac{d_\vf^2}{4}\right)^{2.5 n \text{\L}}.
        \end{equation}
        Moreover,
        it holds
        \begin{equation} \label{eq:PropOCPepsilon}
            \E_{\mu_0}[V_{d,\eta}(Y_0)] \geq \E_{\mu_0}[V^\star(Y_0)] - \eta - \left(2 + \frac{\|\vf\|_\infty^\bX}{\beta} \right) \frac{c_1}{d}.
        \end{equation}
    \end{subequations}
\end{thm}
\begin{proof}
    \begin{subequations}
    We use {\bf Theorem \ref{thm:OldBaldi}}: Denoting $p := g - \beta\,V_{d,\eta} - \vf\cdot\grad\,V_{d,\eta} \geq \beta \, \eta > 0$, we know that $p \in \cQ_\ell(\vh)$ (i.e. $V_{d,\eta}$ is feasible for \eqref{eq:hjbsos}) for
    \begin{equation} \label{eq:ellbound}
        \ell \geq \gamma \deg(p)^{3.5n\text{\L}} \left(\nicefrac{\|p\|}{p^\star_\bX}\right)^{2.5n\text{\L}}
    \end{equation}
    {for $\gamma = \gamma(n,\vh)$ the constant from {\bf Theorem \ref{thm:OldBaldi}}}. Hence, we only need to estimate bounds on $\|p\|$, $\deg(p)$ and $p_{\bX}^\star$.
    \begin{itemize}
        \item $p_{\bX}^\star \geq \beta \, \eta$ is a direct consequence of how we constructed $V_{d,\eta}$ in {\bf Lemma \ref{lem:PerturbOptimalValueFunction}}.
        \item $\deg(p) \leq \max(\deg(g), \deg(V_{d,\eta}), \deg(\vf\cdot\grad \, V_{d,\eta})) \leq d_\vf := \deg(\vf) + d$.
        \item Estimating $\|p\|$ is the difficult part. First, we have, denoting $q := g - \beta \, V^\star - \vf\cdot\grad \, V^\star \geq 0$:
    \end{itemize}
    \begin{align*}
        \|p\|_\infty^\bX - \|q\|_\infty^\bX & \leq \|p - q\|_\infty^\bX \\
        & = \|\beta(V^\star - V_{d,\eta}) + \vf\cdot\grad(V^\star - V_{d,\eta})\|_\infty^\bX \\
        & = \left\|\beta\left(V^\star - V_{d} + \frac{c_1}{d}\left(1 + \frac{\|\vf\|_\infty^\bX}{\beta}\right) + \eta\right) + \vf\cdot\grad(V^\star - V_{d})\right\|_\infty^\bX \\
        & \leq \beta \left(\|V^\star - V_d\|_\infty^\bX + \frac{c_1}{d}\left(1 + \frac{\|\vf\|_\infty^\bX}{\beta}\right) + \eta \right) + \|\vf\|_\infty^\bX \, \|\grad(V^\star - V_d)\|_\infty^\bX \\
        & \leq (\beta + \|\vf\|_\infty^\bX)\left(\|V^\star - V_d\|_{C^1(\bY)} + \frac{c_1}{d}\right) + \eta \, \beta \\
        & \leq 2 (\beta + \|\vf\|_\infty^\bX) \frac{c_1}{d} + \eta \, \beta
    \end{align*}
    which gives the upper bound
    \begin{equation} \label{eq:upboundp}
        \|p\|_\infty^\bX \leq \|q\|_\infty^\bX + 2 (\beta + \|\vf\|_\infty^\bX) \frac{c_1}{d} + \eta \, \beta.
    \end{equation}
    However, {in~\eqref{eq:ellbound} we want to bound in the the term $\frac{\norm{p}}{p^\star_{\bX}}$. We do so using    \begin{equation*}
    \frac{\norm{p}}{p^\star_{\bX}} = \left(\frac{\norm{p}_{\infty}^{\bX}}{p^\star_{\bX}}\right) \left(\frac{\norm{p}}{\norm{p}_{\infty}^{\bX}} \right),
\end{equation*}
    where we bound the first term on the right-hand side by~\eqref{eq:upboundp} and $p_\bX^\star \geq \beta \eta$, and for the second term we use {\bf Lemma \ref{eq:BOUND}.}
    }
\end{subequations}
    We finally get the claimed bound by putting together in~\eqref{eq:ellbound}:
    \begin{center}
    \scalebox{1}{$\ell \geq \gamma d_\vf^{3.5 n \text{\L}} \underset{ \geq \nicefrac{\|p\|_\infty^\bX}{p^\star_\bX}}{\underbrace{\left(\frac{\|q\|_\infty^\bX}{\beta \, \eta} + 2 \frac{\beta + \|\vf\|_\infty^\bX}{\beta \, \eta} \frac{c_1}{d} + 1\right)}}^{2.5 n \text{\L}}\underset{\geq \nicefrac{\|p\|}{\|p\|_\infty^{\bX}}}{\underbrace{\left(1 + \frac{d_\vf^2}{4}\left(\nicefrac{2}{b}\right)^{d_\vf+1}\right)}}^{2.5 n \text{\L}}$}
    \end{center}
    which is exactly \eqref{eq:PropOCPQM}. Eventually, we compute
    \begin{align*}
        0 & \leq \E_{\mu_0}[V^\star(Y_0)] - \overline{V_{\ell}}(\mu_0)  \leq \E_{\mu_0}[V^\star(Y_0) - V_{d,\eta}(Y_0)] \leq \|V^\star - V_{d,\eta}\|_\infty^\bY\\
        & \leq \|V^\star - V_{d,\eta}\|_{C^1(\bY)} \stackrel{\eqref{eq:Vde-V*}}{\leq} \left(2 + \frac{\|\vf\|_\infty^\bX}{\beta}\right) \frac{c_1}{d} + \eta.
    \end{align*}
\end{proof}

\begin{cor}[Convergence rate for optimal control]\label{prop:PropOCPConvRate} \leavevmode

Under {\bf Condition \ref{cond:ocp}}, for $\ell \in \N$ large enough it holds
    \begin{equation}
        |\E_{\mu_0}[V^\star(Y_0)] - \overline{V_{\ell}}(\mu_0)| \in \cO(\nicefrac{1}{\log \ell}) \; \; \text{ as } \ell \rightarrow \infty.
    \end{equation}
\end{cor}
\begin{proof}
    \begin{subequations}
    We are going to use {\bf Theorem \ref{prop:PropOCPepsilon}}. 
    Let $d \in \N$ and take
    $$ \eta := \frac{1}{d}.$$
    Then, {\bf Theorem \ref{prop:PropOCPepsilon}} ensures that
    \begin{equation}
        0 \leq \E_{\mu_0}[V^\star(Y_0)] - \overline{V_{\ell}}(\mu_0) \leq \frac{1}{d}\left( 1 + \left(2 + \frac{\|\vf\|_\infty^\bX}{\beta} \right) c_1 \right) {=:\varepsilon_d} \in \mathcal{O}\left(\frac{1}{d}\right)    
    \end{equation}
    for
    $$ \ell \geq \gamma d_\vf^{3.5 n \text{\L{}}} \left(A \, d + B + 1\right)^{2.5 n \text{\L{}}} (1 + C^{d_\vf+1}\cdot\nicefrac{d_\vf^2}{4})^{2.5 n \text{\L{}}} =: \ell_b$$
    for certain $A,B,C >0$. Now, an asymptotic equivalent is given by
    $$ \ell_b \underset{d\to\infty}{\sim} \gamma \frac{A^{2.5n \text{\L{}}}}{4^{2.5n \text{\L}}} d^{11 n \text{\L{}}} C^{(d+\deg(\vf)+1)2.5 n \text{\L{}}} \in \underset{d\to\infty}{\cO} \left(d^{11 n \text{\L{}}} C^{2.5 n \text{\L{}} d}\right) $$
    and taking the $\log$ yields
    $$ \log(\ell_b) \in \cO(d) = \cO(\nicefrac{1}{\epsilon_d}) \; \; \; \text{ as } d\to\infty. $$
    Finally, remembering that $0 \leq \E_{\mu_0}[V^\star(Y_0)] - \overline{V_{\lceil\ell_b\rceil}}(\mu_0) \leq \varepsilon_d$ and inverting the above asymptotic expression gives the announced result.
    \end{subequations}
\end{proof}

\begin{rem}[Comparison with \cite{korda2017convergence}] \leavevmode

    In \cite{korda2017convergence}, using a previous effective version of Putinar's Positivstellensatz, the authors came up with a much worse convergence rate of $\nicefrac{1}{\log \log \ell}$. Using the effective Putinar Positivstellensatz from \cite{baldi2021moment}, we could remove an exponential dependence in the degree and hence sharply improve the convergence rate. Moreover, at the price of some additional assumptions on the state set $\bY$, it is even possible to derive a \textit{polynomial} convergence rate, as we will show now. Indeed, the remaining exponential dependence is an artifact coming from {\bf Lemma \ref{lem:MatteoBound}}, but the effective Positivstellensatz gives polynomial dependence.
\end{rem}

\subsection{A polynomial convergence rate}

In this paragraph, we derive a convergence rate for the GMP \eqref{eq:hjbsos} in which the level $\ell$ of the hierarchy is polynomial in $\frac{1}{\varepsilon}$ for the relaxation gap $\E_{\mu_0}[V^\star(Y_0)] - \overline{V_{\ell}}(\mu_0) \leq \varepsilon$. The idea is to side-step the exponential growth of the degree bound in {\bf Lemma \ref{lem:MatteoBound}} that arises from bounding the supremum of a polynomial on the hypercube by its supremum on a smaller cube. Here, we will extend the optimal value function $V^\star$ to the whole hypercube $[-1,1]^m$ and, only then, approximate it on the full hypercube $[-1,1]^m$ by a polynomial $V_d$. This allows us to bound $\|V_d\|$ simply by $\|V^\star\| + 1$ (for $d$ large enough) instead of $\|V_d\| \leq C^d\|V_d\|_\infty^\bY$. For this to work we need to guarantee that there exists an extension of $V^\star$ to $[-1,1]^m$ with sufficient regularity.

To extend $V^\star$ we need to introduce the Hölder spaces and norms: We say a function $w$ belongs to the Hölder space $C^{k,a}(\bY)$ for $k\in \N$ and $a \in (0,1]$ if $w\in C^k(\bY)$ and its $k$-th derivative is $a$-Hölder-continuous, i.e. its $a$-Hölder coefficient is finite:
\begin{equation*}
    \zeta_{k,a}^\bY(w) := \max\limits_{|\balpha| = k}\sup\limits_{\vy \neq \vy' \in \bY} \frac{|\partial_{\balpha} w(\vy) - \partial_{\balpha} w(\vy')|}{\|\vy - \vy'\|^a} < \infty.
\end{equation*}
For bounded $\bY$, we equip the space $C^{k,a}(\bY)$ with the norm
\begin{equation*}
    \|w\|_{C^{k,a}(\bY)} := \|w\|_{C^k(\bY)} + \zeta_{k,a}^\bY(w).
\end{equation*}

The notion of Hölder regularity is used to state the following condition, which is instrumental in ensuring a polynomial convergence rate instead of the logarithmic one given in {\bf Corollary \ref{prop:PropOCPConvRate}}.

\begin{cond}\label{cond:ocpBoundaryregularity}
    The set $\bY$ has $C^{1,1}$ boundary, that is the boundary $\partial \bY$ is locally the graph of a $C^{1,1}$ function in the above sense of having a finite Hölder coefficient.
\end{cond}

Next, we provide an extension result for Hölder functions from \cite{gilbarg1977elliptic}.

\begin{lem}[Extension Lemma; {\cite[Lemma 6.37]{gilbarg1977elliptic}}]\label{thm:HölderExtension} \leavevmode

 Let $k \geq 1$ be an integer and $a \in (0,1]$. Let $\bY \subset \R^m$ be compact with $C^{k,a}$ boundary. Let $\bOm$ be an open and bounded set containing $\bY$. Then for every function $w \in C^{k,a}(\bY)$ there exists an extension $\Bar{w} \in C^{k,a}(\overline{\bOm})$ with $w(\vy) = \Bar{w}(\vy)$ for all $\vy \in \bY$ and
\begin{equation} \label{eq:extlemma}
    \|\Bar{w}\|_{C^{k,a}(\overline{\bOm})} \leq c_2 \, \|w\|_{C^{k,a}(\bY)}
\end{equation}
for some constant $c_2 = c_2(m,k,a,\bY,\bOm)$ independent of $w$.  
\end{lem}

Under {\bf Conditions \ref{cond:ocpBoundaryregularity} and \ref{cond:ocp}}.3, {\bf Lemma \ref{thm:HölderExtension}} ensures that there exists an extension $V\in C^{1}([-1,1]^m)$ of $V^\star$ such that $\grad \ V$ is Lipschitz continuous and there exists a constant $c_2 = c_2(m,\bY)$ such that
\begin{subequations}
\begin{equation}\label{eq:VboundC1}
    \|V\|_{C^{1,1}([-1,1]^m)} \leq c_2 \|V^\star\|_{C^{1,1}(\bY)}.
\end{equation}

For the rest of this paragraph, we follow the same path as previously in this section. That is, by {\bf Corollary \ref{cor:CkApprox}}, let $V_d\in \R_d[\vy]$ be a polynomial and $c_1$ be a constant (independent of $V$) with 
\begin{equation} \label{eq:VVd}
    \|V-V_d\|_{C^1([-1,1]^m)} \leq \frac{c_1}{d}.
\end{equation}
As in {\bf Lemma \ref{lem:PerturbOptimalValueFunction}}, for $\eta > 0$, we define
\begin{equation}\label{eq:DefVdeta}
    V_{d,\eta} := V_d - \frac{c_1}{d}\left(1+\frac{\|\vf\|}{\beta}\right) - \eta \in \R_d[\vy],
\end{equation}
\end{subequations}
where we recall that $\|\vf\| = \max\{|\vf(\vy,\vu)| \; ; \, (\vy,\vu) \in [-1,1]^n \}$ with $n = m + n_\vu$. In the following lemma, we show that $V_{d,\eta}$ is strictly feasible.

\begin{lem}\label{lem:ocpExtBound}
    Let {\bf Condition \ref{cond:ocpBoundaryregularity}} hold and $c_1 = c_1(m,\bY,V^\star)$, $c_2 = c_2(m,\bY)$ be the constants from \eqref{eq:VboundC1} and \eqref{eq:VVd}. For $d\in \N$, $\eta >0$ the function $V_{d,\eta}$ satisfies
    \begin{subequations}
    \begin{equation}\label{eq:PolyOCPVde-V*}
        \|V_{d,\eta} - V^\star\|_{C^1(\bY)} \leq \left(2+\frac{\|\vf\|}{\beta}\right) \frac{c_1}{d} + \eta.
    \end{equation}
    Further, for the polynomial function $p := g - \beta V_{d,\eta} + \vf \cdot \grad \, V_{d,\eta} \in \cP(\bX)$ it holds
    \begin{equation}\label{eq:rhodboundbeta}
        p \geq \beta \eta \; \text{ on } \bX = \bY \times \bU
    \end{equation}
    and, recalling that $\|p\| = \max\{|p(\vx)| \; ; \, \vx \in [-1,1]^n\}$,
    \begin{equation}\label{eq:VdeHJBUpperBound}
        \|p\| \leq \|g\| + c_2 \, \|V^\star\|_{C^{1,1}(\bY)} \left( \beta  + \|\vf\| \right) + 2\frac{c_1}{d} \left(\beta+ \|\vf\| \right) + \beta \eta.
    \end{equation}
     \end{subequations}
\end{lem}

\begin{proof}
\begin{subequations}
    The statements \eqref{eq:PolyOCPVde-V*} and \eqref{eq:rhodboundbeta} follow similarly to \eqref{eq:Vde-V*} and \eqref{eq:VdeHJB} in {\bf Lemma \ref{lem:PerturbOptimalValueFunction}}.
    To show \eqref{eq:VdeHJBUpperBound}, we simply apply the triangle inequality
    \begin{eqnarray}\label{eq:rhoTriangleInequ}
        \|p\| & = & \|g - \beta V_{d,\eta} + \vf \cdot \grad \, V_{d,\eta}\|\notag\\
        & \leq & \|g\| + \beta \|V_{d,\eta}\| + \|\vf \cdot \grad \, V_{d,\eta}\|
    \end{eqnarray}
    and separately bound $\|\beta V_{d,\eta}\|$ and $\|\vf \cdot \grad \, V_{d,\eta}\|$. We begin with $\beta \|V_{d,\eta}\|$
    \begin{eqnarray*}
        \|\beta V_{d,\eta}\| & \leq & \beta\left(\|V\| + \|V_d-V\| + \frac{c_1}{d}\left(1+\frac{\|\vf\|}{\beta}\right) + \eta\right)\\
        & \leq & \beta\left(c_2 \, \|V^\star\|_{C^{1,1}(\bY)} + \frac{c_1}{d} + \frac{c_1}{d}\left(1+\frac{\|\vf\|}{\beta}\right) + \eta\right)\\
        & \leq & \beta \left(c_2 \, \|V^\star\|_{C^{1,1}(\bY)} + \frac{c_1}{d} \left(2 + \frac{\|\vf\|}{\beta}\right) + \eta \right)
    \end{eqnarray*}
    where in the last inequality we used \eqref{eq:VboundC1} and \eqref{eq:VVd}. For bounding $\vf \, \cdot \grad \, V_{d,\eta}$ we use $\grad \, V_{d,\eta}= \grad \, V_d$ and we have
    \begin{eqnarray*}
        \|\vf \, \cdot \grad \, V_{d,\eta}\| & = & \|\vf \, \cdot \grad \, V_{d}\| \\
        &\leq & \|\vf\| \left(\|V\|_{C^{1,1}(\bY)} + \|V_d-V\|_{C^{1,1}(\bY)}\right)\\
        & \leq & \|\vf\| \left(c_2 \, \|V^\star\|_{C^{1,1}(\bY)} + \frac{c_1}{d}\right)
    \end{eqnarray*}
    Putting together in \eqref{eq:rhoTriangleInequ} gives \eqref{eq:VdeHJBUpperBound}.
    \end{subequations}
\end{proof}

As in the previous section, {\bf Lemma \ref{lem:ocpExtBound}} ensures that $V_{d,\eta}$ is an inward-pointing perturbation of $V^\star$. Now, all that remains is to apply an effective version of Putinar's Positivstellensatz from {\bf Theorem \ref{thm:OldBaldi}}.

\begin{thm}[Polynomial rate for optimal control] \label{prop:ocpConvRatePoly} \leavevmode

    Under {\bf Conditions \ref{cond:ocp} and \ref{cond:ocpBoundaryregularity}}, for $\ell \in \N$ large enough it holds
    \begin{equation}
        0 \leq \E_{\mu_0}[V^\star(Y_0)] - \overline{V_{\ell}}(\mu_0) \in \cO \left(\ell^{-\frac{1}{6n\text{\L}}}\right) \; \; \; \text{ as } \ell \to\infty.
    \end{equation}
\end{thm}

\begin{proof}\begin{subequations}

    We use the notation and constants from {\bf Lemma \ref{lem:ocpExtBound}}. Let $V_{d,\eta}$ be as in \eqref{eq:DefVdeta}. For $d\in \N$ we choose $\eta = \eta_d := \frac{1}{d}$. Let $d_0 \in \N$ such that
    \begin{equation}\label{eq:ocpd0}
        \left(2+\frac{\|\vf\|}{\beta}\right) \frac{c_1}{d_0} + \eta_{d_0} \quad , \quad  2\frac{c_1}{d_0} \left(\beta+ \|\vf\| \right) + \beta \eta_{d_0} \leq 1.
    \end{equation}
    By monotonicity in $d$, both terms on the left-hand side in \eqref{eq:ocpd0} are bounded by $1$ for all $d\geq d_0$. From {\bf Lemma \ref{lem:ocpExtBound}},
    we have
    \begin{equation}\label{eq:VdetaEps}
        \|V_{d,\eta} - V^\star\|_{C^1(\bY)} \in \cO\left(\frac{1}{d}\right) \; \; \; \text{ as } d\to\infty
    \end{equation}
    and for $p := g - \beta V_{d,\eta} + \vf \cdot \grad \, V_{d,\eta}$ it holds
    \begin{equation}\label{eq:Vdebeta}
        p \geq \frac{\beta}{d} \; \text{ on } \bY \times \bU
    \end{equation}
    and
    \begin{equation}\label{eq:VdeHJBBound}
        \|p\| \leq \|g\| + c\|V^\star\|_{C^{1,1}(\bY)} \left( \beta  + \|\vf\| \right) + 1 =: c_3.
    \end{equation}
    Note that the constant $c_3$ is independent of $d$ and the choice of extension $V$. Inserting \eqref{eq:Vdebeta} and \eqref{eq:VdeHJBBound} into {\bf Theorem \ref{thm:OldBaldi}}, we get that $V_{d,\eta_d}$ is feasible for \eqref{eq:hjbsos} for any $\ell \in \N$ with
    \begin{equation}\label{eq:orderl}
        \ell \geq \gamma(n,\vh) \deg(p)^{3.5n\text{\L}} \left(\frac{c_3d}{\beta}\right)^{2.5n\text{\L}}.
    \end{equation}
    To finalize the proof, recall that $\deg(p) \leq d + \deg(\vf) \in \cO(d)$. Thus, for given $\ell \in \N$ (large enough), we choose the largest $d = d_\ell\in \N$ (with $d\geq d_0$) such that \eqref{eq:orderl} is satisfied. By \eqref{eq:orderl}, such $d_\ell$ is of order $\ell^{\frac{1}{6n\text{\L}}}$ and we get
    \begin{eqnarray*}
        |\E_{\mu_0}[V^\star(Y_0)] - \overline{V_{\ell}}(\mu_0)| & \leq & \displaystyle \int |V^\star(\vy_0) - V_{d_\ell,\eta_{d_\ell}}(\vy_0)| \; \od\mu_0(\vy_0) \\
        & \leq & \|V_{d_\ell,\eta_{d_\ell}} - V^\star\|_\infty^\bY\\
        & = & \|V_{d_\ell,\eta_{d_\ell}} - V^\star\|_{C^1(\bY)} \in \cO\left(\frac{1}{d_\ell}\right) \in \cO\left( \ell^{-\frac{1}{6n\text{\L}}}\right) \; \; \; \text{ as } \ell\to\infty.
    \end{eqnarray*}
    This shows the statement.   
\end{subequations}  
\end{proof}

\begin{rem}[Relaxing the regularity assumption on $V^\star$] \leavevmode

    The same arguments in the proof of {\bf Theorem \ref{prop:ocpConvRatePoly}} work still for $V^\star$ with slightly less regularity, namely, for  $V^\star\in C^{1,a}(\bY)$ and $\bY$ with $C^{1,a}$ boundary for some $a \in (0,1)$. The convergence rate then takes the form
    \begin{equation*}
        \E_{\mu_0}[V^\star(Y_0)] - \overline{V_{\ell}}(\mu_0) \in \cO \left(\ell^{-\frac{1}{2.5n\text{\L} + 3.5n\nicefrac{\text{\L}}{a}}}\right) \; \; \; \text{ as } \ell\to\infty
    \end{equation*}
\end{rem}

\subsection{Exit location of stochastic processes}\label{sec:ExitLocation}

In this example, we apply our framework to \cite{henrion2023moment}, in which the exit location of stochastic processes is computed by a function LP~\eqref{eq:smp}. We recall the setting from \cite{henrion2023moment}. Consider a stochastic differential equation
\begin{equation}\label{eq:StochasticODE}
   dX_t = \vf_0(X_t) \; dt + \mF(X_t) \; dB_t, \; \; X_0 = \vx_0
\end{equation}
for $\vf_0 = (f_{0i})_i :\R^n \rightarrow \R^n$, $\mF = (f_{ij})_{i,j} :\R^n \rightarrow \R^{n \times n_2}$, a deterministic initial condition $\vx_0$ and $(B_t)_{t \geq 0}$ a $n_2$-dimensional Brownian motion. The SDE (\ref{eq:StochasticODE}) is equipped with an open, bounded constraint set $\bX \subset \R^n$ and for a given function $g:\partial \bX \rightarrow \R$, the expected exit value for $\vx \in \bX$ is given by
\begin{equation}\label{eq:exitvalue}
    v^\star(\vx_0) := \mathbb{E}(g(X_{\tau}))
\end{equation}
where $\tau = \inf \{t\geq 0 \; ; \, X_t \in \partial\bX\}$ is the first time at which the process $(X_t)_t$ starting at $X_0 = \vx_0$ hits $\partial \bX$.

In \cite{henrion2023moment}, the following assumptions were made
\begin{cond}\label{cond:SDEDidier} \leavevmode
\begin{enumerate}
    \item[{\bf \ref*{cond:SDEDidier}.1}] It holds $\overline{\bX} = \bS (\vh) \subset \bK = [-1,1]^n$ for some $\vh \in \R[\vx]^r$.
    \item[{\bf \ref*{cond:SDEDidier}.2}] The boundary $\partial\bX$ is smooth and is represented by $\partial\bX = \bS (\vh_\partial)$ for some $\vh_\partial \in \R[\vx]^{r_\partial}$.
    \item[{\bf \ref*{cond:SDEDidier}.3}] We assume $g,f_{0i},f_{ij} \in \R[\vx]$ for $i = 1,\ldots,m$ and $j = 1,\ldots,n$.
    \item[{\bf \ref*{cond:SDEDidier}.4}] The matrix $\mF(\vx) \mF(\vx)^\top$ is positive definite for all $\vx\in \overline{\bX}$.
\end{enumerate}
\end{cond}

\begin{rem}
    The case when the boundary $\partial \bX$ decomposes into several disjoint components $\partial \bX = \bX^\partial_1 \cup \ldots \cup \bX^\partial_l$ is treated similarly in \cite{henrion2023moment}. By \textbf{Condition \ref{cond:SDEDidier}.2}, we restrict to the notationally simpler case of $\partial\bX = \bS (\vh_\partial)$.
\end{rem}

Under the above conditions there exists a unique solution $X_t$ of (\ref{eq:StochasticODE}) for $t\leq \tau$, see \cite{henrion2023moment,evans2012introduction}.

In \cite{henrion2023moment}, from Dynkin's formula, the following functional LP for the value $v^\star(\vx_0)$ of the exit value (\ref{eq:exitvalue}) is derived
\begin{equation}\label{eq:LPDidierSDE}
    v^\star(\vx_0) = \max\limits_{v{\in C^2(\bX)}} v(\vx_0) \quad \text{s.t.} \quad \cL v\leq 0 \text{ on } \; \bX, \; \; \; v\leq g \text{ on } \;\partial \bX
\end{equation}
with moment-SoS hierarchy, for $\ell \in \N$,
\begin{equation}\label{eq:HierarchyLPDidier}
    \begin{tabular}{ccc}
        $v^\star_\ell(\vx_0) :=$ & $\sup\limits_{v \in \R[\vx]}$ & $v(\vx_0)$\\
         & s.t. & $-\cL v \in \cQ_\ell(\vh)$\\
         & & $g-v \in \cQ_\ell(\vh_\partial)$,
    \end{tabular}
\end{equation}
where $\cL$ is the second-order partial differential operator
\begin{equation}\label{eq:DefLGenerator}
    \cL v (\vx) := -\frac{1}{2}\sum\limits_{i,j = 1}^n a_{ij}(\vx) \frac{\partial^2 v}{\partial\vx_i\partial\vx_j}(\vx) + \sum\limits_{i = 1}^n f_{0i} (\vx) \frac{\partial v}{\partial \vx_i}(\vx)
\end{equation}
for $(a_{ij}(\vx))_{i,j = 1,\ldots,n} = \mF(\vx) \mF(\vx)^\top$.

\begin{rem}\label{rem:ExtensionSDEExit}
    In the spirit of \textbf{Remark~\ref{rem:NoMinimizer}}, the LP~\eqref{eq:LPDidierSDE} {is a relaxation of the following LP
    \begin{equation*}
        v^\star(\vx_0) = \sup\limits_{v\in\cP(\bX)} v(\vx_0) \quad \text{s.t.} \quad \cL v\leq 0 \text{ on } \; \bX, \; \; \; v\leq g \text{ on } \;\partial \bX.
    \end{equation*}}
    {The above LP is a function LP~\eqref{eq:LPGeneral}} for $\cY = \cP(\bX)$ with topology induced by $C^2(\bX)$, the sets $\bX_i$ for $i = 1,2$ given by ${\bY = \color{black}\bX_1} := \bX$ and $\bX_2 :=\partial \bX$, the linear form $T$ given by $\langle T,w\rangle := w(\vx_0)$, the linear operator $\cA$ given by $\cA \, w:= (-\cL w, -w\big|_{\partial \bX})$, and $\vg(\vx) := (0,g(\vx))$ for all $\vx \in \bX$. {The LP~\eqref{eq:LPDidierSDE} results, as in \textbf{Remark~\ref{rem:NoMinimizer}}, as a relaxation by choosing $\hat{\cY} = C^2(\bX)$ (still equipped with the induced topology from $C^2(\bX)$).} As we discuss below, the relaxed LP~\eqref{eq:LPDidierSDE} has $v^\star$ as a minimizer.
\end{rem}

The function $v^\star$ from~\eqref{eq:LPDidierSDE} is the solution of the following boundary value problem
\begin{eqnarray}\label{eq:Solv*}
\begin{tabular}{cccll}
     $\cL v$ & $=$ & $0$ & \text{on} & $\bX$\\ 
     $v$ & $=$ & $g$ & \text{on} & $\partial\bX$.
\end{tabular}
\end{eqnarray}
Thus, the question about existence and regularity of minimizers of (\ref{eq:LPDidierSDE}) is transferred to the question about existence and regularity of solutions of (\ref{eq:Solv*}). Fortunately, the answer here is positive, see \cite{gilbarg1977elliptic,urbas1996lecture,henrion2023moment}. Namely, under {\bf Condition \ref{cond:SDEDidier}}, there exists a unique solution $v \in C^{\infty}(\bX)$ of (\ref{eq:Solv*}). Next, we investigate an inward-pointing direction. Therefore, we recall that for $\phi \in C^{\infty}(\bX)$ there exists a unique solution $u_\phi \in C^{\infty}(\overline{\bX})$ of
\begin{eqnarray}\label{eq:PoissionProbSol}
\begin{tabular}{cccll}
     $\cL u_\phi$ & $=$ & $\phi$ & \text{in} & $\bX$\\ 
     $u_\phi$ & $=$ & $0$ & \text{on} & $\partial\bX$.
\end{tabular}
\end{eqnarray}
To construct an inward-pointing direction let $\phi(\vx) := -1$ for all $\vx \in \overline{\bX}$ and $u_\phi$ be a corresponding solution to (\ref{eq:PoissionProbSol}). Let $0 < \eta,\theta \in \R$, we define the function
\begin{equation}\label{eq:DefInwardDirectionSDE}
    v := v^\star + \theta (u_\phi - \eta).    
\end{equation}
In the following lemma, we show that $v$ is indeed strictly feasible for all $\theta > 0$; in other words $u_\phi - \eta$ is an inward-pointing direction for $v^\star$.

\begin{lem}
    For all $\theta > 0$, the function $v$ from (\ref{eq:DefInwardDirectionSDE}) is strictly feasible for (\ref{eq:LPDidierSDE}).    
\end{lem}

\begin{proof}
    On $\bX$ it holds
    \begin{eqnarray*}
        \cL v & = & \cL  (v^\star + \theta (u_\phi - \eta)) = \cL v^\star + \theta \cL (u_\phi - \eta) = \cL v^\star + \theta \cL  u_\phi = 0 +\theta (-1) = - \theta < 0. 
    \end{eqnarray*}
    For $\vx \in \partial \bX$ we have
    \begin{equation*}
        v(\vx) = v^\star(\vx) + \theta (u_\phi(\vx) - \eta) = g(\vx) + \theta (0 - \eta) = g(\vx) - \theta \eta < g(\vx).
    \end{equation*}
    This shows that $v$ is strictly feasible for (\ref{eq:LPDidierSDE}).
\end{proof}

The cost of $v$ for the infinite dimensional LP (\ref{eq:LPDidierSDE}) is given by $v(\vx_0) = v^\star(\vx_0) + \theta (u_\phi(\vx_0) - \eta)$.

Following our procedure from Section \ref{sec:method}, we obtain the following convergence rate for the moment-SoS hierarchy for (\ref{eq:LPDidierSDE}) from \cite{henrion2023moment}.

\begin{thm}[Convergence rate for exit location of stochastic systems]\label{thm:SDEExitLocationConvRate} \leavevmode 

    Let {\bf Assumption \ref{asm:poprateOld}} hold for $\bX$ and $\partial \bX$, and let \L{} (resp. \L$_\partial$) be the \L{}ojaciewicz exponent of $\vh$ (resp. $\vh_\partial$). Then, defining $\widehat{\text{\L}} := \max\{\text{\L},\text{\L}_\partial\}$, under {\bf Condition \ref{cond:SDEDidier}}, it holds for $\ell \in \N$ large enough that 
    \begin{equation}
        v^\star(\vx_0) - v^\star_\ell(\vx_0) \in \cO \left( \ell^{-\frac{1}{(2.5 + s)n\widehat{\text{\L}}}}\right) \text{ for any } s > 0.
    \end{equation}
\end{thm}

\begin{proof}\begin{subequations}
    Let $s > 0$, $k := \lceil\nicefrac{3.5}{s}\rceil$, $v,u \in C^{k+2}(\bK)$ be extensions of $v^\star$ and $u_\phi$ according to {\bf Lemma \ref{thm:HölderExtension}} with
    \begin{equation}
        \|v\|_{C^{k+2}(\bK)} \leq c_2 \|v\|_{C^{k+2}(\bX)} \quad \text{and} \quad \|u\|_{C^{k+2}(\bK)} \leq c_2 \|u_\phi\|_{C^{k+2}(\bX)}
    \end{equation}
    for some constant $c_2 \in \R$. For $d\in \N$, by {\bf Theorem \ref{thm:CkApprox}}, let $p_q,q_d \in \R_d[\vx]$ with
    \begin{equation}
        \|v - p_d\|_{C^{k+2}(\bK)}, \|u - q_d\|_{C^{k+2}(\bK)} \leq \frac{c_1}{d^k}
    \end{equation}
    for some constant $c_1 \in \R$. Further, we set
    \begin{equation*}
        A := \sup\limits_{\vx \in \bK} \frac{1}{2}\sum\limits_{i,j = 1}^n |a_{ij}(\vx)| + \sum\limits_{i = 1}^n |f_{0i}(\vx)|.
    \end{equation*}
    We define $\theta_d,\eta_d > 0$ for large enough $d\in \N$ as
    \begin{equation}\label{eq:DefThetaEtaSDE}
        \theta_d := \frac{2c_2A}{d^k (1-\nicefrac{c_1A}{d^k})} \in \mathcal{O}(d^{-k}) \quad \text{and} \quad \eta_d := \frac{c_2}{d^k}(1 + 2 \theta_d^{-1}) \in \mathcal{O}(1).
    \end{equation}
    Motivated by (\ref{eq:DefInwardDirectionSDE}), we define $v_d$ by
    \begin{equation}
        v_d := p_d + \theta_d (q_d - \eta_d)
    \end{equation}
    and verify that $v_d$ is feasible for (\ref{eq:HierarchyLPDidier}) for $\ell \in \N$ to be determined. We first bound $v_d$ on $\bK$. On $\bK$ we have for large enough $d\in \N$
    \begin{equation*}
        \begin{array}{cccccccccccc}
            |v_d| & \leq & |v| & + & |v-p_d| & + & \theta_d (|u| + |u-q_d| + \eta_d)& \\
            & \leq & \|v\|_\infty^{\bK} & + & \|v-p_d\|_\infty^{\bK} & + & \theta_d (\|u\|_\infty^{\bK} + \|u-q_d\|_\infty^{\bK} + \eta_d)& \\
            & \leq & \|v\|_{C^{2}(\bK)} & + & \|v-p_d\|_{C^{2}(\bK)} & + & \theta_d (\|u\|_{C^{2}(\bK)} + \|u-q_d\|_{C^{2}(\bK)} + \eta_d) & \\
            & \leq & c_2 \|v^\star\|_{C^{k+2}(\bX)} & + & \frac{c_1}{d^k} & + & \theta_d c_2 \|u_\phi\|_{C^{k+2}(\bX)} + \theta_d \frac{c_1}{d^k} + \theta_d \eta_d\\
            & \leq & c_2 \|v^\star\|_{C^{k+2}(\bX)} & + & 1 & + & \theta_d c_2 \|u_\phi\|_{C^{k+2}(\bX)} + 1 + 1\\
            & =: & C_1.
        \end{array}
    \end{equation*}
    Where only in the second but last line we used that $d\in \N$ is large enough such that $\frac{c_1}{d^k},\theta_d \frac{c_1}{d^k},\theta_d \eta_d \leq 1$.
    Similarly, we can bound $\cL v_d$ on $\bK$. Note first that for all $\vx \in \bK$ we have for all $w \in C^2(\bK)$
    \begin{eqnarray*}
        |\cL w(\vx)| & = & \left| \frac{1}{2}\sum\limits_{i,j = 1}^n a_{ij}(\vx) \partial_i \partial_j w(\vx) + \sum\limits_{i = 1}^n f_{0i} (\vx) \partial_i w(\vx)\right|\\
        & \leq & \left(\frac{1}{2}\sum\limits_{i,j = 1}^n |a_{ij}(\vx)| + \sum\limits_{i = 1}^n |f_{0i}(\vx)|\right) \|w\|_{C^2(\bK)} \leq A \|w\|_{C^2(\bK)}.
    \end{eqnarray*}
    For large enough $d\in \N$ we get on $\bK$
    \begin{eqnarray*}
        |\cL v_d| & \leq & |\cL v| + |\cL (v-p_d)| + \theta_d( |\cL u| + |\cL (u - q_d)|)\\
        & \leq & A \|v\|_{C^2(\bK)} + A \|v - p_d\|_{C^2(\bK)} + \theta_d A \|u\|_{C^2(\bK)} + \theta_d A \|u - q_d\|_{C^2(\bK)}\\
        & \leq & A c_2\|v^\star\|_{C^{k+2}(\bX)} + A \frac{c_1}{d^k} + \theta_d A {c_2} \|u_\phi\|_{C^{k+2}(\bX)} + \theta_d A \frac{c_1}{d^k}\\
        & \leq & A c_2\|v^\star\|_{C^{k+2}(\bX)} + 3 =: C_2.
    \end{eqnarray*}
    Next, we verify the strict feasibility of $v_d$ for (\ref{eq:HierarchyLPDidier}). It holds on $\bX$
    \begin{eqnarray}\label{eq:SDELowerBound1}
        \cL v_d & = & \underbrace{\cL v^\star}_{= 0} + \cL (p_d - v^\star) + \theta_d (\underbrace{\cL u_\phi}_{= -1} + \cL (q_d-u_\phi))\notag\\
        & \leq & A \|v^\star - p_d\|_{C^{k+2}(\bX)} - \theta_d + \theta_d A \|u_\phi - q_d\|_{C^{k+2}(\bX)}\notag\\
        & \leq & A \|v - p_d\|_{C^{k+2}(\bX)} - \theta_d + \theta_d A \|u - q_d\|_{C^{k+2}(\bX)}\notag\\
        & \leq & \theta_d (-1 + A\frac{c_1}{d^k}) + \frac{Ac_1}{d^k}\notag\\
        & \overset{(\ref{eq:DefThetaEtaSDE})}{=} & - \frac{Ac_1}{d^k}
    \end{eqnarray}
    and on $\partial \bX$ it holds
    \begin{eqnarray}\label{eq:SDELowerBound2}
        v_d & = & \underbrace{v^\star}_{= g} + p_d - v^\star + \theta_d(\underbrace{u_\phi}_{= 0} + q_d - u_\phi -\eta_d)\notag\\
        & \leq & g + \|v^\star - p_d\|_\infty^{(\partial\bX)} + \theta_d \|u_\phi - q_d\|_\infty^{(\partial\bX)} -\theta_d \eta_d\notag\\
        & \leq & g + \theta_d(1+ \frac{c_1}{d^k} - \eta_d)\notag\\
        & \overset{(\ref{eq:DefThetaEtaSDE})}{=} & g-\frac{c_2}{d^k}.
    \end{eqnarray}
    Applying {\bf Theorem \ref{thm:OldBaldi}}, we get that $v_d$ is feasible for $\ell \in \N$ with
    \begin{equation}\label{eq:SDEellBound}
        \ell \geq d^{3.5n\widehat{\text{\L}}} \max\left\{\gamma(n,\vh) \left(\frac{C_2 d^k}{Ac_1}\right)^{2.5m \text{\L}},\gamma(n,\vh_\partial) \left(\frac{C_1 d^k}{c_2}\right)^{2.5n \text{\L}_\partial} \right\} \in \mathcal{O}\left(d^{(3.5+2.5k)n\widehat{\text{\L}}}\right).
    \end{equation}
    For such $\ell$ the optimal value $v^\star_\ell$ is at least $v_d(\vx_0)$; hence we get
    \begin{eqnarray}\label{eq:SDEConvRated}
        v^\star(\vx_0) - v^\star_\ell(\vx_0) & \leq & (v^\star - v^\star - (p_d - v^\star) - \theta_d(u_\phi + q_d - u_\phi -\eta_d))(\vx_0)\notag\\
        & \leq & \frac{c_1}{d^k} + \theta_d (\|u_\phi\|_\infty^\bX + \frac{c_1}{d^k} + \eta_d) \in \mathcal{O}\left(d^{-k}\right).
    \end{eqnarray}
    Defining $\epsilon_d := d^{-k}$, \eqref{eq:SDEellBound} yields that $v_d$ is feasible for $\ell \geq \ell_b \in \cO\left( \epsilon_d^{-\frac{1}{\left(\nicefrac{3.5}{k}+2.5\right)n\widehat{\text{\L}}}}\right)$ so that \eqref{eq:SDEConvRated} ensures that $|v^\star(\vx_0) - v^\star_{\ell_b}(\vx_0)| \leq \epsilon_d \in \cO\left(\ell_b^{-\frac{1}{\left(\nicefrac{3.5}{k}+2.5\right)n\widehat{\text{\L}}}}\right) \subset \cO\left(\ell_b^{-\frac{1}{\left(s+2.5\right)n\widehat{\text{\L}}}}\right)$, which is the announced result.
\end{subequations}
\end{proof}

\section{Application: Volume computation} \label{sec:vol}

In this section, we analyze the moment-SoS hierarchy for computing the volume $\lambda(\bX)$ of a bounded basic semi-algebraic set
$$ \bX := \bS(\vh) = \{\vx \in \R^n \; : h_1(\vx) \geq 0, \ldots, h_r(\vx) \geq 0\} \subset \bB $$
with $r \geq 1$ integer and $h_1,\ldots,h_r \in \R[\vx]$. 

A standard moment-SoS hierarchy method {is discussed in~\cite{korda2018convergence}, where a bad convergence behavior is highlighted} both in practice and in theory, due to a Gibbs phenomenon occurring in the SoS approximations. An alternative formulation was proposed in~\cite{lasserre2017stokes}, with much better numerical behavior, which was recently supported by a qualitative analysis in~\cite{tacchi2023stokes}, showing that no Gibbs phenomenon occurs in this improved formulation. In this section, we complement the existing qualitative analysis with a first quantitative analysis of how much better the upper bounds on the convergence rate are in the improved formulation. 

\subsection{The standard approach}\label{sec:VolumeStandard}

The standard moment-SoS approach to numerically solve the volume problem is discussed in detail in~\cite{tacchi2023stokes}. The method consists of formulating a GMP whose optimal solution is $\lambda(\bX)$, after which one numerically approximates this optimal solution using the moment-SoS hierarchy.
The LPs read
\begin{center} \begin{subequations}
\begin{minipage}{0.49\textwidth}
\begin{align}\label{eq:volp}
\lambda(\bX) = \sup\limits_{\mu\in\cM(\bX)} & \displaystyle\int 1 \; \od\mu \notag \\
\st \; &
\mu \in \cM(\bX)_+\\
& \lambda_\bY - \hat{\mu} \in \cM(\bY)_+ \notag 
\end{align}
\end{minipage}
\begin{minipage}{0.49\textwidth}
\begin{align}\label{eq:vold}
\lambda(\bX) = \inf\limits_{w \in \cP(\bY)} & \displaystyle\int w \; \od\lambda_\bY \notag \\
\st \; & w|_{\bX} - 1 \in \cC(\bX)_+\\
& w \in \cC(\bY)_+ \notag
\end{align}
\end{minipage}
\end{subequations} \end{center}
where $\bY$ contains $\bX$ and is an Archimedean basic-semialgebraic set, $\lambda_\bY$ denotes the Lebesgue measure on $\bY$ such that the numbers $\int \vy^{\bbeta} \; \od\lambda_\bY(\vy)$, $\bbeta \in \N^n$, are known. The measure $\hat{\mu} \in \cM(\bY)_+$ is the extension (by zero) of $\mu$ to $\bY$ via $\displaystyle\int_\bY h \; \od \hat{\mu} := \int_\bX h \; \od \mu$, i.e. the operator $\mu \mapsto \hat{\mu}$ is the adjoint of the restriction operator $\cC(\bY) \ni w\mapsto w\big|_\bX$.

The optimization problems~\eqref{eq:volp} and\eqref{eq:vold} are of the form~\eqref{eq:LPGeneral} where $\cY = \cP(\bY)$ is equipped with the uniform convergence topology, the sets $\bX_i$ for $i= 1,2$ are given by $\bX_1 := \bX$ and $\bX_2:= \bY$, the linear operator $\cA:\cY \rightarrow \cX$ is defined by $\cA \, w := (w\big|_\bX,w)$, the function $\vg$ given by $\vg (\vx):= (1,0)$ for all $\vx \in \bY$ and $T = \lambda\big|_{\bK}$. 

\begin{rem}\label{rem:ExtensionVol}
The ``optimal point'' $w^\star := \mathds{1}_\bX{:\bY \rightarrow \R}$, the indicator function on $\bX$, is not polynomial or even continuous whenever $\bX$ is not a connected component of $\bY$. Thus, for any feasible $w\in \cP(\bY)$, by continuity, we have
\begin{equation*}
    \displaystyle\int w \; \od\lambda_\bY = \displaystyle\int_{\bX} \underbrace{w(\vx)}_{\geq 1} \; \od \vx + \underbrace{\displaystyle\int_{\bY \setminus \bX} w(\vx) \; \od \vx}_{> 0} > \lambda (\bX).
\end{equation*}
As discussed in \textbf{Remark~\ref{rem:NoMinimizer}}, the LP~\eqref{eq:vold} can be relaxed to admit $w^\star$ as a minimizer. This can be achieved by enlarging $\cP(\bY)$ to {$\hat{\cY} := \mathrm{L}^1(\bY)\cap \mathrm{L}^\infty(\bY)$ the space of bounded integrable functions on $\bY$ equipped with the $\mathrm{L}^1(\bY)$ topology, $\cX$ to $\hat{\cX} := \mathrm{L}^1(\bX) \times \mathrm{L}^1(\bY)$}, and by extending $\cA$ and $T$ accordingly. {\textbf{Lemma \ref{lem:StandVolRelaxation}} assures that this relaxation does not impact the optimal value of the resulting LP (because the operator $\hat{A}(w):= (w\big|_\bX,w)$ is a positive operator and by \textbf{Theorem~\ref{thm:onesideApprox}} the set $\cY$ is one-sided dense in $\hat{\cY}$).}

\end{rem}

In this subsection, the convergence rate from \cite{korda2017convergence}, for the hierarchy of SoS programs for the volume problem, is improved with the help of {\bf Theorem \ref{thm:OldBaldi}}. We make the convention 
\begin{equation}\label{eq:bXHypercube}
\bY = \bK = [-1,1]^n = \bS(\vf) \quad \text{with} \quad \vf = (1 - x_1^2, \ldots, 1-x_m^2)
\end{equation}
(which is the best choice for computing the convergence rate; notice that $\bY$ can be chosen arbitrarily here without changing the optimal value $\lambda(\bX)$). Let us consider a hierarchy of problems from \cite{korda2017convergence}, which can be regarded as SoS strengthenings of Problem \eqref{eq:vold}:

\begin{equation}
    \begin{split}
        \od_\bX^\ell:=\text{\(\inf\limits_{w\in \mathbb{R}_\ell[\vy]}\)} \quad
        &\int w \; \od\lambda_\bY \\
        \text{s.t.} \quad
        &w-1\in\mathcal{Q}_{\ell}(\vh),\\
        &w\in\mathcal{Q}_{\ell}(\vf).
    \end{split}
    \label{eq: volsosKorda}
\end{equation}

To compute the rate of convergence of \eqref{eq: volsosKorda}, we need to estimate the dependence of {$\varepsilon \geq 0$} on the degree $\ell$ for which it holds $|\od_\bX^\ell-\lambda(\bX)|<\varepsilon$.

We shall use the standard condition from \cite{korda2017convergence}:

\begin{cond}[Finite one-sided Gibbs phenomenon; {\cite[Assumption 2]{korda2017convergence}}]
\label{cond:gibbs} \leavevmode

There exists a constant $c_G \geq 0$ depending only on $\bX$ and a sequence $(w_d)_{d\in \N} \subset \R[\vx]$ with
\begin{equation*}
    \begin{array}{cccc}
         w_d & \in & \underset{w \in \R_d[\vx]}{\argmin} & \displaystyle\int_{\bK} w(\vx) \; \od \vx \\
         & & \text{s.t.} & w \geq \ind_\bX
    \end{array}
\end{equation*}
satisfying $\max\{w_d(\vx) \; ; \, \vx \in \bK\} \leq c_G$. 
\end{cond}

\begin{rem}[On finite Gibbs phenomena] \leavevmode

    It is well known in Fourier analysis that the Gibbs phenomenon that occurs when approximating a discontinuous periodic function $\phi$ with trigonometric polynomials induces an overshoot of approximately 9\%, and thus the polynomial approximation is uniformly bounded by some constant $c_\phi$ that only depends on $\phi$. This is also the case for generic $L^1$ approximation of discontinuous functions with algebraic polynomials~\cite{davis2022gibbs}. However, to our best knowledge, these results have not been extended to \textit{one-sided} polynomial approximations, as is the case here. Following~\cite{korda2018convergence}, we conjecture (which is supported by the numerical experiments displayed in~\cite{lasserre2017stokes,tacchi2023stokes}) that {\bf Condition \ref{cond:gibbs}} also holds generically.
\end{rem}

\begin{thm}[Effective Putinar Positivstellensatz for volume computation]
\label{thm: volconvrate} \leavevmode

Define $\gamma(n,\vf,\vh) := \max\left(\gamma(n,\vf), \gamma(m,\vh)\right)$ and $\widehat{\text{\L{}}} := \max\{\text{\L{}},1\}$ where \L{} is the \L{}ojaciewicz exponent of $\vh$. Then, under {\bf Condition \ref{cond:gibbs}}, there exists $C > 0$ such that, for all $\varepsilon \in (0,1)$ it holds $\od_\bX^\ell-\lambda (\bX) < \varepsilon$ for any
\begin{equation}
 \ell \geq \gamma(n,\vf,\vh) \, \left(\nicefrac{C}{\epsilon}\right)^{3.5 n \widehat{\text{\L{}}}} \left(1 + 2^{n+1}\frac{c_G}{\epsilon}\right)^{2.5 n \widehat{\text{\L{}}}}  \in \cO\left(\frac{1}{\varepsilon^{6n\widehat{\text{\L{}}}}}\right)
\label{degreeest}
\end{equation}
\end{thm}

\begin{proof}
\begin{subequations}
We first notice that for any $d,\ell \in \N$ and any $w \in \R_d[\vx]$ feasible for \eqref{eq: volsosKorda} at order $\ell$, it holds
    \begin{equation} \label{eq:vol_triangle}
        0 \leq \od_\bX^\ell - \lambda(\bX) = \od_\bX^\ell - \int \ind_\bX \; \od\lambda_\bK \leq \int (w - \ind_\bX) \; \od\lambda_\bK \leq \int (w_d - \ind_\bX) \; \od\lambda_\bK + \int |w-w_d| \; \od\lambda_\bK,
    \end{equation}
where $w_d$ comes from {\bf Condition \ref{cond:gibbs}}. Let $\epsilon > 0$. From {\bf Theorem \ref{thm:onesideApprox}}, we know that
\begin{equation} \label{eq:vol_osApprox}
    \int (w_d - \ind_\bX) \; \od\lambda_\bK \leq \overline{c} \, \omega_{\ind_\bX,0}^{L^1}(\lambda_\bK,\nicefrac{1}{d}),
\end{equation}
with
$$\omega_{\ind_\bX,0}^{L^1}(\lambda_\bK,\nicefrac{1}{d}) = \int \sup\left\{\vphantom{\sum} |\ind_\bX(\vx)-\ind_\bX(\vy)| \; ; \, \vx \in \bK, \ \|\vx-\vy\|\leq \nicefrac{1}{d}\right\} \; \od\lambda_\bK(\vy).$$
From \cite[{\bf Lemma 1} ]{korda2018convergence}, there exists a $c_4 \geq 0$ 
depending only on $\bX$ such that
$$ \omega_{\ind_\bX,0}^{L^1}(\lambda_\bK,\nicefrac{1}{d}) \leq \frac{c_4}{d}, $$
which we reinject into \eqref{eq:vol_osApprox} to get (introducing $C := 2\,\overline{c}\, c_4$)
\begin{equation}
    \int (w_d - \ind_\bX) \; \od\lambda_\bK \leq \frac{C}{2d}.
\end{equation}
It remains to specify $w$ so that we get a good bound on the second term in \eqref{eq:vol_triangle}. Defining $w := w_d + \nicefrac{C}{2 d\lambda(\bK)} = w_d + \nicefrac{c}{d2^{n+1}}$, we automatically get
$$ \int |w-w_d| \; \od\lambda_\bK = \frac{C}{2d}, $$
which we can reinject into \eqref{eq:vol_triangle} to get
$$ 0 \leq \od_\bX^\ell - \lambda(\bX) \leq \frac{C}{d}, $$
for any $\ell$ such that $w$ is feasible in \eqref{eq: volsosKorda} at order $\ell$. The last remaining piece is a value for such $\ell$, which we compute using again {\bf Theorem \ref{thm:OldBaldi}} on both $\bK$ and $\bX$. We first work on $\bX$: we want a lower bound on $\ell \in \N$ such that $p := w-1\in \cQ_\ell(\vh)$. Denoting \L{} the \L{}ojaciewicz exponent of $\vh$, we get from the effective Putinar Positivstellensatz that $p\in \cQ_\ell(\vh)$ for
$$ \ell \geq \gamma(n,\vh) \, \deg(p)^{3.5n\text{\L{}}}\left(\nicefrac
{\|p\|}{p^\star_\bX}\right)^{2.5 n \text{\L{}}},$$
{where we insert, recalling $w = w_d + \nicefrac{c}{d2^{n+1}}$},
\begin{itemize}
    \item $\deg(p) = \deg(w-1) = \deg(w_d + \nicefrac{C}{(2^{n+1}d)} - 1) = \deg (w_d) \leq d$
    \item $\|p\| = \max\{p(\vx) \; ; \, \vx \in [-1,1]^n\} \leq c_G + \nicefrac{C}{(2^{n+1}d)} - 1 $
    \item $p^\star_\bX = \min \{p(\vx) \; ; \, \vx \in \bX\} \geq \nicefrac{C}{(2^{n+1}d)}$.
\end{itemize}
We get 
\begin{equation} \label{eq:vol_lbound_Yd}
    \ell \geq \gamma(n,\vh) \, d^{3.5 n \text{\L{}}} \, \left(\dfrac{c_G + \nicefrac{C}{(2^{n+1}d)} - 1}{\nicefrac{C}{(2^{n+1}d)
    }}\right)^{2.5 n \text{\L{}}} = \gamma(n,\vh) \, d^{3.5 n \text{\L{}}} \, \left(1 + 2^{n+1}\frac{c_G - 1}{C}d\right)^{2.5 n \text{\L{}}}.
\end{equation}
Next, we work on $\bK = [-1,1]^n = \bS(\vf)$, for which the \L{}ojaciewicz exponent is $1$, and we want a lower bound on $\ell$ such that $w \in \cQ_\ell(\vf)$, which is again given by the effective Putinar Positivstellensatz as
$$ \ell \geq \gamma(n,\vf) \, \deg(w)^{3.5 n} \left(\nicefrac{\|w\|}{w_\bK^\star}\right)^{2.5 n} \qquad \text{with:} $$
\begin{itemize}
    \item $\deg(w) = \deg(w_d + \nicefrac{C}{(2^{n+1}d)}) \leq d$
    \item $\|w\| = \max\{w(\vx) \; ; \vx \in \bK\} \leq c_G + \nicefrac{C}{(2^{n+1}d)}$
    \item $w^\star_\bK = \min\{w(\vx) \; ; \, \vx \in \bK\} \geq \nicefrac{C}{(2^{n+1}d)}$
\end{itemize}
so that we get
\begin{equation} \label{eq:vol_bound_Xd}
    \ell \geq \gamma(n,\vf) \, d^{3.5 n} \left(1 + 2^{n+1}\frac{c_G}{C}d\right)^{2.5 n}
\end{equation}
Eventually, taking $\ell$ larger than the maximum between the right hand sides of \eqref{eq:vol_lbound_Yd} and \eqref{eq:vol_bound_Xd} with $d := \lceil\nicefrac{C}{\epsilon}\rceil$ yields the announced bound.
\end{subequations}

\end{proof}

\begin{cor}
Using the notations in {\bf Proposition \ref{thm: volconvrate}} and under {\bf Condition \ref{cond:gibbs}}, it holds 
$$\od_\bX^\ell - \lambda(\bX) \in \cO\left(\ell^{-\frac{1}{6 n \widehat{\text{\L{}}}}}\right) \quad \text{as} \quad \ell\to\infty.$$
\end{cor}

\begin{proof}
Simply inverting the expression in \eqref{degreeest}.
\end{proof}

\subsection{Stokes constraints}\label{sec:Stokes}

In this subsection, we investigate the effect of smoothness of optimal solutions to the infinite-dimensional LP. We consider the case of only one defining polynomial inequality, i.e. $r=1$. This means we compute the volume of the \textit{open set}

$$ \bX := \{\vx \in \R^n \; ; h(\vx) > 0 \}$$
and, as in \eqref{eq:bXHypercube}, let $\bK$ be given by $\bK = [-1,1]^n = \bS(\vf)$. Moreover, we add the following condition
\begin{cond} \label{cond:smooth}
It holds $\grad \, h(\vx) \neq 0 \text{ for all } \vx \in \partial \bX$, in particular the boundary $\partial \bX$ is smooth.
\end{cond}

\begin{rem}[No more Gibbs phenomenon] \leavevmode

    Note that now, we do not assume the finite Gibbs phenomenon from {\bf Condition \ref{cond:gibbs}}. As we will show, this is because in the following formulations, optimal solutions cease to be discontinuous and thus the Gibbs phenomenon does not occur anymore, see~\cite{tacchi2023stokes} for a more in-depth discussion on that topic.
\end{rem}

In~\cite{tacchi2023stokes}, a new formulation is designed to cope with the slow convergence of the moment-SoS hierarchy corresponding to \eqref{eq:volp} and \eqref{eq:vold} using the divergence theorem:
\begin{subequations} \label{eq:stokes}
\begin{equation}\label{eq:stokesp}
\begin{tabular}{ccll}
    $\lambda(\bX) =$ & $\sup\limits_{\substack{\mu\in\cM(\overline{\bX}) \\ \nu \in \cM(\partial\bX)}}$ & $\displaystyle\int 1 \; \od\mu$ &\\
    & \st \; & $\mu \in \cM(\overline{\bX})_+$ & \\
    & & $\nu \in \cM(\partial\bX)_+$ & \\
& & $\lambda_\bK - \mu \in \cM(\bK)_+$ &  \\
& & $\displaystyle\int_\bX \Delta u \; \od \mu = - \displaystyle\int_{\partial \bX}(\grad \ h) \cdot \grad \ u \ \od \nu$ & $\forall u \in C^2(\R^n)$
\end{tabular}
\end{equation}
and its dual problem
\begin{equation}\label{eq:stokesd}
\begin{tabular}{ccl}
    $\lambda(\bX) =$ & $\inf\limits_{\substack{w \in \cP(\bK) \\ u \in \cP(\overline{\bX})}}$ & $\displaystyle\int w \; \od\lambda_\bK$ \\
& \st & $w \in \cC(\bK)_+$ \\
& & $w|_{\overline{\bX}} - \Delta u - 1 \in \cC(\overline{\bX})_+$\\
& & $-(\grad \ u \cdot \grad \ h)|_{\partial\bX} \in \cC(\partial\bX)_+$
\end{tabular}
\end{equation}
\end{subequations}

The optimization problems~\eqref{eq:stokesd} and~\eqref{eq:stokesp} are of the form LP~\eqref{eq:LPGeneral} where
\begin{equation*}
    \cY := \cP(\bK) \times \cP(\overline{\bX})
\end{equation*}
is equipped with the induced topology from $\cC(\bK) \times C^2(\overline{\bX})$, the sets $\bX_i$ for $i = 1,2,3$ are given by
\begin{equation*}
    \bX_1 = \bK, \quad \bX_2 := \overline{\bX} \quad \text{and} \quad  \bX_3 := \partial \bX,
\end{equation*}
the linear form $\bT :\cY \rightarrow \R$ is given by
\begin{equation*}
    \langle \bT, \vw \rangle := \displaystyle\int w \; \od\lambda_\bK \quad \text{for all } \vw := (w,u) \in \cY,
\end{equation*}
the operator $\cA:\cY\rightarrow \cX$ is defined as
\begin{equation*}
    \cA \ (w,u) := (w|_{\overline{\bX}} - \Delta u, w,-(\grad \ u \cdot \grad \ h)|_{\partial\bX}),
\end{equation*}
and $\vg$ is the constant function taking the value $(1,0,0)$.

\begin{rem}[Interpretation of the additional constraints in~\eqref{eq:stokesp} and~\eqref{eq:stokesd}]
    Compared to the standard approach LPs~\eqref{eq:volp} and~\eqref{eq:vold} the LPs~\eqref{eq:stokesp} and~\eqref{eq:stokesd} contain additional feasibility constraints. They enforce additional structure of the optimal solution. To be more precise: For the GMP~\eqref{eq:volp}, an optimal solution is $\mu = \lambda\big|_\bX$, i.e. $\mu$ is the Lebesgue measure restricted to $\bX$. Applying divergence theorem we have for all $u \in C^2(\R^n)$
    \begin{equation*}
        \int_\bX \Delta u \; \od\mu = \int_\bX \div \ \grad \ u \; \od\mu = \int_\bX \div \ \grad \ u(\vx) \; \od(\vx) = \int_{\partial \bX} \grad \ u \cdot \vn \; \od \mathcal{H}^{n-1}
    \end{equation*}
    where $\div$ denotes the divergence, $\mathcal{H}^{n-1}$ the $n-1$ dimensional Hausdorff measure and $\vn(\vx) = {-}\frac{\grad \ h(\vx)}{\|\grad \ h(\vx)\|}$ the outer normal for $\bX$ at $\vx \in \partial \bX$. This shows that $(\mu, {\frac{1}{\|\grad \ h\|}}\mathcal{H}^{n-1}\big|_{\partial \bX})$ is feasible (and optimal) for the GMP~\eqref{eq:stokesp}. Similarly for~\eqref{eq:stokesd}, for feasible $(w,u)$ the feasibility constraints together with the divergence theorem give
    \begin{eqnarray*}
        \displaystyle\int w \; \od\lambda_\bK \geq \displaystyle\int w \; \od\lambda_\bX \geq \displaystyle\int \Delta \ u + 1 \; \od\lambda_\bX = -{\frac{1}{\|\grad \ h\|}}\displaystyle\int_{\partial \bX} \grad \ u \cdot \grad \ h \; \od \mathcal{H}^{n-1}  + \lambda (\bX) \geq \lambda (\bX).
    \end{eqnarray*}
    For more insights into the LPs~\eqref{eq:stokesp} and~\eqref{eq:stokesd} we refer to~\cite{tacchi2023stokes}.
\end{rem}

\begin{rem}
    The formulations~\eqref{eq:stokesp} and~\eqref{eq:stokesd} are slight refinements of the formulation in~\cite{tacchi2023stokes}. The functional LP formulation in~\cite{tacchi2023stokes} treats decisions variables $(w,\vu)\in \cC(\bK) \times C^1(\overline{\bX})^n$ while in~\eqref{eq:stokesd} we treat $(w,u)\in \cC(\bK) \times C^2(\overline{\bX})$. The relation between both formulations is that we can take $\vu = \grad \ u$ for $(w,u)$ feasible for~\eqref{eq:stokesd}, see~\cite{tacchi2023stokes} and {\bf Theorem~\ref{thm:smooth}}. We thank an anonymous reviewer for suggesting this refined formulation.
\end{rem}

It has been proved in~\cite{tacchi2023stokes} that the existence of an optimal solution to \eqref{eq:stokesd} can be deduced from the existence of a solution to a Poisson PDE with Neumann boundary condition:

\begin{equation} \label{eq:poisson}
\begin{array}{rrrlr}
- \Delta u & = & \phi & \quad \text{ in } & \bX \\
\partial_\vn u & = & 0 & \quad \text{ on } & \partial \bX \\
\phi & \leq & 1 & \quad \text{ in } & \bX \\
\phi & = & 1 & \quad \text{ on } & \partial \bX
\end{array}
\end{equation}
Namely, for a pair $(u,\phi)$ satisfying \eqref{eq:poisson}, set
\begin{equation}\label{eq:vu,wfromPDE}
    w(\vx) := \begin{cases}
        1-\phi(\vx), & \vx \in \bX\\
        0, & \text{ else,}
    \end{cases}
\end{equation}
then $(w,u)$ is optimal for (\ref{eq:stokesd}).

In~\cite{tacchi2023stokes}, $\phi$ is proposed under the form 
$$ \phi(\vx) = 1 - h(\vx) \sum\limits_{i=1}^N \frac{\lambda(\bX_i)}{\int {h} \; d\lambda_{\bX_i}} \ind_{\bX_i}(\vx) $$
where the $\bX_i$ are the connected components of $\bX$. As a result, $\phi$ was proved to be only Lipschitz continuous, so that the optimal function $w = 1 - \phi$ was also only Lipschitz continuous. However, another, smooth optimal function can be designed. 

\begin{thm}[Existence of smooth solutions] \label{thm:smooth} \leavevmode

There exist smooth functions $u,\phi \in C^\infty(\overline{\bX})$ solutions to \eqref{eq:poisson}. Further, $u,\phi \in C^\infty(\overline{\bX})$ can be chosen such that $w$ given by \eqref{eq:vu,wfromPDE} is smooth and that $(w,u)$ is optimal for \eqref{eq:stokesd}, i.e. it holds $\displaystyle\int w\; \od \lambda_\bK = \lambda(\bX).$
\end{thm}
\begin{proof} See Appendix \ref{appendix:Stokes}.
\end{proof}

\begin{rem}\label{rem:ExtensionVolStokes}
    {As discussed in \textbf{Remark~\ref{rem:NoMinimizer}} it can be beneficial to enlarge the decision space for the LP at hand so that a minimizer exists. We can do so for the LP~\eqref{eq:stokesd}: By extending $\cY = \cP(\bK) \times \cP(\overline{\bX})$ to $\hat{\cY} := \cC(\bK) \times C^2(\overline{\bX})$ in~\eqref{eq:stokesd}, we infer from \textbf{Theorem~\ref{thm:smooth}} (and \textbf{Lemma \ref{lem:d*=bard*}}) that the pair $(w,u)\in \hat{\cY}$ is a minimizer of the relaxed LP with cost $\lambda(\bX)$.}
\end{rem}

The regularity result in {\bf Theorem \ref{thm:smooth}} allows us to incorporate higher order approximation rates via the Jackson-inequality {\bf Theorem \ref{thm:CkApprox}}. Its effect on the convergence rate of the moment-SoS hierarchy for the GMP \eqref{eq:stokesd} is stated in the following theorem. For $\ell \in \N$, let us denote by $\mathrm{Vol}_\ell$ the optimal value in the $\ell$-th level of the moment-SoS hierarchy for \eqref{eq:stokesd}.

\begin{thm}[Rate for Stokes-augmented volume computation]\label{thm:VolSmoothConRate} \leavevmode

    Under {\bf Condition \ref{cond:smooth}} it holds, for $\ell \in \N$ large enough and with $\widehat{\text{\L{}}} := \max \{1,\text{\L}\}$, that
    \begin{equation}\label{eq:VolStokesConvRate}
        0 \leq \mathrm{Vol}_\ell - \lambda(\bX) \in \cO \left( \ell^{-\frac{1}{(2.5 + s)n \widehat{\text{\L{}}}}}\right) \; \; \; \text{ as } \ell \to\infty \; \; \text{ for any } s > 0.
    \end{equation}    
\end{thm}

\begin{proof} {
\renewcommand\theequation{\roman{equation}}
\setcounter{equation}{0}

    Recall that we assume $\bK = [-1,1]^n$. By {\bf Theorem \ref{thm:smooth}}, let $u,\phi$ be smooth solutions of (\ref{eq:poisson}) such that $w = (1-\phi)\big|_\bX$ from (\ref{eq:vu,wfromPDE}) is smooth and optimal for (\ref{eq:stokesd}). Let $k\in \N$ and $\Bar{w}$ and $\Bar{u}$ be $C^{k+1}$ respectively $C^{k+3}$ extensions of $w$ respectively $u$ from {\bf Theorem \ref{thm:HölderExtension}}, i.e. $\Bar{w} \in C^{k+1}(\bK),\Bar{u} \in C^{k+3}(\bK)$ with
    \begin{equation*}\label{eq:Ckboundw,u}
        \|\Bar{w}\|_{C^{k+1}(\bK)} \leq c \|w\|_{C^{k+1}(\bX)}, \quad \|\Bar{u}\|_{C^{k+3}(\bK)} \leq c \|u\|_{C^{k+3}(\bX)}
    \end{equation*}
    for some constant $c = c(k,\bX)$. We denote by $W,U \in \R$ the following constants
    \begin{equation}\label{eq:VolConstantswu}
        W := \|\Bar{w}\|_\infty^\bK \leq c \|w\|_{C^{k+1}(\bX)}, \quad U:= \|\Bar{u}\|_{C^2(\bK)} \leq c \|u\|_{C^{k+3}(\bK)}.
    \end{equation}
    In the rest of the proof, we will also use the following constants:
    \begin{equation}\label{eq:VolConstants}
        a_1:= \|\Delta h\|_\infty^{\overline{\bX}}, \quad , a_2 := \inf\limits_{\vx \in \partial \bX}\|\grad \ h(\vx)\|^2 , \quad a_3 := \|h\|_{C^2(\overline{\bK})}
    \end{equation}
    Note that $a_2 > 0$ by {\bf Condition \ref{cond:smooth}}. We define an inward-pointing direction, namely, for $\theta >0$ it holds
    \begin{equation}\label{eq:defwtut}
        (w_\theta,u_\theta):= (w+2a_1 \theta,u - \theta \ h) \text{ is strictly feasible.}
    \end{equation}
    To verify this, note first that by feasibility of $w$ it holds $w\geq 0$ on $\bK$ and thus
    \begin{equation}\label{eq:wthetalower}
        0< 2a_1 \theta \leq w + 2a_1 \theta = w_\theta \leq W + 2a_1 \theta \quad \text{on }\bK.
    \end{equation}
    In particular, this shows feasibility for the first constraint in \eqref{eq:stokesd}. For the last constraint in \eqref{eq:stokesd} let $\vx \in \partial \bX$; we have
    \begin{eqnarray}\label{eq:VolCon2BoundLow}
        -(\grad \ u_\theta \cdot \grad \ h)(\vx)  = -\underbrace{(\grad \ u \cdot \grad \ h)(\vx)}_{= 0, \text{ by~\eqref{eq:poisson}}} + \theta \|\grad \ h(\vx)\|^2 = \theta \|\grad \ h(\vx)\|^2 = \theta a_2 > 0
    \end{eqnarray}
    with the constant $a_2$ from \eqref{eq:VolConstants}; i.e. feasibility for the last constraint in \eqref{eq:stokesd}. We further have, for $\vx \in \bK$
    \begin{equation}\label{eq:VolCon2BoundUp}
        |(\grad \ u_\theta \cdot \grad \ h)(\vx)| = \left|\left((\grad \ u - \theta \ \grad \ h\right) \cdot \grad \ h)(\vx)\right| \leq (U + \theta a_3) \cdot a_3.
    \end{equation}
    Now, let us verify strict feasibility for the remaining -- the second -- constraint in \eqref{eq:stokesd}. On $\bX$ we have
    \begin{eqnarray}\label{eq:bound1ConVolLow}
        w_\theta- \Delta u_\theta - 1 & = & w - \Delta u - 1 + 2\theta a_1 + \theta \Delta h = 2\theta a_1 + \theta \Delta h \geq 2\theta a_1 - \theta a_1 = \theta a_1 > 0
    \end{eqnarray}
    with $a_1$ from \eqref{eq:VolConstants}.
    Further, on $\bK$ we have
    \begin{equation}\label{eq:bound1ConVolUp}
        |w_\theta- \Delta u_\theta - 1| = | w + 2a_1 \theta - \Delta \ (u-\theta h) - 1| \leq W + 2\theta a_1 + U + \theta a_3 + 1 
    \end{equation}
    for the constants $W,U$ from \eqref{eq:VolConstantswu} and $a_1,a_3$ from \eqref{eq:VolConstants}.
    The cost for $(w_\theta,u_\theta)$ is simply
    \begin{equation}\label{eq:VolCostwtheta}
        \displaystyle\int w_\theta\; \od \lambda_\bK = \displaystyle\int w\; \od \lambda_\bK + 2a_1 \theta \lambda(\bK) = \lambda(\bX) + 2^{m+1}a_1 \theta.
    \end{equation}
    In the next step, we approximate the pair $(w_\theta,u_\theta)$ from~\eqref{eq:defwtut} by feasible polynomials. Because the pair $(w_\theta,u_\theta)$ is obtained by a polynomial perturbation of $(w,u)$ it is sufficient to approximate $w$ and $u$ by polynomials. Furthermore, doing so allows us to keep explicit control on the effect of $\theta$. We use smoothness of $w, \vu$ and {\bf Theorem \ref{thm:CkApprox}}. Let $c = c_k \in \R$ be the constant from {\bf Theorem \ref{thm:CkApprox}}. That is, there exist polynomials $p_d, q_d \in \R_d[\vx]$ with
    \begin{eqnarray}\label{eq:wvuapprox}
        \|w-p_d\|_{\infty}^\bK \;,\; \|u-q_d\|_{C^2(\bK)} \leq \frac{c_k}{d^k}.
    \end{eqnarray}
    For $\theta > 0$ and $d\geq \deg (h)$ we define
    \begin{equation*}
        p_{d,\theta} := p_d + 2\theta a_1 \in \R_d[\vx], \quad q_{d,\theta} := q_d - \theta \ h \in \R_d[\vx].
    \end{equation*}
    We have $\|w_{\theta}-p_{d,\theta}\|_{\infty}^\bK = \|w-p_d\|_{\infty}^\bK$ and $\|u_\theta -q_{d,\theta}\|_{C^2(\bK)} = \|u-q_d\|_{C^2(\bK)}$, thus, by \eqref{eq:wvuapprox}, we get
    \begin{equation}\label{eq:wthetaapprox}
        \|w_{\theta}-p_{d,\theta}\|_{\infty}^\bK \;,\; \|u_\theta -q_{d,\theta}\|_{C^2(\bK)} \leq \frac{c_k}{d^k}.
    \end{equation}
    On $\bX$ we have, for the second constraint in \eqref{eq:stokesd},
    \begin{eqnarray*}
        p_{d,\theta} - \Delta q_{d,\theta} - 1 & = & w_\theta -  \Delta u_\theta - 1 + p_{d,\theta} -  w_\theta + \Delta \ (u_\theta - q_{d,\theta}),
    \end{eqnarray*}
    and hence, from \eqref{eq:bound1ConVolLow} and \eqref{eq:wthetaapprox}, we get
    \begin{equation}\label{eq:VolPolyAppCon1Low}
        p_{d,\theta} - \Delta q_{d,\theta} - 1 \geq a_1 \theta - 2 \frac{c_k}{d^k}.
    \end{equation}
    Further, on $\bK$ we have by \eqref{eq:bound1ConVolUp} and \eqref{eq:wthetaapprox}
    \begin{eqnarray}\label{eq:VolPolyAppCon1Up}
        |p_{d,\theta} - \Delta q_{d,\theta} - 1 |& \leq & |w_\theta - \Delta u_\theta - 1| + |p_{d,\theta} - w_\theta | + |\Delta q_{d,\theta} - \Delta \ u_\theta | \notag \\
        & \leq & W + 2\theta a_1 + U + \theta a_3 + 1 + 2 \frac{c_k}{d^k}.
    \end{eqnarray}
    For the third constraint in \eqref{eq:stokesd} we have on $\partial \bX$, by \eqref{eq:VolCon2BoundLow} and \eqref{eq:wvuapprox},
    \begin{eqnarray}\label{eq:VolPolyAppCon2Low}
        -\grad \ q_{d,\theta} \cdot \grad \ h & = & -\grad \ u_\theta \cdot \grad \ h + (\grad \ u_\theta - \grad \ q_{d,\theta}) \cdot \ \grad \ h\notag\\
        & \geq & \theta a_2 + (\grad \ u_\theta - \grad \ q_{d,\theta}) \cdot \ \grad \ h \geq \theta a_2 - \frac{c_k}{d^k} \sqrt{a_2}.
    \end{eqnarray}
    Further, on $\bK$ we have, by \eqref{eq:VolCon2BoundUp} and \eqref{eq:wvuapprox},
    \begin{eqnarray}\label{eq:VolPolyAppCon2Up}
        |\grad \ q_{d,\theta} \cdot \grad \ h| & \leq & |{\grad \ } \vu_\theta \cdot \grad \ h| + |\grad \ (u_\theta - q_{d,\theta}) \cdot \ \grad \ h|\notag \\
        & \leq & (U + \theta a_3) \cdot a_3 + \frac{c_k}{d^k} a_3.
    \end{eqnarray}
    And for the third constraint in \eqref{eq:stokesd}, we have by \eqref{eq:wthetalower}
    \begin{equation}\label{eq:VolPolyAppCon3Low}
        2a_1 \theta - \frac{c_k}{d^k} \leq w_{\theta} - \frac{c_k}{d^k} \leq p_{d,\theta} \leq W + 2a_1 \theta + \frac{c_k}{d^k}.
    \end{equation}
    Before invoking the effective version of Putinar's Positivstellensatz, {\bf Theorem \ref{thm:OldBaldi}}, we make the choice
    $$\theta := \theta_d := \frac{c_k}{d^k} \max\left\{\frac{3}{a_1},\frac{1+\sqrt{2}}{a_2}\right\} \in \cO (d^{-k}).$$
    For this choice of $\theta$ we have for \eqref{eq:VolPolyAppCon1Low} and \eqref{eq:VolPolyAppCon2Low} on $\bX$ that
    \begin{equation}\label{eq:VolConvStokesPolLow}
        p_{d,\theta_d} - \Delta q_{d,\theta_d} - 1, \; -\grad \ q_{d,\theta_d} \cdot \grad \ h \geq \frac{c_k}{d^k} > 0.
    \end{equation}
    On $\bK$ it holds
    \begin{equation}\label{eq:VolConvStokesPoldLow}
        p_{d,\theta_d} \geq 6 \frac{c_k}{d^k} - \frac{c_k}{d^k} = 5 \frac{c_k}{d^k} > 0,
    \end{equation}
    in particular $(p_{d,\theta},q_{d,\theta})$ is feasible for \eqref{eq:stokesd}. Further, for the upper bounds \eqref{eq:VolPolyAppCon1Up},\eqref{eq:VolPolyAppCon2Up} and \eqref{eq:VolPolyAppCon3Low}, we have on $\bK$ for $d$ large enough (such that $\theta_d \leq 1, \nicefrac{c_k}{d^k} \leq 1$)
    \begin{equation}\label{eq:VolConvStokesPolUp}
        p_{d,\theta}, \; p_{d,\theta} - \Delta \ q_{d,\theta} - 1, \; -\grad \ q_{d,\theta} \cdot \grad \ h  \leq K
    \end{equation}
    for the constant $K := \max\{W + 2 a_1 + U + a_3 + 3, (U + a_3 +1) \cdot a_3 \}$. Now, by {\bf Theorem \ref{thm:OldBaldi}} and inserting \eqref{eq:VolConvStokesPolLow} and \eqref{eq:VolConvStokesPolUp}, the pair $(p_{d,\theta},q_{d,\theta})$ is feasible for the first two constraints in the $\ell$-th level of the moment-SoS hierarchy for \eqref{eq:stokesd} for
    \begin{equation*}
        \ell \geq \gamma(n,h) d^{3.5n \text{\L}}\left(\frac{K}{\nicefrac{c_k}{d^k}}\right)^{2.5 n\text{\L}}.
    \end{equation*}
    Similarly, by inserting \eqref{eq:VolConvStokesPoldLow} and \eqref{eq:VolConvStokesPolUp} into {\bf Theorem \ref{thm:OldBaldi}} for $\bK$ (note that $\text{\L} = 1$ in that case), we get that the pair $(p_{d,\theta},q_{d,\theta})$ is feasible for the third constraint in the $\ell$-th level of the moment-SoS hierarchy for \eqref{eq:stokesd} for
    \begin{equation*}
        \ell \geq \gamma(n,\vf) d^{3.5n}\left(\frac{K}{\nicefrac{5c_k}{d^k}}\right)^{2.5 n}
    \end{equation*}
    Taking the maximum of the just obtained two bounds for $\ell$ we get that $(p_{d,\theta},q_{d,\theta})$ for the optimization problem \eqref{eq:stokesd} for
    \begin{equation} \label{eq:ellVolSmooth}
        \ell \leq \max \{\gamma(n,h), \gamma(n,\vf) \} \left(\frac{K}{c_k}\right)^{2.5n \widehat{\text{\L{}}}} d^{(3.5 + 2.5k)n\widehat{\text{\L{}}}}.
    \end{equation}
    The cost of $(p_{d,\theta},q_{d,\theta})$ for the optimization problem \eqref{eq:stokesd} is bounded by
    \begin{align*}
        \displaystyle\int p_{d,\theta}\; \od \lambda_\bK & \leq \displaystyle\int w_\theta\; \od \lambda_\bK + \frac{c_k}{d^k} \lambda(\bK) \\
        & \overset{\eqref{eq:VolCostwtheta}}{=} \lambda(\bX) + 2^n
        \left(2a_1 \theta(d) + \frac{c_k}{d^k}\right) \\
        & = \lambda(\bX) + \frac{c_k}{d^k}2^n 
        \max\left\{\frac{3}{a_1},\frac{1+\sqrt{2}}{a_2}\right\}.
    \end{align*}
    This shows \eqref{eq:VolStokesConvRate}; namely, for $\varepsilon > 0$ take the smallest $d \in \N$ with $d \geq \left(2^n\frac{c_k}{\varepsilon} \max\{\frac{3}{a_1},\frac{1+\sqrt{2}}{a_2}\}
    \right)^\frac{1}{k}$. Then, from \eqref{eq:ellVolSmooth}, $\mathrm{Vol}_\ell -\lambda(\bX) \leq \varepsilon$ for $\ell \in \N$ with
    \begin{equation*}
        \ell \geq \max \{\gamma(n,h), \gamma(n,\vf) \} \left(\frac{K}{c_k}\right)^{2.5n \widehat{\text{\L{}}}} d^{(3.5 + 2.5k)n\widehat{\text{\L{}}}} \in \mathcal{O}\left(\left(\frac{1}{\varepsilon}\right)^{2.5n\widehat{\text{\L{}}} + \frac{3.5n\widehat{\text{\L{}}}}{k}}\right).
    \end{equation*}
    In other words, $\mathrm{Vol}_\ell -\lambda(\bX) \in \mathcal{O}\left(\ell^{-\frac{1}{2.5n {\widehat{\text{\L{}}}} + \nicefrac{3.5n{\widehat{\text{\L{}}}}}{k}}} \right)$. Taking $k\in \N$ arbitrarily large proves the claim.
    }
\end{proof}

\begin{rem}[Quantifying the efficiency of Stokes constraints] \leavevmode

    The upper bound on the convergence rate in {\bf Theorem \ref{thm:VolSmoothConRate}} improves the bound in {\bf Theorem \ref{thm: volconvrate}} by more than the power of two. This improvement originates from the smoothness of solutions $(w,u)$ of \eqref{eq:stokesd}.
\end{rem}

\section{Limitations and improved rates}\label{sec:Limitations}

\setcounter{equation}{55}

In this section, we emphasize limitations of our approach as well as consequences of further improved effective Positivstellensätz.

\paragraph{Limitations.} We identify two major obstacles preventing a direct application of our framework to several of the GMPs mentioned in the introduction. Those two limiting factors are
\begin{enumerate}
    \item Existence and regularity of minimizers (to apply Jackson-like inequalities)
    \item Existence of an inward-pointing direction (i.e. strict positivity, required by Putinar's theorem)
\end{enumerate}

For each of the two limiting factors, we state an example of an existing SoS formulation for two important problems in dynamical systems.

We begin with the approximation of the region of attraction from \cite{jones2021converse}. Consider a dynamical system of the form
\begin{equation}\label{eq:AppDynSys}
    \dot{\vx} = \vf(\vx),\quad \vx_0 \in \R^n
\end{equation}
for a polynomial vector field $\vf\in \R[\vx]^n$ with $\vf(\vO) = \vO$. We denote the corresponding flow map by $\bvarphi_t$. The region of attraction $\textsc{Roa}_\bX(\vf)$ of the equilibrium point ${\vx}^\star = \vO$ with respect to a constraint set $\bX$ is defined as
\begin{equation*}
    \textsc{Roa}_\vf:= \left\{\vx_0 \in \R^n ; \|\bvarphi_t(\vx_0) - \vx^\star\|_2 \underset{t\to\infty}{\longrightarrow} 0 \right\}.
\end{equation*}
It is assumed to have access to a compact set $\bX \subset \R^n$ for which we know $\mathrm{RoA}_\vf \subset \bX$. The authors in~\cite{jones2021converse} proposed {to approximate the ROA from the inside} by solving the following function LP: For parameters $\beta, \rho > 0$, consider
\begin{equation}\label{eq:LPJonesROA}
    \begin{tabular}{ccll}
         $\gmd$ := & $\inf\limits_{w \in \R[\vx]}$ & $\displaystyle\int_{\bX} w(\vx) \; \od \vx$\\
         & \text{s.t.} & $w(\vx) \geq 0$ & on $\bX$\\
         & & $\grad \, w(\vx) \cdot \vf(\vx) \leq -\rho \|\vx\|_2^{2\beta} (1-w(\vx))$ & on $\bX$\\
         & & $w(\vx) \geq 1$ & on  $\partial \bX$,
    \end{tabular}
\end{equation}
{for which any feasible point $w$ for~\eqref{eq:LPJonesROA} gives rise to an inner approximation of the $\mathrm{RoA}_{\vf}$ via  $w^{-1}([0,1))$, see \cite[Corollary 3]{jones2021converse}. Furthermore, } under certain regularity conditions, a (non-polynomial) {``}minimizer{''} $\Bar{w} \in C^1(\R^n)$ exists, and it holds $\mathrm{RoA}_{\vf} = \Bar{w}^{-1}([0,1))$ \cite{jones2021converse}. 

The optimization problem~\eqref{eq:LPJonesROA} is a function LP of the form~\eqref{eq:LPGeneral} for $T$ being the Lebesgue measure restricted to $\bX$, the sets $\bX_1,\bX_2,\bX_3$ given by
\begin{equation*}
    \bX_1 := \bX_2 := \bX \quad \text{and} \quad \bX_3 := \partial \bX,
\end{equation*}
and $\bY := \bX$, the operator $\cA = (\cA_1,\cA_2,\cA_3)$ given by
\begin{equation*}
    \cA_1 \, w := w, \quad \cA_2 \, w := -\grad \, w \cdot \vf +  \rho \|\cdot\|_2^{2\beta} w, \quad \cA_3 \, w := w\big|_{\partial \bX},
\end{equation*}
and on $\cY$ we use the induced topology from $C^1(\bY)$, and $\vg \in \cX$ is given by $\vg(\vx) := (0,\rho \|\vx\|_2^{2\beta},1)$.

However, in the equilibrium point $\vx^\star = 0$ it holds for any feasible $w$ that
\begin{equation*}
    \grad \, w(\vx^\star) \cdot \vf(\vx^\star) = \grad \, w(\vx^\star) \cdot 0 = 0 = 0(1-w(\vx^\star) = -\rho \|x^\star\|_2^{2\beta} (1-w(\vx^\star)).
\end{equation*}
This shows that no inward-pointing direction exists, the application of Putinar's Positivstellensatz is hampered and our framework does not apply to (\ref{eq:LPJonesROA}).

For an example that highlights difficulties between the interplay of existence and regularity of minimizers, we return to optimal control problems. Let us consider the following optimal control problem~\cite[Section 5.2.1]{liberzon2011calculus}
\begin{align}\label{eq:OptContrNonSmooth}
    V^\star(y_0) := \inf\limits_{\vy(\cdot),\vu(\cdot)} & \displaystyle\int_0^\infty \e^{-2t} \ y(t) \; \od t \notag \\
    \st \; \; & \dot{y}(t) = u(t)y(t)\\
              & y(0) = y_0\notag\\
              & y(t) \in [-1,1], \quad u(t) \in [-1,1] \notag
\end{align}
For given initial value $y_0$, the optimal value function $V^\star(y_0)$ can be computed explicitly from the following optimal control policy (depending on $y_0$)
\begin{equation}\label{eq:non_smooth_ocp_u}
    u^\star(t) = \begin{cases}
        -1, & y_0 \geq 0, \ t\geq 0\\
        1, & y_0 <0, \ 0\leq t \leq -\log (-y_0)\\
        0, & y_0 < 0, \ -\log (-y_0)< t.
    \end{cases}
\end{equation}
This control steers positive initial values closer to zero and negative initial values closer to $-1$, see Figure~\ref{fig:non_smooth_ocp_uy}.
\begin{figure}
    \centering
    \begin{subfigure}{.48\textwidth}
        \centering
        \includegraphics[width=\linewidth]{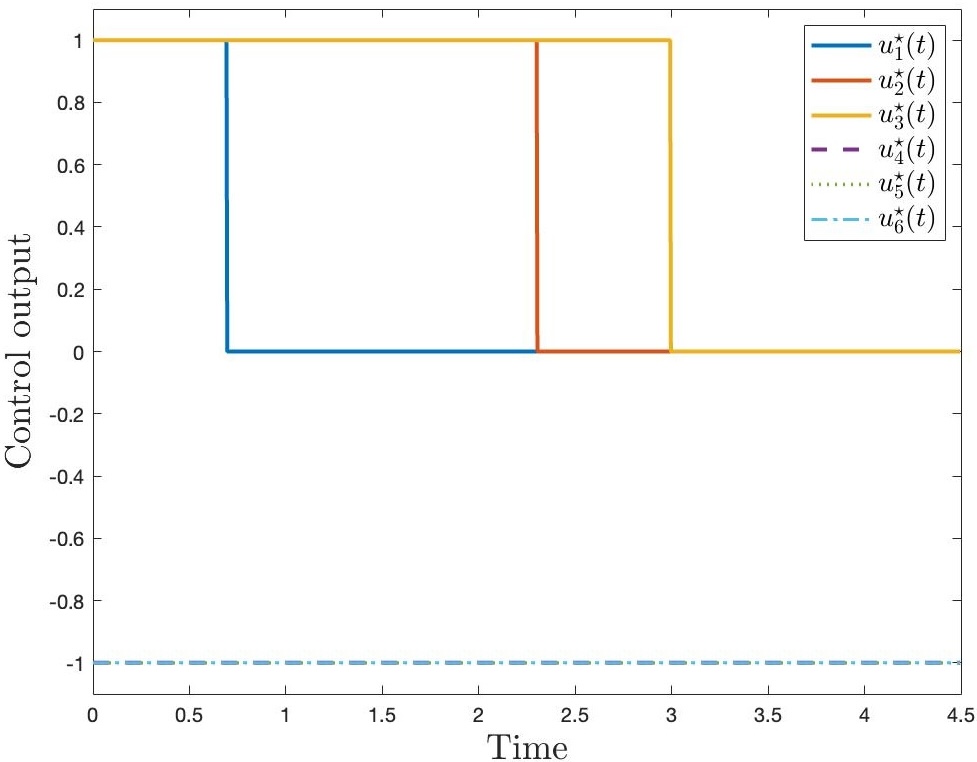}
    \end{subfigure}
    \hspace{4mm}
    \begin{subfigure}{.48\textwidth}
        \centering
        \includegraphics[width=\linewidth]{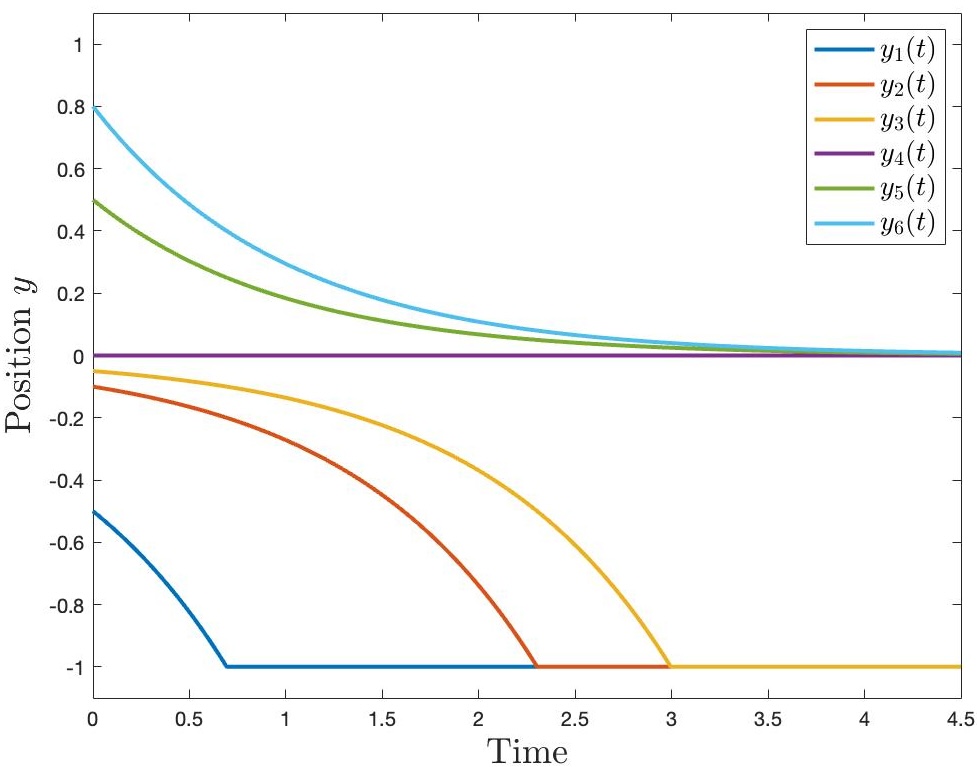}
    \end{subfigure}
    \caption{Left: Optimal control policies $u^\star(t)$ from~\eqref{eq:non_smooth_ocp_u} for initial values $y_0 \in \{-0.5,-0.1, -0.05, 0, 0.5, 0.8\}$. Right: corresponding trajectories $y(t)$.}
    \label{fig:non_smooth_ocp_uy}
\end{figure}
The optimal value function $V^\star$, illustrated in Figure~\ref{fig:non_smooth_ocp_v}, is given by
\begin{equation*}
    V^\star(y_0) = \begin{cases}
        \frac{1}{3} y_0, & y_0 \geq 0\\
        y_0 + \frac{1}{2} y_0^2, & y_0 <0.
    \end{cases}
\end{equation*}
The function $V^\star$ is not differentiable in $y_0 = 0$; hence, Condition~\ref{cond:ocp3} is not satisfied. Consequently, our procedure and the arguments from Section~\ref{sec:ocp} do not apply.

\begin{figure}
    \centering
        \includegraphics[width=0.7\linewidth]{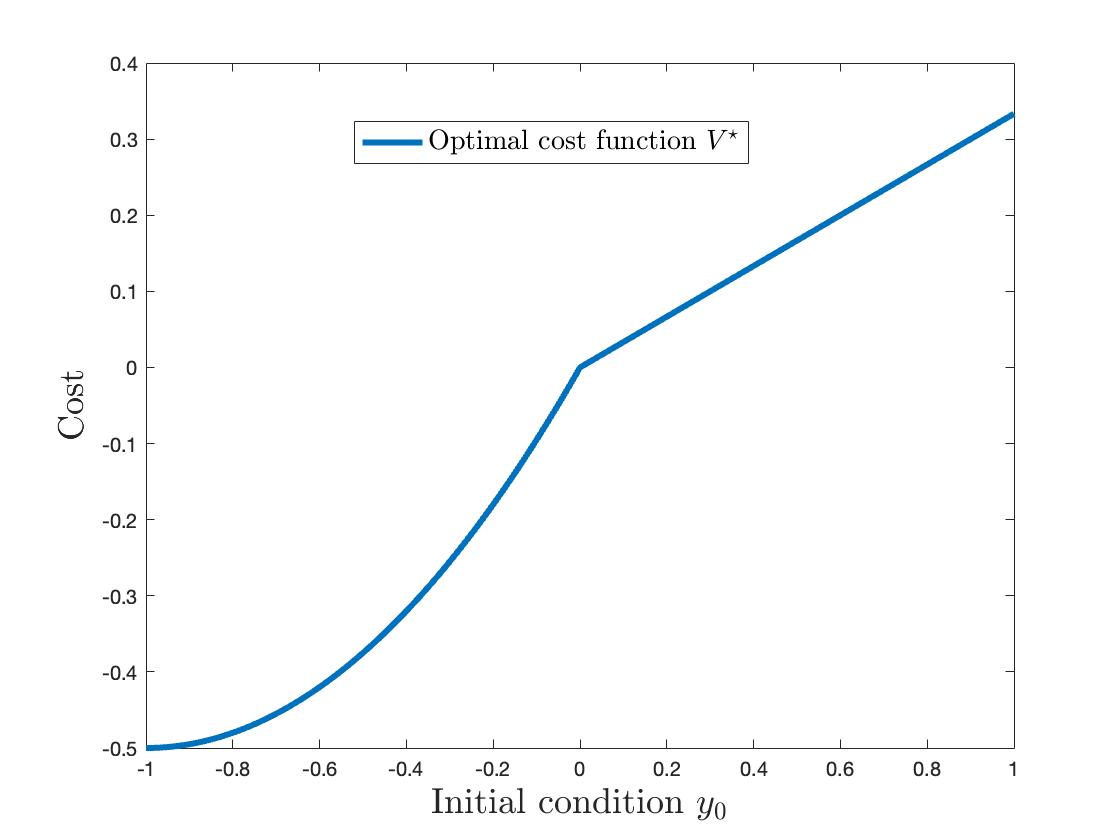}
        \caption{Optimal value function $V^\star$ for~\eqref{eq:OptContrNonSmooth}.}
        \label{fig:non_smooth_ocp_v}
\end{figure}

\begin{rem}
    In the existing examples, we observe that the existence and regularity of an optimal point for the functional LP is more of a limiting factor than the existence of an inward-pointing direction. The pragmatic argument is that as long as Positivstellensätze, such as Putinar's or Schmüdgen's Positivstellensatz \cite{schmudgen2017moment} are properly used, strict feasibility for the LP of at least one point is necessary. However, by Proposition \ref{lem:Inpointing}, this implies that the inward-pointing condition is satisfied.
\end{rem}

\begin{rem}\label{rem:LimitationsPrimalProb}
    Concerning the existence of optimal points, the primal problem on measures typically enjoys better properties than the functional LP. The reason is that for compact sets $\bX$, the space $\cM(\bX)$ is the dual space of $\cC(\bX)$ -- and thus, by the Banach-Alaoglu Theorem, bounded and closed sets in $\cM(\bX)$ are compact with respect to the weak$^\star$ topology. This property is leveraged for the existence of minimizers and plays a central role in the convergence of the moment-SoS hierarchy \cite{tacchi2022convergence}. Furthermore, an effective Positivstellensatz, such as {\bf Theorem \ref{thm:OldBaldi}}, transfers to a dual formulation on measures, see \cite[Theorem 5.7]{baldi2021moment}. While compactness arises rather naturally for the GMP, inducing compactness in the functional LP is more subtle. In {\bf Remark~\ref{rem:NoMinimizer}} we discuss relaxing the functional LP to allow for minimizers.
\end{rem}

\paragraph{Improved convergence rates.} We formulated our procedure so that any preferred effective Positivstellensatz can be applied. Thus, improved convergence rates in effective Positivstellensätze directly transfer to the convergence rates for the moment-SoS hierarchy for functional LPs that fall into our framework. For this reason, we want to emphasize which type of improvement in effective Positivstellensätze proves to be most powerful for our approach. In this direction, we see three main factors:
\begin{enumerate}
    \item Effective Positivstellensätze with mild behavior w.r.t the degree of the concerned polynomial.\label{item1}
    \item Effective Positivstellensätze without the need to embed $\bX$ into an ambient set $\bK$.\label{item2}
    \item Positivstellensätze tailored towards certain semi-algebraic sets, such as the hypercube, the unit ball, the sphere, or the simplex, with improved convergence rates.\label{item3}
\end{enumerate}

The first point refers to the fact that in our framework -- to obtain convergence -- we allow the degree of the polynomial of interest to go to infinity. We discuss this point in the following {\bf Remark~\ref{rem:DegreeBehaviour}}

\begin{rem}\label{rem:DegreeBehaviour}
    The bound on hierarchy level $\ell\in \N$ in {\bf Theorem \ref{thm:OldBaldi}} depends (strongly) on the degree $\deg (p)$ of the polynomial $p$. A milder growth with respect to $\deg (p)$ immediately translates to stronger convergence rates in our framework. Consequently, effective Positivstellensätze such as \cite{fang2021quantum,baldi2022moment} are disadvantageous for our approach compared {\bf Theorem~\ref{thm:OldBaldi}}, despite having certain very desirable properties for different application. 
\end{rem}

Point~\eqref{item2} in the above list, concerns the fact that, for applying {\bf Theorem~\ref{thm:OldBaldi}}, we need to bound the concerned polynomial on the ambient set $[-1,1]^n \supset \bX$ rather than only on $\bX$, which we discuss in the next remark.

\begin{rem}\label{rem:AmbientSet}
    The constant $\max\limits_{\vx \in [-1,1]^n} p(\vx)$ appears in the effective Positivstellensatz {\bf Theorem~\ref{thm:OldBaldi}} even though we are interested in $p$ on the set $\bX \subsetneq [-1,1]^n$. Effective Positivstellensätze that only consider the values of $p$ on $\bX$ instead would prove very useful for our approach.
    This is indicated by {\bf Lemma \ref{lem:MatteoBound}} and can be observed in the proofs of {\bf Theorems \ref{thm:VolSmoothConRate}, \ref{prop:PropOCPepsilon}} and \textbf{\ref{thm:SDEExitLocationConvRate}} where we obtained an upper bound of the function of interest on $[-1,1]^n$ by first extending it from $\bX$ to a function on $[-1,1]^n$. In~\cite{schlosser2024specialized}, we applied specialized Positivstellensätze for the unit ball and the sphere with no need for an ambient set and showed super-polynomial convergence rate {for the} exit location problem from {Section~\ref{sec:ExitLocation}}. For the Stokes-constraint volume computation problem in {Section~\ref{sec:Stokes}} such improved Positivstellensätze would highlight a much sharper difference between {\bf Theorems \ref{thm: volconvrate}} and \textbf{\ref{thm:VolSmoothConRate}}, because contrary to {\bf Theorem \ref{thm:VolSmoothConRate}}, {\bf Theorem \ref{thm: volconvrate}} would not benefit from such improvements.
\end{rem}

As we have indicated in the above remark, another possibility to get improved rates is to use effective Positivstellensätze tailored to specific sets and provide stronger bounds than general Positivstellensätze.

\begin{rem}[On specialized Positivstellensätze] \label{rem:spec_psatz} \leavevmode

    There exist specialized (and probably tighter) versions of effective Putinar Positivstellensätze on a variety of sets, such as the unit ball~\cite{slot2021christoffel}, the unit sphere~\cite{fang2021quantum,schlosser2024specialized} and, more recently, the hypercube~\cite{baldi2023psatz,baldi2021moment}. They enjoy refined analysis and therefore provide -- in some sense -- sharper convergence rates. Therefore, if available, they may be preferable. However, some of these effective Positivstellensätze do not come with explicit bounds depending on $\deg(p)$ in addition to $\max_\bX p$ and $\min_\bX p$. Some others scale unfavorably in the degree $\deg (p)$ \cite{baldi2022moment,fang2021quantum} as a trade-off for other properties that are less important for our application. In~\cite{slot2021christoffel,schlosser2024specialized} specialized Positivstellensätze on the unit ball, respectively the sphere, are presented with relatively well-behaved growth in $\deg (p)$.
\end{rem}
    
\section{Conclusion}\label{sec:Conclusion}

We state a structured approach to obtaining convergence rates for the moment-SoS hierarchy for the generalized moment problem. For the analysis of the convergence rates, we distinguish three important objects and properties. Namely, the existence and regularity of optimal points, an effective version of Putinar's Positivstellensatz, and a geometric feasibility condition (see the inward-pointing condition in Section \ref{sec:InwardPointingCond}). Our proposed procedure points out how those properties interact, and it is demonstrated to obtain upper bounds on the convergence rate for certain instances of the moment-SoS hierarchy: Using recent improvements on an effective version of Putinar's Positivstellensatz, we build upon and strongly improve existing convergence rates for the optimal control and the volume computation of a semialgebraic set; and we give an original convergence rate for a moment-SoS hierarchy of exit location computation for stochastic differential equations. We hope our work provides a guideline and the necessary tools for computing convergence rates of the moment-SoS hierarchy for various generalized moment problems that are actively formulated in the field in recent and following years.

Future work and improvement of effective Positivstellensätze can be integrated within our work simply by applying the most suited available convergence rate for Putinar's Positivstellensatz. Furthermore, we observe in our analysis that a well-suited effective Positivstellensatz could strongly further improve the convergence rate. As mentioned in {\bf Remark \ref{rem:AmbientSet}}, particularly advantageous for our method would be an effective Positivstellensatz -- for a polynomial $p$ on a semialgebraic set $\bX$ -- that only takes into account the values of $p$ on $\bX$ without the need of bounding its value on an ambient set (such as the hypercube in {\bf Theorem \ref{thm:OldBaldi}}). Similarly, specialized Positivstellensätze could be improved by expliciting all the terms in their degree bounds. Considering the recent improvement and active work on degree bounds for Positivstellensätze, we see here a very interesting, exciting, and promising development for further improvements of existing convergence rates for the moment-SoS hierarchy for generalized moment problems, as well as quantitative analysis of many other moment-SoS-based methodologies that will appear in the future.

Another future direction could include an analysis of the primal problem on measures. The primal problem has the advantage that, due to better compactness properties in the space of measures, optimal points often exist for the primal problem on measures -- even when the dual problem on functions does not have one. In \textbf{Section \ref{sec:Limitations}}, we discuss that non-existence of optimal points is a limiting factor for our approach. See \textbf{Remark \ref{rem:LimitationsPrimalProb}} and \textbf{Remark \ref{rem:NoMinimizer}}. Along a similar line of reasoning is inferring the desired properties from a minimizing sequence instead of an optimal point. In contrast to optimal points, minimizing sequences always exist. Due to its increased generality, this approach comes with the need for a finer and more tailored analysis of the problem at hand.

We think it is important to mention that the asymptotic analysis of the moment-SoS hierarchy for generalized moment problems might not transfer to practical applications. The reason is twofold. Firstly, current computational capacities restrict the computation of the moment-SoS hierarchy already for medium-sized problems to low-degree instances. Secondly, the conditioning of the $\ell$-th level of the moment-SoS hierarchy gets worse with increasing $\ell \in \N$, hampering the convergence in practice. In other words, this work essentially addressed the recasting from the infinite-dimensional GMP into SoS programming problems, while future works will shift the focus onto the translation from SoS programming to actual SDP, involving deeper investigations on what polynomial basis to choose in that process (the usual one being the numerically ill-behaved basis of monomials).

\appendix

\section{Duality in the moment-SoS hierarchy}

The three steps displayed in section~\ref{subsec:hierarchy} can be complemented as follows:

{\bf Step 4:} Using Lagrange duality, from the SoS tightening \eqref{eq:sosd} we deduce the following moment \textit{relaxation}:

\begin{equation} \label{eq:momd}
\fcolorbox{red}{white}{$
\begin{array}{lcl}
\op_{\ell} := & \sup\limits_{Z \in \cQ_\ell(\vh)'} & \langle Z , g \rangle \\
& \st & \cA_\ell' \, Z = T\left|_{\R_{d_\ell}[\vy]}\right.
\end{array}
$}
\end{equation}

\noindent where the ``projected adjoint'' $\cA_\ell' := (\cA|_{\R_{d_\ell}[\vx]})' : \R_{2\ell}[\vx]' \rightarrow \R_{d_\ell}[\vy]'$ coincides with $\cA'$ on $\R_{2\ell}[\vx]'$ via
$$ \forall \bbeta \in \N^n_{d_\ell}, \qquad \langle \vy^{\bbeta}, \cA'_\ell\, Z \rangle = \langle \vy^{\bbeta}, \cA'\, Z \rangle = \langle \varphi_{\bbeta}, Z \rangle $$
under {\bf Assumption \ref{asm:momop}} and with $\deg(\varphi_{\bbeta}) \leq 2 \ell$ by {\bf Corollary \ref{cor:momop}} and its proof.

Several results, ranging from practical to theoretical, come with the moment-SoS tightenings and relaxations. First, \cite[Proposition 2.1]{lasserre2009moments} gives a representation of $\cQ_\ell(\vh)$ with {positive semidefinite matrices} (and hence by duality $\cQ_\ell(\vh)'$ is represented by linear matrix inequalities), so that the moment relaxations and SoS strengthenings are equivalent to semidefinite programming (SDP) problems.

Second, a strong duality theorem can be stated for the moment-SoS hierarchy:

\begin{thm}[Strong duality in the hierarchy; extension of {\cite[Proposition 6]{tacchi2022convergence}}] \label{thm:msos-dual} \leavevmode

    Assume that one of the following conditions holds:
    \begin{enumerate}
        \item $\exists \ell \in \N, B>0$ s.t. $\forall Z \in \cQ_\ell(\vh)'$ feasible for the moment relaxation in~\eqref{eq:momd}, one has $\langle Z,1 \rangle \leq B$. 
        \item $\exists \hat{w} \in \R[\vy]$ with $\cA \, \hat{w} > \delta > 0$ on $\bX$.
        \item $g \geq 0$ on $\bX$ and {\bf Condition~\ref{condition:slater}} holds.
    \end{enumerate}
    Then, strong duality holds in the hierarchy above some degree $\ell_{\min}$:
    $$ \forall \ell \geq \ell_{\min}, \qquad \op_\ell = \gmd_\ell $$
\end{thm}
\begin{proof}
    The fact that condition \textit{(1)} implies strong duality at degree $\ell$ and above is \cite[Proposition 6]{tacchi2022convergence}. Here we prove that conditions \textit{(2)} and \textit{(3)} imply condition \textit{(1)}. First, we have already proved that \textit{(3)} $\Longrightarrow$ \textit{(2)} at the end of section~\ref{sec:gmp}, so it remains to prove \textit{(2)} $\Longrightarrow$ \textit{(1)}, which can be done as in {\bf Lemma~\ref{lem:StricPosImageA}}, with some additional arguments. Let $\hat{w}$ be as in condition \textit{(2)} and let $\ell \in \N$ such that $\cA \, \hat{w} - \delta \in \cQ_{\ell}(\vh)$ (such $\ell$ exists by {\bf Assumption~\ref{asm:momop}} and {\bf Theorem~\ref{thm:psatz}}). Then, it holds 
    $$\deg(\cA \, \hat{w}) = \deg(\cA \, \hat{w} - \delta + \delta) \leq \max\left(\deg(\cA \, \hat{w} - \delta), \deg(\delta)\right) \leq 2\ell$$ so that $\langle Z, \cA \,\hat{w} \rangle $ is well defined for $Z \in \cQ_\ell(\vh)'$. Now, let $Z \in \cQ_\ell(\vh)'$ be feasible in~\eqref{eq:momd}. By {\bf Corollary~\ref{cor:momop}}, $\cA \, \hat{w}$ has a preimage $w_\ell \in \R_{d_\ell}[\vx]$ such that $\cA \, \hat{w} = \cA \, w_\ell$, and thus it holds 
    $$\langle Z,  \cA\,\hat{w}\rangle = \langle Z,  \cA\,w_\ell\rangle = \langle \cA_\ell' \, Z, w_\ell \rangle = \langle T, w_\ell \rangle$$ 
    ($\deg(w_\ell) \leq d_\ell$ is needed for the adjoint operation to be well defined, as $\cA'_\ell \, Z \in \R_{d_\ell}[\vy]')$. Eventually, defining $B := \delta^{-1}\langle T, w_\ell\rangle$, it holds
    \begin{align*}
        0 & \leq \delta \langle Z, 1 \rangle \\
        & = \langle Z, \cA \, \hat{w} \rangle - \langle Z, \cA \, \hat{w} \rangle + \delta \langle Z, 1 \rangle \\
        & = \langle T , w_\ell \rangle - \langle Z, \cA \, \hat{w} - \delta \rangle \\
        & \leq \langle T, w_\ell \rangle = \delta \, B,
    \end{align*}
where the first inequality is due to $Z \in \cQ_\ell(\vh)'$ and $1 \in \cQ_\ell(\vh)$, and the second inequality is due to $Z \in \cQ_\ell(\vh)'$ and $\cA \, \hat{w} - \delta \in \cQ_{\ell}(\vh)$. As $\delta > 0$, this shows that \textit{(1)} holds and concludes the proof.
\end{proof}
\begin{rem}
    Condition \textit{(1)} in~{\bf Theorem~\ref{thm:msos-dual}} is a rephrazing of {\bf Condition~\ref{condition:tacchi}} with $Z$ instead of $\mu$.
\end{rem}

Eventually, the following theorem ensures strong convergence guarantees of the corresponding numerical scheme:

\begin{thm}[Convergence of the moment-SoS hierarchy {\cite[Theorem 4 \& Corollary 8]{tacchi2022convergence}}] \label{thm:lascv} \leavevmode

Assume that one of the conditions in {\bf Theorem~\ref{thm:msos-dual}} holds. Then, under {\bf Assumption~\ref{asm:poprateOld}} there exists $\ell_{\min} \in \N \ \st \ \forall \ell \geq \ell_{\min}$ $$\gmd_{\ell} = \op_{\ell} \underset{\ell\to\infty}{\longrightarrow} \op = \gmd.$$

Moreover, if \eqref{eq:smp} has a unique solution $\mu^\star$, then for optimal $Z_\ell$ in~\eqref{eq:momd} it holds
$$\forall \balpha \in \N^m, \qquad \langle \vx^{\balpha} , Z_\ell \rangle \underset{\ell\to\infty}{\longrightarrow} \int \vx^{\balpha}\; \od\mu^\star(\vx).$$
\end{thm}

\section{The functional LP and the GMP in product spaces}\label{appendix:ProductSpaces}

Let $N,M \in \N^\star$, $\vn = (n_1,\ldots,n_N) \in \N^N$, $\vm = (m_1,\ldots,m_M) \in \N^M$. For $i\in \lint N\rint$ and $j \in \lint M \rint$, let $\bX_i \Subset \R^{n_i}$ and $\bY_j \Subset \R^{m_j}$ be compact. We set

\vspace*{1em}

\begin{center}
\begin{tabular}{ll}
$\cX := \cC(\bX_1)\x\ldots\x\cC(\bX_N)$,  & $\cY := \cP(\bY_1)\x\ldots\x\cP(\bY_M)$.\\
$\cX' := \cM(\bX_1)\x\ldots\x\cM(\bX_N)$,  & $\cY' \:= \cM(\bY_1)\x\ldots\x\cM(\bY_M)$,  \\\\
\end{tabular}
\end{center}

\noindent We equip $\bar{\cX},\bar{\cX}',\cY,\cY'$ with the product topology. For $\vv = (v_1,\ldots,v_N) \in \cX$ and $\bmu = (\mu_1,\ldots,\mu_N) \in \cX'$, we define the vector integral as
$$ \int \vv \cdot \od\bmu := \sum_{i=1}^N \int v_i \; \od\mu_i. $$

We recall the general LPs from~\eqref{eq:LPGeneral}

\begin{equation}
\fcolorbox{red}{white}{$
\begin{array}{lclclcl}
\gmd = & \inf\limits_{\vw \in \cY} & \langle \bT ,  \vw \rangle & \quad \text{ and } \quad \quad & \op = & \sup\limits_{\mu \in \cX'} & \displaystyle\int g \; \od \mu\\
& \st & \forall i \in \lint N \rint & & & \st & \forall i \in \lint N \rint, \quad \mu_i \in \cM(\bX_i)_+ \vphantom{\sup\limits_{\mu \in \cM(\bX)} \displaystyle\int}\\
& & (\cA \, \vw - g_i)_i \in \cC(\bX_i)_+ & & & & \cA' \, \mu = \bT
\end{array}$}\tag{\text{\ref*{eq:LPGeneral}}}
\end{equation}

Concerning the functional LP and the GMP we made {\bf Assumptions~\ref{asm:existenceSol},~\ref{asm:momop},~\ref{asm:poprateOld}} and {\bf Conditions~\ref{condition:slater},~\ref{condition:tacchi}}. We now want to extend those to the setting of this section. 

{\bf Assumptions~\ref{asm:existenceSol}} states that a feasible solution to the functional LP exists. Therefore no adjustment has to be made.

Addressing {\bf Assumption~\ref{asm:momop}}, we first need to extend the notion of polynomial operators to product spaces.

\begin{rem}[Polynomial operators and moment operators on product spaces]
    The generalization of polynomial and moment operators to linear operators $\cA: \cP(\bY_1)\times \ldots \times \cP(\bY_M) \rightarrow \cC(\bX_1)\times \ldots \cC(\bX_N)$ is straightforward. {Such an operator $\cA$ is a polynomial operator if $\cA\left(\cP(\bY_1)\x\ldots\x\cP(\bY_M)\right) \subset \cP(\bX_1)\x\ldots\x\cP(\bX_N)$. Moment operators on product spaces are adjoints of polynomial operators}. The results from {Section~\ref{sec:lasserre}} on polynomial operators and moment operators generalize as well.
\end{rem}

For the general case, {\bf Assumption~\ref{asm:momop}} reads

\begin{manualasm}{\ref*{asm:momop}'}[Polynomial operator and $\vg$ polynomial] \label{asm:momopGen} \leavevmode
The operator $\cA$ is a polynomial operator and $\vg \in \cP(\bX_1)\times \ldots \cP(\bX_N)$.
\end{manualasm}

For product spaces the {\bf Assumption~\ref{asm:poprateOld}} generalizes simply by applying it componentwise.

\begin{manualasm}{\ref*{asm:poprateOld}'}\label{asm:poprateOldGen} \leavevmode For $i \in \lint N \rint$, the sets $\bX_i$ are compact basic semialgebraic sets given by $\bX_i := \bS(\vh_i) \subset \R^{n_i}$ with $\vh_i = (h_{i1},\ldots,h_{ir_i}) \in \R[\vx_i]^{r_i}$ for some $\vr = (r_1,\ldots,r_N) \in (\N^\star)^N$, and for which it holds
\begin{enumerate}
    \item $1 - \vx^\top \vy \in \cQ(\vh)$ \emph{(normalized Archimedean property)},
    \item $\forall j \in \lint r_i \rint, \quad \|h_{ij}\| := \max\limits_{\vx \in [-1,1]^{n_i}} |h_{ij}(\vx)| \leq \frac{1}{2}$
\end{enumerate}
\end{manualasm}

Similarly, {\bf Conditions~\ref{condition:slater},~\ref{condition:tacchi}} apply componentwise. They read
\begin{manualcond}{\ref*{condition:slater}'}[Slater {\cite{slater1950lagrange}}] \label{condition:slaterGen}
    $\exists \overset{\circ}{\vw} \in \cY \ \st \ \forall i \in \lint N \rint, \ (\cA' \,\overset{\circ}{\vw})_i - g_i \in \cC(\bX_i)_\oplus$.
\end{manualcond}
\begin{manualcond}{\ref*{condition:tacchi}'}[Primal compactness {\cite{tacchi2021thesis}}] \label{condition:tacchiGen}
    $\exists B > 0 \ \st \ \forall \bmu \in \cX'$ feasible for the generalized moment problem in~\eqref{eq:LPGeneral}, one has $\forall i \in \lint N \rint, \ \displaystyle \int 1 \; \od\mu_i \leq B $.
\end{manualcond}

Similary to {\bf Lemma~\ref{lem:StricPosImageA}} the {\bf Conditions~\ref{condition:slaterGen},~\ref{condition:tacchiGen}} hold if there exists $\hat{\vw} \in \cY$ with $(\cA\ \hat{\vw})_i > 0$ on $\bX_i$ for all $\forall i \in \lint N \rint$.

\section{On norm equivalence in polynomial spaces}\label{appendix:MattBound}

\begin{lem}

    Let $\bX \subset \bK := [-1,1]^n$ satisfy {\bf Condition \ref{cond:ocp}.2}. For any nonnegative polynomial $p \in \cP(\bX)_+$, it holds
    \begin{equation} \tag{\ref*{eq:BOUND}}
        \|p\| = \max\limits_{\vx \in \bK}\{p(\vx)\} \leq \left(1 + \frac{\deg(p)^2}{4}\left(\nicefrac{2}{b}\right)^{\deg(p)+1}\right) \|p\|_\infty^\bX,
    \end{equation}
    where $b \in (0,1)$ is such that $[-b,b]^n \subset \bX$ (whose existence is guaranteed by {\bf Condition \ref{cond:ocp}.2}).
\end{lem}
    \begin{proof}
    The proof of \cite[Lemma 28]{baldi2023psatz} actually shows that, for $\varphi \in \cP(\bK)_+$ of degree $k$ and $\rho > 0$, defining $\bK_\rho := [-1-\rho,1+\rho]^n$, one has
    \begin{equation} \label{eq:bullshit} \tag{$\dagger$}
        \varphi^\star_{\bK_\rho} \geq \varphi^\star_{\bK} - T_k(1+\rho)\cdot\rho\cdot k^2 \cdot \max_\bK \varphi
    \end{equation}
    Where $T_k$ denotes the degree $k$ Chebyshov polynomial of the first kind. We apply this result to
    $$ \varphi(\vx) := \max_\bX p - p(b\,\vx), $$
    so that $k = \deg(\varphi) = \deg(p)$. Let $\rho := \frac{1-b}{b}$ so that $b\cdot\bK_{\rho} = \bK$. Then, one has
    \begin{align*}
        \varphi^\star_{\bK_\rho} & = \max_\bX p - \max_\bK p \\
        \varphi^\star_\bK & = \max_\bX p - \max_{b\cdot\bK} p \\
        \max_\bK \varphi & = \max_\bX p - \min_{b\cdot\bK} p
    \end{align*}
    which can be reinjected in \eqref{eq:bullshit} to get
    $$\cancel{\max_\bX p} - \max_\bK p \geq \cancel{\max_\bX p} - \max_{b\cdot\bK} p - T_k(\underset{\nicefrac{1}{b}}{\underbrace{1+\rho}})\cdot\rho\cdot k^2 \left(\max_\bX p - \min_{b\cdot\bK} p\right).$$
    This expression in turn rephrases, accounting for inequalities $\min_{b\cdot\bK} p \geq \min_\bX p \geq 0$ and $\max_{b\cdot\bK} p \leq \max_\bX p = \|p\|_\infty^\bX$ (because $b\cdot\bK \subset \bX$), as
    $$ \max_\bK p \leq \max_{b\cdot\bK} p + T_k(\nicefrac{1}{b})\cdot\frac{1-b}{b}\cdot k^2 \cdot \max_\bX p \leq \left(1 + T_k(\nicefrac{1}{b})\cdot\frac{1-b}{b}\cdot k^2\right) \|p\|_\infty^\bX. $$
    
    It remains to compute an upper bound of $T_k(\nicefrac{1}{b})$. Here we recall that the Chebyshov polynomials are defined by $T_0(s)=1$, $T_1(s)=s$ and the recurrence formula $T_{k+1}(s) = 2s\,T_{k}(s) - T_{k-1}(s)$ (for $k \geq 1$). Moreover, $T_k(s)$ oscillates between $-1$ and $1$ for $s \in [-1,1]$ (because $T_k(\cos\theta) = \cos(k\cdot\theta)$), and is always strictly increasing on $(1,+\infty)$, so that for $s>1$ one has $T_k(s) > T_k(1) = 1 > 0$. Eventually, we can prove that, for all $k \geq 1$, it holds
    \begin{equation} \label{cheby} \tag{$\ddagger$}
        T_k(\nicefrac{1}{b}) \leq \frac{1}{2} \left(\frac{2}{b}\right)^k.
    \end{equation}
    First, we check that this holds for $k=1$: $T_1(\nicefrac{1}{b}) = \nicefrac{1}{b} \leq \nicefrac{1}{2} \cdot \nicefrac{2}{b}$. Second, defining $s = \nicefrac{1}{b}$, we notice that by design of $b \in (0,1)$, $s > 1$. Then we simply use the recurrence formula to get, for $k \geq 1$:
    $$T_{k+1}(s) = 2s \, T_{k}(s) - \underset{>0}{\underbrace{T_{k-1}(s)}} \leq 2s \, T_{k}(s)$$
    so that, if $T_k(s) \leq \nicefrac{1}{2} (2s)^k$ (which holds for $k=1$) then $T_{k+1}(s) \leq \nicefrac{1}{2} (2s)^{k+1}$. This way, we get
    $$ \max_\bK p \leq \left(1 + \frac{1}{2}\left(\frac{2}{b}\right)^k\cdot\frac{1-b}{b}\cdot k^2\right) \|p\|_\infty^\bX. $$
    Finally, let us not forget that it is $\max_\bK |p|$ that we want to upper bound, and not only $\max_\bK p$, so we still have to upper bound $\max_\bK (-p) = -\min_\bK p = -p^\star_\bK$. Again we use \eqref{eq:bullshit} with $\varphi(\vx) = p(b\vx)$ to get
    \begin{align*}
        p^\star_\bK & = \varphi^\star_{\bK_\rho} \\
        & \geq \varphi^\star_\bK - T_k(1+\rho)\cdot\rho\cdot k^2 \cdot \max_\bK \varphi \\
        & \geq \min_{b\cdot\bK} p - \frac{1}{2}\left(\frac{2}{b}\right)^k\cdot\frac{1-b}{b}\cdot k^2 \cdot \max_{b\cdot\bK}p \\
        & \geq \min_{\bX} p - \frac{1}{2}\left(\frac{2}{b}\right)^k\cdot\frac{1-b}{b}\cdot k^2 \cdot \max_{\bX}p \\
        & \geq - \frac{1}{2}\left(\frac{2}{b}\right)^k\cdot\frac{1-b}{b}\cdot k^2 \cdot \|p\|_\infty^\bX
    \end{align*}
    and hence
    $$ \max_\bK(-p) = -p^\star_\bK \leq \frac{1}{2}\left(\frac{2}{b}\right)^k\cdot\frac{1-b}{b}\cdot k^2 \cdot \|p\|_\infty^\bX $$
    leading to 
    $$ \|p\| \leq \left(1 + \frac{1}{2}\left(\frac{2}{b}\right)^{\deg(p)}\cdot\frac{1-b}{b}\cdot \deg(p)^2\right) \|p\|_\infty^\bX. $$
    Finally we deduce the announced inequality by observing that $(1-b) \in (0,1)$, so that
    $$ \frac{1}{2}\left(\frac{2}{b}\right)^{\deg(p)}\cdot\frac{1-b}{b} \leq \frac{1}{4}\left(\frac{2}{b}\right)^{\deg(p)+1} $$
    \end{proof}

\section{Extension of Hölder continuous functions}\label{appendix:Hölder}

\begin{lem}[Extension Lemma; {\cite[Lemma 6.37]{gilbarg1977elliptic}}] \leavevmode

 Let $k \geq 1$ be an integer and $a \in (0,1]$. Let $\bY \subset \R^m$ be a compact with $C^{k,a}$ boundary. Let $\bOm$ be an open and bounded set containing $\bY$. Then every function $w \in C^{k,a}(\bY)$ there exists an extension $\Bar{w} \in C^{k,a}(\bOm)$ with $w(\vy) = \Bar{w}(\vy)$ for all $\vy \in \bY$ and
\begin{equation} \tag{\ref*{eq:extlemma}}
    \|\Bar{w}\|_{C^{k,a}(\bOm)} \leq c \|w\|_{C^{k,a}(\bY)}
\end{equation}
for some constant $c = c(m,k,a,\bY,\bOm)$ independent of $w$.  
\end{lem}

As a corollary, we obtain the following extension result that aims at preserving the maximum value for the extension.

\begin{cor}
    Let $\bY \subset \R^m$ be a compact with $C^{\infty}$ boundary and $f \in C^{\infty}(\bY)$. Then, for any $k\in \N$ and $\varepsilon >0$, there exists an extension $\Bar{f} \in C^{k}(\R^m)$ of $f$ with
\begin{equation}\label{eq:appendixExtensionBound} \tag{$\diamond$}
    \|\Bar{f}\|^{\R^m}_\infty \leq \|f\|_{\infty}^\bY + \varepsilon.
\end{equation}
\end{cor}

\begin{proof}
    Let $\varepsilon > 0$ and $k\in \N$ and let $R> 0$ with $B_{0.5 R}(0) \supset \bY$. By the extension Theorem \ref{thm:HölderExtension} there exists an extension $\Tilde{f}\in C^k(B_R(0))$ of $f$ with $F:= \|\Tilde{f}\|_{C^1(B_R(0))} < \infty$. Without loss of generality we assume $\frac{\varepsilon}{F} < \frac{R}{2}$, otherwise we take a smaller $\varepsilon$. Set $\bOm:= \{\vx + \vy ; \vx \in \bY, \|\vy\| < \frac{\varepsilon}{F}\}$ and let $\phi\in C^\infty(\R^m)$ with $0\leq \phi\leq 1$ with $\phi = 1$ on $\bY$ and $\phi = 0$ on $\R^m \setminus \bOm$. We claim that the function $\Bar{f} := \phi \cdot \Tilde{f}$ (and extended by zero on $\R^m \setminus \bOm$) is $C^k$ and satisfies (\ref{eq:appendixExtensionBound}). From the construction, it follows that $\Tilde{f}$ is $C^k$. For $\vx \in \bY$ it holds $\Tilde{f}(\vx) = f(\vx)$ and for $\vx \in \R^m \setminus \bOm$ we have $\Tilde{f}(\vx) = 0$. It remains to bound $\Tilde{f}(\vx)$ for $\vx \in \bOm \setminus \bY$. Let $\vx \in \bOm \setminus \bY$ and $\vy$ with $\|y\| < \frac{\varepsilon}{F}$ such that $\vx - \vy \in \bY$. We have
    \begin{equation*}
        |\Tilde{f}(\vx)| \leq \left|\Tilde{f}(\vx - \vy) + \|\grad \Tilde{f}\|_\infty^{B_R(0)} \|y\|\right| \leq \|f\|_{\infty}^\bY + F \frac{\varepsilon}{F} = \|f\|_{\infty}^\bY + \varepsilon.
    \end{equation*}
\end{proof}

\section{Smooth solutions to the Poisson PDE} \label{appendix:Stokes}

Here we state the proof of \textbf{Theorem \ref{thm:smooth}} for the existence of smooth solutions of (\ref{eq:stokesd}) and \eqref{eq:poisson}.

\begin{proof} \textit{Of \textbf{Theorem \ref{thm:smooth}}.}
We prove this Theorem using two lemmas.
\begin{lem}[Existence of smooth source term] \label{lem:source} \leavevmode

There exists a smooth $\phi \in C^\infty(\R^n)$ such that
\begin{enumerate}
\item $\phi$ satisfies conditions {\normalfont (\ref*{eq:poisson}.c)}, {\normalfont (\ref*{eq:poisson}.d)} and $\phi = 1$ on $\R^n \setminus \bX$.
\item $\displaystyle\int_{\bX_i} \phi \; d\lambda = 0$ for all $i \in \{1,\ldots,\Omega\}$, where
$$ \bX = \bigsqcup_{i=1}^\Omega \bX_i$$
is the partition of $\bX$ into its connected components.
\end{enumerate}
\end{lem}
\begin{proof}
We work on a connected component $\bX_i$, $i \in \{1,\ldots,\Omega\}$. As $\bX_i$ is an open set, there exists $\bomega_i \in \bX_i$, $R_i > 0$ such that
$$ \bB_i := \{\vx \in \R^n \; ; |\vx - \bomega_i| \leq R_i\} \subset \bX_i. $$
According to \cite[\textbf{Proposition 2.25}]{Lee2013manifolds}, there exists a smooth bump function $\varphi_i \in C^\infty(\R^n)$ such that:
\begin{itemize}
\item $\forall \vx \in \R^n \setminus \bX_i$, $\varphi_i(\vx) = 0$
\item $\forall \vx \in \bB_i$, $\varphi_i(\vx) = 1$
\item $\forall \vx \in \bX_i \setminus \bB_i$, $0 \leq \varphi_i(\vx) \leq 1$.
\end{itemize}
In particular, $\varphi_i \geq 0$ in $\bX_i$ and $\varphi_i = 0$ on $\partial \bX_i$. Next, we define for $\vx \in \R^n$,
$$ \psi_i(\vx) := \frac{\lambda(\bX_i)}{\displaystyle\int_{\bX_i} \varphi_i \; d\lambda} \varphi_i(\vx),$$
with $0 < \lambda(\bB_i) < \displaystyle\int_{\bX_i} \varphi_i \; d\lambda$ by design of $\varphi_i$.

\vspace*{0.5em}
Again, $\psi_i \geq 0$ in $\bX_i$ and $\psi_i = 0$ on $\partial \bX_i$. Moreover, now
\begin{equation} \label{eq:int} \tag{$\ast$}
\int_{\bX_i} \psi_i \; d\lambda = \lambda(\bX_i). 
\end{equation}
Eventually, we construct, for $\vx \in \R^n$,
$$ \phi(\vx) := 1 - \sum_{i=1}^\Omega \psi_i,$$
so that condition \textit{1.} is trivially satisfied, and smoothness of $\phi$ follows from smoothness of the $\varphi_i$s. We conclude by checking condition \textit{2.}: for $i \in \{1,\ldots,\Omega\}$,
\begin{align*}
\int_{\bX_i} \phi \; d\lambda & = \int_{\bX_i} \left( 1 - \sum_{j=1}^\Omega \psi_j \right) d\lambda = \int_{\bX_i} 1 \; d\lambda - \sum_{j=1}^\Omega \int_{\bX_i} \psi_j \; d\lambda \\
& = \lambda(\bX_i) - \int_{\bX_i} \psi_i \; d\lambda - \sum_{\substack{j=1 \\ j \neq i}}^\Omega \cancel{\int_{\bX_i} \psi_j \; d\lambda} \stackrel{\eqref{eq:int}}{=} 0
\end{align*}
\end{proof}

\begin{lem}[Existence of smooth PDE solution] \label{lem:solution} \leavevmode

Let $\phi \in C^\infty(\overline{\bX})$ be given by {\normalfont \textbf{Lemma \ref{lem:source}}}. Then, there exists a solution $u \in C^\infty(\overline{\bX})$ to the Poisson PDE with Neumann boundary condition:
\begin{equation} \tag{\ref*{eq:poisson}}
    \left\lbrace\begin{array}{rcllr}
    -\Delta u &= & \phi & \text{in} & \bX \\
    \partial_\vn u & = & 0 & \text{on} & \partial\bX
\end{array}\right.
\end{equation}
\end{lem}
\begin{proof}
If $\Omega = 1$ (\textit{i.e.} $\bX$ is connected), then this is a classical result, see e.g. \cite{gilbarg1977elliptic}\footnote{Smoothness of $\partial \bX$ is necessary here. This is the reason for \textbf{Assumption \ref{cond:smooth}}.}. Else, we just solve the problem separately on each connected component and glue the resulting solutions $u_i$ together into
$$ u = \sum_{i=1}^\Omega u_i \, \ind_{\bX_i} \in C^\infty(\overline{\bX})$$
because by construction of $\bX$ (with smooth boundary) the $\overline{\bX}_i$ are disjoint.
\end{proof}
The $\phi$ and $u$ given by \textbf{Lemmas \ref{lem:source}} and \textbf{\ref{lem:solution}} are a valid solution to \eqref{eq:poisson}. By construction, they also have the required smoothness. It remains to show that we can choose $\phi,u$ such that the functions $\vu,w$ given by (\ref{eq:vu,wfromPDE}) are optimal for (\ref{eq:stokesd}). Let us take $\phi$ as in Lemma \ref{lem:source} and $u$ the corresponding solution of (\ref{eq:poisson}). By the above, the functions $\vu,w$ from (\ref{eq:vu,wfromPDE}) are smooth. Further, as shown in \cite{tacchi2023stokes}, $\vu,w$ are feasible (and optimal) for (\ref{eq:stokesd}). Here, we only recall optimality, i.e.
\begin{equation}\label{eq:intwVolumeO} \tag{$\triangledown$}
    \displaystyle\int w\; \od \lambda = \lambda(\bX).
\end{equation}
To show (\ref{eq:intwVolumeO}), we use condition 2. in Lemma \ref{lem:source} and simply integrate $w = 1-\phi \geq 0$ on $\bX$. This gives
\begin{eqnarray*}
    \displaystyle\int w\; \od \lambda_\bK  = \displaystyle\int 1-\phi\; \od \lambda_\bK = \lambda(\bX) - \displaystyle\int \phi\; \od \lambda_\bK = \lambda(\bX).
\end{eqnarray*}
\end{proof}

\bibliographystyle{abbrv}
\bibliography{references3}

\end{document}